\documentclass[10pt]{amsart}
\usepackage[margin=2.5cm]{geometry}
\usepackage{tikz-cd}
\usepackage{graphicx}					
\usepackage{amssymb}
\usepackage{amsmath}
\usepackage{subfig, graphicx}
\usepackage{mathrsfs}
\usepackage{amssymb, amsthm, indentfirst}
\usepackage{relsize}
\usepackage{hyperref}
\usepackage{stmaryrd}
\usepackage{lineno}

\numberwithin{equation}{section}
\usepackage{color}
\theoremstyle{plain}
\newtheorem{thm}{Theorem}[section]

\newtheorem{lemma}[thm]{Lemma}
\newtheorem{prop}[thm]{Proposition}
\newtheorem{corollary}[thm]{Corollary}
\theoremstyle{definition}

\newtheorem{rmk}[thm]{Remark}

\newtheorem*{hypothesis*}{Hypothesis}

\title{On Petersson norms of generic cusp forms and special values of adjoint $L$-functions for $\GSp_4$}
\author{SHIH-YU CHEN AND ATSUSHI ICHINO}

\address{Department of Mathematics, Kyoto University, Kitashirakawa Oiwake-cho, Sakyo-ku, Kyoto 606-8502, Japan}
\email{sychen.math@gmail.com}

\address{Department of Mathematics, Kyoto University, Kitashirakawa Oiwake-cho, Sakyo-ku, Kyoto 606-8502, Japan}
\email{ichino@math.kyoto-u.ac.jp}
	
\def\SL{{\rm{SL}}}
\def\GL{{\rm{GL}}}
\def\GSp{{\rm GSp}}
\def\PGSp{{\rm PGSp}}
\def\Sp{{\rm Sp}}
\def\PGL{{\rm PGL}}
\def\o{\frak{o}}
\def\c{\frak{c}}
\def\A{{\mathbb A}}
\def\C{{\mathbb C}}
\def\R{{\mathbb R}}
\def\Q{{\mathbb Q}}
\def\Z{{\mathbb Z}}

\def\<{\langle}
\def\>{\rangle}
\def\G{{G}}

\def\x{\times}

\def\bp{\begin{pmatrix}}
\def\ep{\end{pmatrix}}

\def\<{\langle}
\def\>{\rangle}

\def\Ad{\operatorname{Ad}}
\def\diag{\operatorname{diag}}
\def\GL{\operatorname{GL}}

\def\GSp{\operatorname{GSp}}
\def\Hom{\operatorname{Hom}}
\def\Ind{\operatorname{Ind}}

\def\O{\operatorname{O}}
\def\Re{\operatorname{Re}}
\def\pr{\operatorname{pr}}
\def\Res{\operatorname{Res}}
\def\SL{\operatorname{SL}}

\def\Sp{\operatorname{Sp}}

\def\Sym{\operatorname{Sym}}
\def\tr{\operatorname{tr}}
\def\U{\operatorname{U}}
\def\vol{\operatorname{vol}}

\def\sp{\mathfrak{sp}}

\def\AA{\mathbb{A}}

\def\GG{\mathbb{G}}

\def\RR{\mathbb{R}}
\def\ZZ{\mathbb{Z}}

\def\calE{\mathcal{E}}

\def\calF{\mathcal{F}}
\def\calI{\mathcal{I}}

\def\calR{\mathcal{R}}
\def\calS{\mathcal{S}}
\def\calZ{\mathcal{Z}}

\def\1{\mathbf{1}}

\def\frka{\mathfrak{a}}
\def\frkg{\mathfrak{g}}

\def\frko{\mathfrak{o}}

\def\frkz{\mathfrak{z}}

\RequirePackage[normalem]{ulem} %DIF PREAMBLE
\RequirePackage{color}\definecolor{RED}{rgb}{1,0,0}\definecolor{BLUE}{rgb}{0,0,1} %DIF PREAMBLE
\RequirePackage{listings} %DIF PREAMBLE
\RequirePackage{color} %DIF PREAMBLE
\lstdefinelanguage{DIFcode}{ %DIF PREAMBLE
  moredelim=[il][\color{red}\sout]{\%DIF\ <\ }, %DIF PREAMBLE
  moredelim=[il][\color{blue}\uwave]{\%DIF\ >\ } %DIF PREAMBLE
} %DIF PREAMBLE
\lstdefinestyle{DIFverbatimstyle}{ %DIF PREAMBLE
	language=DIFcode, %DIF PREAMBLE
	basicstyle=\ttfamily, %DIF PREAMBLE
	columns=fullflexible, %DIF PREAMBLE
	keepspaces=true %DIF PREAMBLE
} %DIF PREAMBLE
\lstnewenvironment{DIFverbatim}{\lstset{style=DIFverbatimstyle}}{} %DIF PREAMBLE
\lstnewenvironment{DIFverbatim*}{\lstset{style=DIFverbatimstyle,showspaces=true}}{} %DIF PREAMBLE
\begin{document}
%\pagewiselinenumbers
\maketitle
\begin{abstract}
We prove an explicit formula for the Petersson norms of some normalized generic cuspidal newforms
on $\GSp_4$ whose archimedean components belong to either discrete series representations or spherical principal series representations. Our formula expresses the Petersson norms in terms of special values of adjoint $L$-functions and some elementary constants depending only on local representations.
\end{abstract}

\tableofcontents
\section{Introduction}\label{s:intro}
In the study of automorphic forms, it is natural and important to measure their size.
For example, let $f \in S_\kappa(\SL_2(\Z))$ be a normalized Hecke eigenform and consider its Petersson norm 
\[
 \langle f, f \rangle = \int_{\SL_2(\Z) \backslash \mathfrak{H}}
 |f(\tau)|^2 \operatorname{Im}(\tau)^{\kappa-2} \, d \tau.
\]
Then
\[
 \langle f, f \rangle = 2^{-\kappa} L(1, f, \Ad),
\]
where $L(s, f, \Ad)$ is the (completed) adjoint $L$-function of $f$, and this formula has many applications in analytic number theory and arithmetic geometry.

More generally, let $k$ be a number field with adele ring $\A_k$ and $G$ a connected reductive linear algebraic group over $k$.
Let $\pi = \bigotimes_v \pi_v$ be an irreducible cuspidal automorphic representation of $G(\A_k)$.
We assume that $G$ is quasi-split over $k$, $\pi$ occurs with multiplicity one in the automorphic discrete spectrum of $G$, and $\pi$ is globally generic, i.e. the Whittaker function $W_f$ is non-zero for some $f \in \pi$.
Then the Lapid--Mao conjecture \cite{LM2015} predicts that
\[
 \langle f, f \rangle = |\mathcal{S}_\pi| \cdot \frac{L^S(1, \pi, \Ad)}{\Delta_G^S} \cdot \prod_{v \in S} \alpha_v(f_v)^{-1}
\]
for any $f = \bigotimes_v f_v \in \pi$ normalized so that $W_f(1) = 1$.
Here $\langle f, f \rangle$ is the Petersson norm of $f$, $\mathcal{S}_\pi$ is the global component group of the (conjectural) Arthur parameter of $\pi$, $S$ is a sufficiently large finite set of places of $k$, $L^S(s, \pi, \Ad)$ is the partial adjoint $L$-function of $\pi$, $\Delta_G^S$ is a special value of some partial $L$-function depending only on $G$, and $\alpha_v(f_v)$ is the Whittaker integral of the matrix coefficient associated to $f_v \in \pi_v$.
In fact, Lapid and Mao deduced the conjecture for $G = \GL_n$ from the theory of Rankin-Selberg integrals developed by Jacquet, Piatetski-Shapiro, and Shalika.
Moreover, in \cite{LM2017}, they established an analogous formula when $G$ is the metaplectic group $\mathrm{Mp}_{2n}$ (which is a nonlinear two-fold cover of the symplectic group $\Sp_{2n}$) under the assumption that $\pi_v$ is a discrete series representation for all archimedean places $v$ of $k$.
Recently, Furusawa and Morimoto proved the conjecture for $G=\GSp_4$ in \cite[Theorem 6.3]{FM2022}.
However, to derive from the Lapid-Mao conjecture an explicit formula for $\langle f, f \rangle$ suitable for applications, it is necessary to carry out the computation of the local integral $\alpha_v(f_v)$, which is extremely hard when $v$ is archimedean.

The purpose of this paper is to prove some explicit formula for Petersson norms when $G = \PGSp_4$.
For simplicity, we assume in the introduction that $k=\Q$, $\pi_p$ is unramified for all primes $p$, $\pi_\infty$ is either a discrete series representation or a spherical principal series representation, and $f \in \pi$ is the normalized new vector (see below for the details).
Although the Lapid-Mao conjecture has been proved in this case, it seems impossible to compute $\alpha_\infty(f_\infty)$ directly.
Nevertheless, by adopting a different approach and using reduction to the endoscopic case, we establish an explicit formula for $\langle f, f \rangle$, which we now describe in more detail.

\subsection{An explicit formula for Petersson norms}

Let $\pi=\bigotimes_v\pi_v$ be an irreducible globally generic cuspidal automorphic representation of $\GSp_4(\A_\Q)$ with trivial central character. By \cite[Theorem 12.1]{GT2011} (see also \cite{CKPS2004} and \cite{AH2006}),
$\pi$ has a strong functorial lift $\Pi$ to $\GL_4(\AA_\Q)$.
We say that $\pi$ is stable (resp.~endoscopic)
if $\Pi$ is cuspidal (resp.~non-cuspidal).
We assume $\pi_p$ is unramified for all primes $p$, and consider the following two cases:
\begin{itemize}
\item[(DS)] 
\[
\pi_{\infty}\vert_{\Sp_4(\R)} = D_{(\lambda_1,\lambda_2)}\oplus D_{(-\lambda_2,-\lambda_1)},
\]
where $D_{(\lambda_1,\lambda_2)}$ is the (limit of) discrete series representation of $\Sp_4(\R)$ with Blattner parameter $(\lambda_1,\lambda_2) \in \Z^2$ such that $1-\lambda_1 \leq \lambda_2 \leq 0$. 
\item[(PS)] 
\[
\pi_\infty \vert_{\Sp_4(\R)} = {\rm Ind}_{\Sp_4(\R)\cap {\bf B}(\R)}^{\Sp_4(\R)}(|\mbox{ }|^{\lambda_1}\boxtimes|\mbox{ }|^{\lambda_2})
\]
for some $\lambda_1, \lambda_2 \in \C$.
\end{itemize}
Here ${\bf B}$ is the standard Borel subgroup of $\GSp_4$. 

Let $f=\bigotimes_v f_v \in \pi$ be a non-zero cusp form satisfying the following conditions:
\begin{align}\label{E:K-type}
\begin{split}
&f_p\mbox{ is }\GSp_4(\Z_p)\mbox{-invariant for all primes }p.\\
&\mbox{In Case (DS), }f_{\infty}\mbox{ is a lowest weight vector of the minimal }{\rm U}(2)\mbox{-type of }D_{(-\lambda_2,-\lambda_1)}.\\
&\mbox{In Case (PS), }f_\infty \mbox{ is an $(\Sp_4(\R)\cap {\rm O}(4))$-invariant vector.}
\end{split}
\end{align}
For the choice of $f_\infty$ in Case (DS), we refer to \S\,\ref{SS:main thm} for more detail. 
By \cite{JS2007}, the multiplicity of $\pi$
in the space of cusp forms on $\GSp_4(\AA_\Q)$ is one.
Hence the conditions in (\ref{E:K-type}) uniquely determine $f \in \pi$ up to scalars.
Let $U$ be the unipotent radical of ${\bf B}$ and $\psi_U$ the standard non-degenerate character of $U(\Q)\backslash U(\A_\Q)$ (see \S\,\ref{S:notation} for precise definition).
The Whittaker function of $f$ with respect to $\psi_U$ is defined by
\[
W(g)=\int_{U(\Q)\backslash U(\A_\Q)}f(ug)\overline{\psi_U(u)}\,du.
\]
Here $du$ is the Tamagawa measure on $U(\A_\Q)$. We may decompose $W$ into a product of local Whittaker functions $W=\prod_v W_v$. By explicit formulae for Whittaker functions
\cite{CS1980}, \cite{Oda1994}, \cite{Moriyama2004}, \cite{Ishii2005}, and \cite{RS2007},
$W_v(1) \ne 0$ for all places $v$. We normalize $f$ as follows (see also Remark \ref{R:normalization} below): Let $W_p(1)=1$ for all primes $p$.
In Case (DS), let
\begin{align}\label{E:normalization DS}
\begin{split}
W_\infty(1)&=
e^{-2\pi}\int_{c_1-\sqrt{-1}\infty}^{c_1+\sqrt{-1}\infty}\frac{ds_1}{2\pi \sqrt{-1}}\,\int_{c_2-\sqrt{-1}\infty}^{c_2+\sqrt{-1}\infty}\frac{ds_2}{2\pi \sqrt{-1}}\,(4\pi^{3})^{(-s_1+\lambda_{1}+1)/2}(4\pi )^{(-s_2+\lambda_{2})/2}\\
&\quad\quad\quad\quad\quad\quad\quad\quad\quad\times \Gamma\left(\frac{s_1+s_2-2\lambda_{2}+1}{2}\right)\Gamma\left(\frac{s_1+s_2+1}{2}\right)\Gamma\left(\frac{s_1}{2}\right)\Gamma\left(\frac{-s_2}{2}\right),
\end{split}
\end{align}
where $c_1,c_2 \in \R$ satisfy 
$c_1+c_2+1>0$ and $c_1>0>c_2.$ In Case (PS), let
\begin{align}\label{E:normalization PS}
\begin{split}
W_\infty(1) &= \int_{c_1-\sqrt{-1}\infty}^{c_1+\sqrt{-1}\infty}\frac{ds_1}{2\pi \sqrt{-1}}\,\int_{c_2-\sqrt{-1}\infty}^{c_2+\sqrt{-1}\infty}\frac{ds_2}{2\pi \sqrt{-1}}\,2^{-4}\pi^{-s_1-s_2}\\
&\times\Gamma\left(\frac{s_1+\lambda_{1}}{2}\right)\Gamma\left(\frac{s_1-\lambda_{1}}{2}\right)\Gamma\left(\frac{s_1+\lambda_{2}}{2}\right)\Gamma\left(\frac{s_1-\lambda_{2}}{2}\right)\\
&\times\Gamma\left(\frac{s_2}{2}+\frac{\lambda_{1}+\lambda_{2}}{4}\right)\Gamma\left(\frac{s_2}{2}-\frac{\lambda_{1}+\lambda_{2}}{4}\right)\Gamma\left(\frac{s_2}{2}+\frac{\lambda_{1}-\lambda_{2}}{4}\right)\Gamma\left(\frac{s_2}{2}-\frac{\lambda_{1}-\lambda_{2}}{4}\right)\\
&\times\Gamma\left(\frac{s_1+s_2}{2}+\frac{\lambda_{1}+\lambda_{2}}{4}\right)^{-1}\Gamma\left(\frac{s_1+s_2}{2}-\frac{\lambda_{1}+\lambda_{2}}{4}\right)^{-1}\\
&\times {}_3F_2\left(\frac{s_1}{2},\frac{s_2}{2}+\frac{\lambda_{1}-\lambda_{2}}{4}, \frac{s_2}{2}-\frac{\lambda_{1}-\lambda_{2}}{4} ;\, \frac{s_1+s_2}{2}+\frac{\lambda_{1}+\lambda_{2}}{4},\frac{s_1+s_2}{2}-\frac{\lambda_{1}+\lambda_{2}}{4};\,1           \right),
\end{split}
\end{align}
where ${}_3F_2$ is a generalized hypergeometric function and $c_1,c_2 \in \R$ satisfy
\[
c_1>\max\left\{ |{\rm Re}(\lambda_{1})|,\,|{\rm Re}(\lambda_{2})|\right\},\quad c_2>\max\left\{\left\vert{\rm Re}\left(\frac{\lambda_{1}+\lambda_{2}}{2}\right)\right\vert,\,\left\vert{\rm Re}\left(\frac{\lambda_{1}-\lambda_{2}}{2}\right)\right\vert\right\}.
\]
Note that by \cite[Theorem 3.2 and Proposition 3.5]{Ishii2005}, the right-hand side of (\ref{E:normalization PS}) depends only on the orbit of $(\lambda_1,\lambda_2)$ under the action of the Weyl group.
Let $\<f,f\>$ be the Petersson norm of $f$ defined by 
\[
\<f,f\>=\int_{\A_\Q^\times\GSp_4(\Q)\backslash \GSp_4(\A_\Q)}|f(g)|^2\,dg.
\]
Here $dg$ is the Tamagawa measure on $\GSp_4(\A_\Q)$. 

\begin{thm}\label{T:intro}
We have
\[\<f,f\>=2^c\cdot \frac{L(1,\pi,{\rm Ad})}{\Delta_{\PGSp_4}}\cdot C_\infty.\]
Here ${\rm Ad}$ is the adjoint representation of $\GSp_4(\C)$ on $\frak{pgsp}_4(\C)$,
\begin{align*}
\Delta_{\PGSp_4} &= \xi(2)\xi(4),\\
c&=\begin{cases}1 & \mbox{ if $\pi$ is stable},\\
2 & \mbox{ if $\pi$ is endoscopic},\end{cases}\\
C_\infty & = \begin{cases}
\displaystyle{2^{\lambda_{1}-\lambda_{2}+5}\pi^{3\lambda_{1}-\lambda_{2}+5}(1+\lambda_{1}-\lambda_{2})^{-1}}  & \mbox{ in Case (DS)},\\
\displaystyle{2^{-4}} & \mbox{ in Case (PS)},
\end{cases}
\end{align*}
where $\xi(s)$ is the completed Riemann zeta function.
\end{thm}

%The purpose here is to compute the constant $C_{\infty}$. Let $\kappa_1 = \lambda_1-\lambda_2$ and $\kappa_2 = \lambda_1+\lambda_2$. Note that $\pi_{\infty}$ can realized as a local theta lift from discrete series representation of ${\rm GO}(2,2)$ with weight $(\kappa_1,\kappa_2)$. The idea is to replace $\pi$ by Yoshida lifts and compute the Petersson norm by the Rallis inner product formula.
%\begin{rmk}
%Let $\calL_{\QQ}$ be the hypothetical Langlands group of $\QQ$.
%Let $\psi_{\pi}: \calL_{\QQ} \times \SL_2(\CC) \rightarrow \G(\CC)$
%be the (conjectural) Arthur parameter associated to $\pi$
%and $\calS_{\psi_{\pi}}$ the centralizer of the image of $\psi_{\pi}$
%in $\G(\CC)$.
%Then Arthur's conjecture \cite{Arthur1989}, \cite{Arthur2004} asserts that
%\[
% |\calS_{\psi_{\pi}}| =
% \begin{cases}
%  2 & \text{if $\pi$ is stable,} \\
%  4 & \text{if $\pi$ is endoscopic.}
% \end{cases}
%\]
%Hence the constant $2^c$ is related to the theory of endoscopy.
%(See also \cite{IchinoIkeda2010}.) In particular, the Theorem is compatible with the Lapid-Mao conjecture \cite{LM2015}. In the conjectural formula proposed by Lapid and Mao, the analogue archimedean local constant is defined by certain stable integral of matrix coefficient of $\pi_{\infty}$. In general, the integral seems difficult to compute directly.
%\end{rmk}

\begin{rmk}
In Case (DS), $1 + \lambda_1 - \lambda_2$ is the dimension of the minimal $\U(2)$-type of $D_{(-\lambda_2,-\lambda_1)}$.
\end{rmk}

\begin{rmk}\label{R:normalization}
Recall that, in the case of $\GL_2(\R)$, we may normalize the Whittaker function so that its value at $\diag(a,1)$ with $a > 0$ is given by 
\[
 a^{\kappa/2} e^{- 2 \pi a} = \frac{1}{2 \pi \sqrt{-1}} \int_{c-\sqrt{-1} \infty}^{c+\sqrt{-1} \infty} (2 \pi)^{-s-\kappa/2} \Gamma\left(s+\frac{\kappa}{2} \right) a^{-s} \, ds
\]
for the (limit of) discrete series representation with minimal weight $\kappa \in \Z_{\ge 1}$ and 
\[
 a^{1/2} K_\lambda(2 \pi a) = \frac{1}{2 \pi \sqrt{-1}} \int_{c-\sqrt{-1} \infty}^{c+\sqrt{-1} \infty} 2^{-2} \pi^{-s-1/2} \Gamma\left( \frac{s+\lambda}{2} + \frac{1}{4} \right) \Gamma\left( \frac{s-\lambda}{2} + \frac{1}{4} \right) a^{-s} \, ds
\]
for the spherical principal series representation $\Ind^{\GL_2(\R)}_{B(\R)}(|\ |^\lambda \boxtimes |\ |^{-\lambda})$, where $B$ is the standard Borel subgroup of $\GL_2$. 
Then the corresponding Jacquet-Langlands' local zeta integral for $\GL_2$ gives the standard $L$-factor.
Similarly, if we normalize the Whittaker function on $\GSp_4(\R)$ as in (\ref{E:normalization DS}) and (\ref{E:normalization PS}), then the corresponding Novodvorsky's local zeta integral for $\GSp_4$ gives the spinor $L$-factor (cf.\,\cite{Moriyama2004} and \cite{IM2008}).
It would also be interesting to compare our normalization in Case (DS) with the rational structures studied by Harris and Kudla \cite{HK1992}. 
\end{rmk}

\begin{rmk}
In Case (DS), Theorem \ref{T:intro} plays an important role in the result of Lemma and Ochiai \cite{LO2018} on the congruence between a fixed irreducible globally generic endoscopic cuspidal automorphic representation of $\GSp_4(\A_\Q)$ and irreducible stable cuspidal automorphic representations of $\GSp_4(\A_\Q)$.
\end{rmk}

\subsection{An outline of the proof}
The proof of Theorem \ref{T:intro} consists of two steps.
\begin{itemize}
\item
We first prove that there exists a constant $\widetilde{C}_\infty$ depending only on $\pi_\infty$ such that 
\begin{equation}
\label{eq:pf1}
 \langle f, f \rangle = 2^c \cdot \frac{L(1, \pi, \Ad)}{\Delta_{\PGSp_4}} \cdot \widetilde{C}_\infty.
\end{equation}
For this, we consider the automorphic $L$-function
\[
 L(s, \pi \times \pi^{\vee}) = \xi(s) \cdot L(s, \pi, {\rm std}) \cdot L(s, \pi, \Ad),
\]
where $\pi^{\vee}$ is the contragredient representation of $\pi$ and $L(s, \pi, {\rm std})$ is the standard $L$-function of $\pi$.
Recall the integral representations (ignoring factors which do not depend on $\pi$)
%\begin{align*}
% \langle \mathcal{E}(s), \bar{f} \otimes f \rangle 
% & = L_{\mathrm{fin}} \left( \frac{s+1}{2}, \pi \times \pi^{\vee} \right) \cdot \mathcal{Z}(s, \pi_\infty), \\
% \langle E(s), \bar{f} \otimes f \rangle  
% & = \langle f, f \rangle \cdot L_{\mathrm{fin}}\left( s+\frac{1}{2}, \pi, {\rm std} \right) \cdot Z(s, \pi_\infty)
%\end{align*}
\begin{align*}
 \langle \mathcal{E}(s), \bar{f} \otimes f \rangle 
 & = L\left( \frac{s+1}{2}, \pi \times \pi^{\vee} \right) \cdot \mathcal{Z}(s, \pi_\infty), \\
 \langle E(s), \bar{f} \otimes f \rangle  
 & = \langle f, f \rangle \cdot L\left( s+\frac{1}{2}, \pi, {\rm std} \right) \cdot Z(s, \pi_\infty)
\end{align*}
due to Jiang \cite{Jiang1996} and Piatetski-Shapiro and Rallis \cite{LNM1254}, respectively. 
Here $\mathcal{E}(s)$ and $E(s)$ are Eisenstein series on $\GSp_8$ induced from characters of Levi subgroups $\GL_3 \times \GSp_2$ and $\GL_4 \times \GL_1$, respectively, %$L_{\mathrm{fin}}(s,\dots)$ denotes the $L$-function without the archimedean factor, 
and $\mathcal{Z}(s, \pi_\infty)$ and $Z(s, \pi_\infty)$ are the associated archimedean zeta integrals.
Recall also that these Eisenstein series have Laurent expansions of the form
\begin{align*}
 \mathcal{E}(s) & = \frac{\mathcal{E}_{-2}(1)}{(s-1)^2} + \frac{\mathcal{E}_{-1}(1)}{s-1} + \cdots, \\
 E(s) & = \frac{E_{-1}(\frac{1}{2})}{s-\frac{1}{2}} + E_0\left( \frac{1}{2} \right) + \cdots.
\end{align*}
By the Siegel-Weil formula \cite{GQT2014}, we have
\begin{align*}
\mathcal{E}_{-2}(1)& = E_{-1}\left(\frac{1}{2}\right),\\
\mathcal{E}_{-1}(1)& = E_0\left(\frac{1}{2}\right) + \mathbb{E}_0\left(\frac{1}{2}\right)+{\widetilde{E}}_{-1}\left(\frac{1}{2}\right)
\end{align*}
%\DIFaddbegin\DIFadd{
for some Eisenstein series $\mathbb{E}(s)$ and ${\widetilde{E}}(s)$ on $\GSp_8$ induced from characters of Levi subgroups $\GL_2 \times \GSp_4$ and $\GL_4 \times \GL_1$, respectively.
Then we have
\[
\langle \mathcal{E}_{-2}(1), \bar{f} \otimes f \rangle 
 = \left\langle  E_{-1} \left( \frac{1}{2} \right), \bar{f} \otimes f \right\rangle.
\]
In fact, this identity says $0=0$ when $\pi$ is stable.
Moreover, when $\pi$ is stable, we show that
%\DIFaddend
\[
%\DIFaddbegin\DIFadd{
\left \langle \mathbb{E}_0\left(\frac{1}{2}\right), \bar{f} \otimes f\right \rangle = \left \langle \widetilde{E}_{-1}\left(\frac{1}{2}\right), \bar{f} \otimes f\right \rangle=0,%\DIFaddend
\]
so that
\begin{align*}
 \langle \mathcal{E}_{-1}(1), \bar{f} \otimes f \rangle 
 & = \left\langle E_{0} \left( \frac{1}{2} \right), \bar{f} \otimes f \right\rangle.
\end{align*}
Combining this with the integral representations, and noting that $L(s, \pi, {\rm std})$ is holomorphic and non-zero (resp.~has a simple pole) at $s=1$ if $\pi$ is stable (resp.~endoscopic), we can deduce \eqref{eq:pf1} with 
\[
 \widetilde{C}_\infty = \frac{1}{L(1, \pi_\infty, \Ad)} \cdot 
 \frac{\mathcal{Z}(1, \pi_\infty)}{Z(\frac{1}{2}, \pi_\infty)}.
\]
\item
We next show that 
\begin{equation}
\label{eq:pf2}
 \widetilde{C}_\infty = C_\infty.
\end{equation}
This is a local problem, but we address it by a global argument.
Suppose that $\pi$ is endoscopic. Recall that $\pi$ is of type (DS) or (PS) at the archimedean place and unramified at all finite places. Since $\pi$ is endoscopic, it can be realized as a global theta lift from the orthogonal similitude group ${\rm GO}_{2,2}$.
Hence, by using the Rallis inner product formula \cite{GQT2014}, we can show that
\begin{equation}\label{eq:pf3}
 \langle f, f \rangle = 2^c \cdot \frac{L(1, \pi, \Ad)}{\Delta_{\PGSp_4}} \cdot {C}_\infty'
\end{equation}
for some constant $C_\infty'$ defined by doubling local zeta integral depending only on the archimedean component. Moreover, by explicit computation, we have
\begin{align}\label{eq:pf4}
C_\infty' = C_\infty.
\end{align}
 Therefore, if $\pi_\infty$ is the archimedean component of such $\pi$, then we can deduce (\ref{eq:pf2}) for $\pi_\infty$ from (\ref{eq:pf1}), (\ref{eq:pf3}), and (\ref{eq:pf4}).
Now we take an arbitrary representation $\pi_\infty$ of $\GSp_4(\R)$ of type (DS) or (PS). When $\pi_\infty$ is of type (DS) and sufficiently regular, the existence of its globalization $\pi$ is clear.
%For example, when $\pi_\infty$ is of type (DS) and sufficiently regular, the existence of $\pi$ is clear. 
When $\pi_\infty$ is of type (PS), we can deduce from the limit multiplicity formula \cite{Shin2012} that there exists $\pi$ whose archimedean component is arbitrarily close to $\pi_\infty$.
(Strictly speaking, we need to consider cusp forms of square-free level over totally real number fields.)
Thus we may conclude \eqref{eq:pf2} for $\pi_\infty$ by showing that $\widetilde{C}_\infty = \widetilde{C}_\infty(\pi_\infty)$ is analytic as a function of $\pi_\infty$. 
Finally, when $\pi_\infty$ is of type (DS) without regularity assumption, we consider cusp forms of square-free level in the global argument. In this case, the formulas analogous to (\ref{eq:pf1}) and (\ref{eq:pf3}) still hold but some extra constants depending only on the non-archimedean components appear in the formulas. Thus we also need to determine these constants, but by a similar argument using the limit multiplicity formula \cite{Shin2012}, we are reduced to compute certain non-archimedean doubling local zeta integral.
\end{itemize}

\subsubsection*{Acknowledgement}
The authors would like to thank Valentin Blomer, Yao Cheng, Masaaki Furusawa, Ming-Lun Hsieh, Tamotsu Ikeda, Francesco Lemma, and Tadashi Ochiai for their advice and encouragement. The first-named author is partially supported by MOST GSSAP with Grant Number 2917-I-002-003. The second-named author is partially supported by JSPS KAKENHI Grant Number 26287003.

\section{Main result}\label{Main results}
\subsection{Notation}\label{S:notation}

Fix a totally real number field $k$. Let $\o$, $\frak d$, and $\frak{D}$ be the ring of integers, the different ideal, and the absolute discriminant of $k$, respectively.
Let $\A=\A_k$ be the ring of adeles of $k$ and $\A_f$ its finite part. For a finite dimensional vector space $V$ over $k$, let $\mathcal{S}(V(\A))$ be the space of Schwartz functions on $V(\A)$. Let $\psi_0=\bigotimes_v\psi_{v,0}$ be the standard additive character of $\Q\backslash \A_\Q$ defined so that
\begin{align*}
\psi_{p,0}(x) & = e^{-2\pi \sqrt{-1}x} \mbox{ for }x \in \Z[p^{-1}],\\
\psi_{\infty,0}(x) & = e^{2\pi \sqrt{-1}x} \mbox{ for }x \in \R.
\end{align*}
Let $\psi= \psi_k = \psi_0\circ{\rm tr}_{k/\Q}$ be the standard additive character of $k \backslash \A$.

Let $v$ be a place of $k$. We denote by $\psi_v$ the $v$-component of $\psi$. If $v$ is a finite place, let $\frak{o}_v$, $\varpi_v$, and $q_v$ be the maximal compact subring of $k_v$, a generator of the maximal ideal of $\frak{o}_v$, and the cardinality of $\frak{o}_v / \varpi_v\frak{o}_v$. Let $|\mbox{ }|_v$ be the absolute value on $k_v$ normalized so that $|\varpi_v|_v = q_v^{-1}$. Let $\c_v$ be the integer such that $\frak{d}_v = \varpi_v^{\c_v}\o_v$. Note that $\frak{c}_v$ is the largest integer so that $\psi_v$ is trivial on $\varpi_v^{-\frak{c}_v}\o_v$. If $v$ is a real place, let $|\mbox{ }|_v$ be the ordinary absolute value on $k_v$.

Let $\xi(s) = \xi_k(s) = \prod_{v}\zeta_v(s)$ be the completed Dedekind zeta function of $k$, where $v$ ranges over the places of $k$ and 
\[\zeta_v(s) = \begin{cases}
(1-q_v^{-s})^{-1} & \mbox{ if $v$ is finite},\\
\pi^{-s/2}\Gamma(s/2) & \mbox{ if $v$ is real}.
\end{cases}
\]
Here $\Gamma(s)$ is the gamma function. Define $\rho = {\rm Res}_{s=1}\xi(s)$. For a finite set $S$ of places of $k$, let $\xi^S(s) = \prod_{v \notin S}\zeta_v(s)$ be the partial Dedekind zeta function of $k$.

If $S$ is a set, then we let $\mathbb{I}_S$ be the characteristic function of $S$. Let ${\rm M}_{n , m}$ be the matrix algebra of $n$ by $m$ matrices.
Let $\GSp_{2n}$ denote the symplectic similitude group of rank $n+1$ defined by
\[
 \GSp_{2n} =
 \left\{ g \in \GL_{2n} \, \left| \, g
 \begin{pmatrix}
  0      & \1_n \\
  - \1_n & 0
 \end{pmatrix}
 {}^t \! g = 
  \nu(g) \begin{pmatrix}
  0      & \1_n \\
  - \1_n & 0
 \end{pmatrix}, \,\nu(g) \in \mathbb{G}_m
 \right. \right\}.
\]
Let $\Sp_{2n} = {\rm ker}(\nu)$. Throughout this article, we write 
\[G= \GSp_4\]
unless otherwise specified. Let $Z_G$ be the center of $G$. Let \[{\bf B}=\left.\left\{\bp t_1 & * & *&* \\0&t_2&*&*\\0&0&\nu t_1^{-1}&0\\0&0& *&\nu t_2^{-1}  \ep\in \G  \mbox{ }\right\vert\mbox{ } t_1,t_2,\nu \in {\mathbb G}_m\right\}\]
be the standard Borel subgroup of $\G$, and let $U$ be its unipotent radical. Let $\psi_U$ be the non-degenerate character of $U(k)\backslash U(\A)$ defined by
\[
\psi_U\left(\bp 1 & x & *&* \\0&1&*&y\\0&0&1&0\\0&0&-x&1 \ep\right)= \psi(-x-y).
\] 
Let $\bf T \subset {\bf B}$ be the standard maximal tours of $G$. 
In $\GL_2$, let $B$ be the Borel subgroup consisting of upper triangular matrices, and put
\[
{\bf a}(\nu) = \bp \nu & 0 \\ 0 & 1 \ep,\quad {\bf d}(\nu) =   \bp 1 & 0 \\ 0 & \nu \ep,\quad {\bf m}(t)= \bp t & 0 \\ 0 & t^{-1}\ep,\quad {\bf n}(x) = \bp 1 & x \\ 0 & 1\ep,\quad w = \bp 0 & 1 \\ -1 & 0\ep
\]
for $\nu,t \in \mathbb{G}_m$ and $x \in \mathbb{G}_a$. Let
\[
{\rm SO}(2) = \left.\left\{ k_\theta = \bp \cos\theta & \sin\theta \\ -\sin\theta & \cos\theta \ep \mbox{ }\right\vert\mbox{ } \theta \in \R/2\pi\Z \right\}.
\]

Let $v$ be a finite place of $k$. Let $K_0(\varpi_v)$ be the Iwahori subgroup of $\GL_2(k_v)$ defined by
\[K_0(\varpi_v) = \bp \o_v & \o_v \\ \varpi_v\o_v & \o_v\ep \cap \GL_2(\o_v).\] The paramodular group ${\rm K}(\varpi_v)$ of level $\varpi_v\o_v$ is the subgroup of $G(k_v)$ consisting of $g \in G(k_v)$ such that $\nu(g) \in \o_v^\times$ and 
\[g \in \bp \o_v & \o_v & \varpi_v^{-1}\o_v & \o_v\\
\varpi_v\o_v & \o_v & \o_v & \o_v\\
\varpi_v\o_v & \varpi_v\o_v & \o_v & \varpi_v\o_v\\
\varpi_v\o_v & \o_v & \o_v & \o_v
 \ep.\]

For $\nu \in \C$, let $K_{\nu}(z)$ be the modified Bessel function defined by
\begin{align}\label{E:K Bessel}
K_{\nu}(z) = \frac{1}{2}\int_{0}^{\infty}e^{-z(t+t^{-1})/2}t^{\nu-1}\,dt
\end{align}
if ${\rm Re}(z)>0$. 
\subsection{Measures}\label{SS:measures}
Let $v$ be a place of $k$. 
If $v$ is finite, we normalize the Haar measures on $k_v$ and $k_v^\times$ so that ${\rm vol}(\o_v)=1$ and ${\rm vol}(\o_v^\times)=1$, respectively. If $v$ is real, we normalize the Haar measures on $k_v \simeq \R$ and $k_v^\times \simeq \R^\times$ so that ${\rm vol}([1,2])=1$ and ${\rm vol}([1,2])=\log 2$, respectively.
Let $m$ be a positive integer. Let $dg_v$ be the Haar measure on $\GL_m(k_v)$ defined as follows: For $\phi \in L^1(\GL_m(k_v))$, we have 
\begin{align}\label{E:standard measure on GL}
\int_{\GL_m(k_v)}\phi(g_v)\,dg_v = \prod_{1 \leq i<j \leq m}\int_{k_v}du_{ij}\prod_{1\leq i \leq m}\int_{k_v^{\times}}d^{\times}t_i\int_{K_v}dk\,\phi\left(\bp t_1 & u_{12} & \cdots &u_{1m} \\ 0 & t_2 & \cdots & u_{2m} \\ \vdots & \vdots & \ddots& \vdots \\ 0 & 0 & \cdots & t_m\ep k\right)\prod_{1\leq i \leq m}|t_i|_v^{-m+i},
\end{align}
where 
\begin{align*}
K_v = \begin{cases}    
\GL_m(\frak{o}_v)& \mbox{ if $v$ is finite},\\
{\rm O}(m) & \mbox{ if $v$ is real},
\end{cases}
\end{align*}
and ${\rm vol}(K_v)=1$. Let $H$ be a connected reductive linear algebraic group defined and split over $k_v$. Fix a Chevalley basis of ${\rm Lie}(H)$. The basis determines a top differential form on $H$ over $\Z$ which is unique up to $\pm 1$. The top differential form together with the self-dual Haar measure on $k_v$ with respect to $\psi_v$ determines a Haar measure on $H(k_v)$ as explained in \cite[\S 6]{Vos1996}. We call it the local Tamagawa measure on $H(k_v)$.

For a connected linear algebraic group $H$ over $k$, we take the Tamagawa measure on $H(\A)$ (cf.~\cite[\S 6]{Vos1996}). For any compact group $K$, we take the Haar measure on $K$ such that ${\rm vol}(K)=1$.
Let $dg$ be the Tamagawa measure on $\GL_m(\A)$. Then (cf.\,\cite[\S 5]{Lai1980})
\begin{align}\label{Tamagawa measure on GL}
dg = \frak{D}^{-m^2/2}\rho^{-1}\prod_{i=2}^m\xi(i)^{-1} \cdot\prod_v dg_v.
\end{align}

\subsection{Automorphic representations of $\GSp_4$}
Let $\pi=\bigotimes_v \pi_v$ be an irreducible globally generic cuspidal automorphic representation of $G(\A)$ with trivial central character. 
By \cite{CKPS2004}, \cite{AH2006}, \cite[Theorem 12.1]{GT2011}, $\pi$ has a strong functorial lift $\Pi$ to $\GL_4(\A)$.
By \cite{GRS2001}, \cite[Theorem 12.1]{GT2011}, 
either $\Pi$ is cuspidal or $\Pi = \sigma_1 \boxplus \sigma_2$ for some irreducible cuspidal automorphic representations $\sigma_1$ and $\sigma_2$
of $\GL_2(\AA)$ with trivial central character such that $\sigma_1 \neq \sigma_2$.
We say that $\pi$ is stable (resp.~endoscopic)
if $\Pi$ is cuspidal (resp.~non-cuspidal).

Recall that the dual group of $G$ is $\GSp_4(\C)$.
Let ${\rm Ad}$ denote the adjoint representation of $\GSp_4(\C)$ on $\frak{pgsp}_4(\C)$, and $\rm std$ the composition of the projection $\GSp_4(\C) \rightarrow \PGSp_4(\C)$ with the standard
representation of $\PGSp_4(\C) \simeq {\rm SO}_5(\C)$ on $\C^5$.
Let $S$ be a finite set of places of $k$
including the archimedean places such that,
for $v \notin S$, $\pi_v$ is unramified.
Then the partial adjoint and standard $L$-functions of $\pi$ are defined as the Euler products
\[
L^S(s,\pi,{\rm Ad}) = \prod_{v \notin S} L(s,\pi_v,{\rm Ad}),\quad L^S(s,\pi,{\rm std}) = \prod_{v \notin S} L(s,\pi_v,{\rm std})
\]
for $s \in \C$, which are absolutely convergent for ${\rm Re}(s)$ sufficiently large.
Also, we have 
\begin{align*}
 L^S(s, \Pi, \Sym^2) & = L^S(s, \pi, \Ad), \\
 L^S(s, \Pi, \wedge^2) & = \xi^S(s) L^S(s, \pi, {\rm std}).
\end{align*}
In particular, $L^S(s,\pi,{\rm Ad})$ and $L^S(s,\pi,{\rm std})$ admit meromorphic continuations to $\C$.
(In a more general context, the meromorphic continuation of $L^S(s,\pi,{\rm std})$ was established by Piatetski-Shapiro and Rallis \cite{LNM1254} much earlier.)
By \cite{GRS2001}, \cite[Theorem 12.1]{GT2011}, 
$L^S(s, \Pi, \wedge^2)$ has a simple (resp.~double) pole at $s=1$
if $\pi$ is stable (resp.~endoscopic).
Hence $L^S(s, \pi, {\rm std})$ is holomorphic and non-zero
(resp.~has a simple pole) at $s=1$ if $\pi$ is stable (resp.~endoscopic).
Moreover, $L^S(s, \pi, \Ad)$ is holomorphic and non-zero at $s=1$.

For any place $v$ of $k$, we denote by $\Phi_v : L_{k_v} \rightarrow \GSp_4(\C)$ the local $L$-parameter attached to $\pi_v$ by the local Langlands correspondence established by Gan and Takeda \cite{GT2011} if $v$ is finite and by Langlands \cite{Langlands1989} if $v$ is real. 
Here $L_{k_v}$ is the Weil--Deligne group of $k_v$ if $v$ is finite but the Weil group of $k_v$ if $v$ is real.
Since $\Pi_v$ is unitary and generic (and hence ``almost tempered''), the adjoint $L$-factor 
\[
 L(s,\pi_v,{\rm Ad}) = L(s,{\rm Ad}\circ \Phi_v)
\]
defined as in \cite[\S 3]{Tate1979} is holomorphic at $s=1$.
In fact, the same holds for any irreducible generic admissible representation of $G(k_v)$ (see \cite[Conjecture 2.6]{GP1992}, \cite{AS2008}, \cite{GT2011}, \cite[Proposition B.1]{GI2016}).
Hence the completed adjoint $L$-function $L(s, \pi, \Ad)$ is holomorphic and non-zero at $s=1$.

\subsection{Main result}\label{SS:main thm}
Let $\pi=\bigotimes_v \pi_v$ be an irreducible globally generic cuspidal automorphic representation of $G(\A)$ with trivial central character. Denote by $\frak{n} \unlhd \o$ the paramodular conductor of $\pi$ (cf.\,\cite{RS2007}). We assume $\pi$ satisfies the following conditions:
\begin{itemize}
\item $\frak n$ is square-free.
\item For each real place $v$, $\pi_v$ is in one of the following two types:
\item[(DS)] $$\pi_{v}\vert_{\Sp_4(k_v)} = D_{(\lambda_{1,v},\lambda_{2,v})}\oplus D_{(-\lambda_{2,v},-\lambda_{1,v})},$$
where $D_{(\lambda_{1,v},\lambda_{2,v})}$ is the (limit of) discrete series representation of $\Sp_4(k_v)$ with Blattner parameter $(\lambda_{1,v},\lambda_{2,v}) \in \Z^2$ such that $1-\lambda_{1,v} \leq \lambda_{2,v} \leq 0$. 
\item[(PS)] $$\pi_v \vert_{\Sp_4(k_v)} = {\rm Ind}_{\Sp_4(k_v)\cap {\bf B}(k_v)}^{\Sp_4(k_v)}(|\mbox{ }|_v^{\lambda_{1,v}}\boxtimes|\mbox{ }|_v^{\lambda_{2,v}})$$ for some $\lambda_{1,v},\lambda_{2,v} \in \C$. %with  $|{\rm Re}(\lambda_{1,v})| \geq |{\rm Re}(\lambda_{2,v})|$.
\end{itemize}
In Case (DS), we follow \cite{Moriyama2004} for the choice of the Cartan subalgebra in $\frak{sp}_4(\R)$ and the positive systems.

Denote by $S(\frak{n})$ the set of finite places dividing $\frak{n}$, by $S({\rm DS})$ and $S({\rm PS})$ the sets of real places of type (DS) and (PS), respectively.

Let $f=\bigotimes_v f_v \in \pi$ be a non-zero cusp form satisfying the following conditions:
\begin{align}\label{E:K-type2}
\begin{split}
&f_v\mbox{ is }\G(\o_v)\mbox{-invariant for all finite places }v \notin S(\frak{n}).\\
&f_v\mbox{ is ${\rm K}(\varpi_v)$-invariant for all $v \in S(\frak{n})$}.\\
&f_{v}\mbox{ is a lowest weight vector of the minimal }{\rm U}(2)\mbox{-type of }D_{(-\lambda_{2,v},-\lambda_{1,v})}\mbox{ for all $v \in S({\rm DS})$.}\\
&f_v \mbox{ is a $(\Sp_4(k_v)\cap {\rm O}(4))$-invariant vector for all $v \in S({\rm PS})$}.
\end{split}
\end{align}
For $v \in S({\rm DS})$, the vector $f_v$ is proportional to the vector $v_0$ in the notation of \cite[\S 1.2]{Moriyama2004}. By \cite{JS2007}, the multiplicity of $\pi$
in the space of cusp forms on $\G(\AA)$ is one. Therefore, the conditions (\ref{E:K-type2}) characterize $f \in \pi$ up to scalars. 
Let $W$ be the Whittaker function of $f$ with respect to $\psi_U$ defined by
\[W(g)=\int_{U(k)\backslash U(\A)}f(ug)\overline{\psi_U(u)}\,du.\]
Here $du$ is the Tamagawa measure on $U(\A)$.
We may decompose $W = \prod_v W_v$ as a product of local Whittaker functions of $\pi_v$ with respect to $\psi_{U,v}$. We normalize $f$ as follows:
\begin{align}\label{E:normalization}
\begin{split}
&W_v(\diag( \varpi_v^{-\frak{c}_v}, 1, \varpi_v^{2\frak{c}_v}, \varpi_v^{\frak{c}_v} ))=1\mbox{ for all finite places }v,\\
&W_v(1)\mbox{ is normalized as in (\ref{E:normalization DS}) for all $v \in S({\rm DS})$,}\\
&W_v(1)\mbox{ is normalized as in (\ref{E:normalization PS}) for all $v \in S({\rm PS})$.}
\end{split}
\end{align}
Let $\<f,f\>$ be the Petersson norm of $f$ defined by
\[\<f,f\>=\int_{Z_G(\A)G(k)\backslash G(\A)}|f(g)|^2\,dg.\]

Our main result is the following theorem.

\begin{thm}\label{T:main thm}
We have
\begin{align}\label{E:main thm}
\<f,f\>=2^c\cdot\frac{L(1,\pi,{\rm Ad})}{\Delta_{{\rm PGSp}_4}}\cdot \prod_vC(\pi_v).
\end{align}
Here %${\rm Ad}$ is the adjoint representation of $\GSp_4(\C)$ on $\frak{pgsp}_4(\C)$,
\begin{align*}
\Delta_{\PGSp_4} &= \xi(2)\xi(4),\\
c&=\begin{cases}1 & \mbox{ if $\pi$ is stable},\\
2 & \mbox{ if $\pi$ is endoscopic},\end{cases}\\
C(\pi_v) &= \begin{cases}
\displaystyle{q_v^{-5\c_v}} & \mbox{ if $v \nmid \infty \frak{n}$},\\
\displaystyle{q_v^{-1-5\c_v}\zeta_v(2)^{-1}\zeta_v(4)} & \mbox{ if $v \in S(\frak{n})$},\\
\displaystyle{2^{\lambda_{1,v}-\lambda_{2,v}+5}\pi^{3\lambda_{1,v}-\lambda_{2,v}+5}(1+\lambda_{1,v}-\lambda_{2,v})^{-1}}  & \mbox{ if $v \in S({\rm DS})$},\\
\displaystyle{2^{-4}} & \mbox{ if $v \in S({\rm PS})$}.
\end{cases}
\end{align*}
\end{thm}

\begin{rmk}
For any irreducible globally generic cuspidal automorphic representation $\pi$ of $G(\A)$, we have a formula for the Petersson norm as in (\ref{E:main thm}) (cf.\,Proposition \ref{P:main identity1} below). 
But our assumption on $\pi$ will be used to determine the constant $C(\pi_v)$ explicitly.
\end{rmk}

\section{Siegel-Weil formula}\label{s:sw}
In this section, we introduce Eisenstein series on $\Sp_{2n}(\A)$, theta integrals and its regularization, and recall the Siegel-Weil formula in the second term range for $n=4$. The results and notation in this section will be used in \S\,\ref{s:autom}.

\subsection{Eisenstein series} \label{ss:eisenstein}

Let $n$, $r$ be positive integers such that $n \ge r$.
Let
\[
 P_{n,r} =
 \left\{ \left. 
 \begin{pmatrix}
  a & *  & * & * \\
  0 & a' & * & b' \\
  0 & 0  & \nu(g)\,{}^t  \! a^{-1} & 0 \\
  0 & c' & * & d' 
 \end{pmatrix}
 \in \GSp_{2n} \, \right| \, 
  a \in \GL_r, \,
  g = \begin{pmatrix}
        a' & b' \\
        c' & d' 
       \end{pmatrix}
  \in \GSp_{2n - 2r}
 \right\}
\]
be a maximal parabolic subgroup of $\GSp_{2n}$. 
We define a maximal compact subgroup ${\bf K} = \prod_v {\bf K}_v$ of $\GSp_{2n}(\AA)$ by
\[
 {\bf K}_v =
 \begin{cases}
  \GSp_{2n}(\frak{o}_v)         & \text{if $v$ is finite,} \\
  \GSp_{2n}(k_v) \cap \O(2n) & \text{if $v$ is real.}
 \end{cases}
\]
Let $K_v = {\bf K}_v \cap \Sp_{2n}(k_v)$ for each place $v$ of $k$, and $K = \prod_vK_v$. Let $\frkg_{\infty}$ be the complexified Lie algebra of $ \prod_{v \mid \infty}\Sp_{2n}(k_v)$ and $K_{\infty} = \prod_{v \mid \infty}K_v$ a maximal compact subgroup of $\prod_{v \mid \infty}\Sp_{2n}(k_v)$.

For $s \in \C$, let $I_{n,r}(s) = \Ind^{\GSp_{2n}(\AA)}_{P_{n,r}(\AA)} (\delta_{P_{n,r}}^{s/(2n-r+1)})$
denote the degenerate principal series representation of $\GSp_{2n}(\AA)$.
Here $\delta_{P_{n,r}}$ is the modulus character of $P_{n,r}(\A)$.
For a holomorphic section $F$ of $I_{n,r}(s)$,
define an Eisenstein series $E^{(n,r)}(s, F)$ by 
\[
 E^{(n,r)}(g; s,F) = \sum_{\gamma \in P_{n,r}(k) \backslash \GSp_{2n}(k)} F(\gamma g, s)
\]
for $\Re(s) \gg 0$,
and by the meromorphic continuation otherwise.
Let 
\[
 E^{(n,r)}(s,F) = \sum_{d \gg - \infty} (s - s_0)^d E^{(n,r)}_d(s_0, F)
\]
be the Laurent expansion of $E^{(n,r)}(s, F)$ at $s = s_0$. %When $F$ is ${\bf K}$-invariant and $F(1,s)=1$, we write $E^{(n,r)}(s) = E^{(n,r)}(s,F)$.}

\subsection{Theta integrals}\label{SS:Theta integrals}

In this section,
we review the result of Kudla and Rallis \cite{KR1994}
on the regularization of theta integrals.

Assume that $r \geq 2$.
Let $V_{r,r} = k^{2r}$ be the space of column vectors
equipped with a non-degenerate symmetric bilinear form $(\ ,\ )$ given by
\[
 (x, y) = {}^t \! x 
 \begin{pmatrix}
  0    & \1_r \\
  \1_r & 0
 \end{pmatrix}
 y
\]
for $x,y \in V_{r,r}$.
Let $G' = {\rm GO}_{r,r}$ denote the orthogonal similitude group of $V_{r,r}$ defined by
\[
 {\rm GO}_{r,r} = \left\{ g' \in \GL_{2r} \, \left| \, {}^t \! g'
 \begin{pmatrix}
  0    & \1_r \\
  \1_r & 0
 \end{pmatrix}
 g' =\nu(g') 
 \begin{pmatrix}
  0    & \1_r \\
  \1_r & 0
 \end{pmatrix},\,\nu(g') \in {\mathbb G}_m
 \right. \right\}.
\]
Let $G_1' = {\rm O}_{r,r}$. We define a maximal compact subgroup ${\bf K}' = \prod_v {\bf K}'_v$ of $G'(\AA)$ by
\[
 {\bf K}'_v =
 \begin{cases}
  \bp {\bf 1}_r & 0 \\ 0 & \varpi_v^{\frak{c}_v} {\bf 1}_r\ep G'(\frak{o}_v)\bp {\bf 1}_r & 0 \\ 0 & \varpi_v^{-\frak{c}_v} {\bf 1}_r\ep & \text{if $v$ is finite,} \\
  G'(k_v) \cap \O(2r) & \text{if $v$ is real.}
 \end{cases}
\]
Let $K_v' = {\bf K}_v'\cap G_1'(k_v)$ for each place $v$ of $k$, and $K' = \prod_vK_v'$. Let $r' \leq r$ be a non-negative integer. Let
\[
 P'_{r'} =
 \left\{ \left. 
 \begin{pmatrix}
  a & *  & * & * \\
  0 & a' & * & b' \\
  0 & 0  & \nu(g)\,{}^t \! a^{-1} & 0 \\
  0 & c' & * & d' 
 \end{pmatrix}
 \in G' \, \right| \,  
  a \in \GL_{r-r'},\,
  g=\begin{pmatrix}
   a' & b' \\
   c' & d' 
  \end{pmatrix}
  \in {\rm GO}_{r',r'}
 \right\}
\]
be a maximal parabolic subgroup of $G'$. We denote by $M'_{r'}$ and $N'_{r'}$ the standard Levi subgroup and unipotent radical of $P'_{r'}$, respectively.  
Let $d g'$ be the Haar measure on $G'_1(\AA)$ such that
$\vol(G'_1(k) \backslash G'_1(\AA)) = 1$,
$dm'$ (resp.~$dn'$) the Tamagawa measure on
$M'_{r'}(\A)$ (resp.~$N'_{r'}(\AA)$),
and $d k'$ the Haar measure on $K'$ such that $\vol(K') = 1$.
Define a constant $\kappa_{r,r'}$ by
\[
 \int_{G'_1(\AA)} \phi(g') \, dg'
 = \kappa_{r,r'} \int_{M'_{r'}(\A) \times N'_{r'}(\AA) \times K'} 
 \phi \left( 
 m'
 n' k' \right) \, dm' \, dn' \, dk'
\]
for $\phi \in L^1(G'_1(\AA))$.
%By \cite[\S 9]{Ikeda1996}, we have 
%\[
% \kappa_{r,0} = 
% \frac{\rho}{\xi(r)} \prod_{i=1}^{\left[ \frac{r-1}{2} \right]}
% \frac{\xi(2i+1)}{\xi(2r-2i)}.
%\]

Assume $n \geq r$. Let $\omega = \omega_{\psi, V_{r,r}, n}$ denote the Weil representation
of $\Sp_{2n}(\AA) \times G'_1(\AA)$
on $\calS(V_{r,r}^n(\AA)) = \calS({\rm M}_{2r , n}(\AA))$
with respect to $\psi$. We recall the explicit formula for $\omega$ as follows:
\begin{itemize}
\item for $g' \in G_1'(\A)$, we have
\[
\omega(1,g')\varphi(x) = \varphi(g'^{-1}x);
\]
\item for $a \in \GL_n(\A)$, we have
\[
\omega\left( \bp a & 0 \\ 0 & {}^t\!a^{-1}\ep,1\right)\varphi(x) = |\det(a)|_\A^{r}\varphi(xa);
\]
\item for $b \in {\rm M}_{n,n}(\A)$ with $b={}^tb$, we have
\[
\omega\left( \bp {\bf 1}_n & b \\ 0 & {\bf 1}_n\ep,1\right)\varphi(x) = \psi({\rm tr}(\tfrac{1}{2}b(x,x)))\varphi(x);
\]
\item we have
\[
\omega\left(\bp 0 & -{\bf 1}_n \\ {\bf 1}_n & 0\ep,1\right)\varphi(x) = \int_{V_{r,r}^n(\A)}\psi(-{\rm tr}(x,y))\varphi(y)\,dy,
\]
where $dy$ is the Tamagawa measure on $V_{r,r}^n(\A)$.
\end{itemize}
Let 
\[{\rm G}(\Sp_{2n} \times G'_1) = \left\{ (g,g') \in \GSp_{2n}\times G' \mbox{ }\vert\mbox{ }\nu(g)=\nu(g') \right\}.\]
We extend $\omega$ to a representation of ${\rm G}(\Sp_{2n} \times G'_1)(\A)$ as follows:
\begin{align}\label{Weil rep on R}
\omega(g,g')\varphi = \omega\left (g \bp {\bf 1}_n & 0 \\ 0 & \nu(g')^{-1}{\bf 1}_n\ep,1 \right )L(g')\varphi
\end{align}
for $(g,g') \in {\rm G}(\Sp_{2n} \times G'_1)(\A)$ and $\varphi \in \mathcal{S}(V_{r,r}^n(\A))$. Here
\[
L(g')\varphi(x)=\vert \nu(g')  \vert_\A^{-nr/2}\varphi(g'^{-1}x).
\]
Let $S(V_{r,r}^n(\AA))$ be the subspace of $\calS(V_{r,r}^n(\AA))$
consisting of functions which correspond to polynomials in the Fock model
at the archimedean places.
For $(g,g') \in {\rm G}(\Sp_{2n} \times G'_1)(\A)$ and $\varphi \in S(V_{r,r}^n(\AA))$, let
\[
 \Theta(g, g'; \varphi) = \sum_{x \in V_{r,r}^n(k)} \omega(g, g') \varphi(x).
\]
Fix a real place $v$ of $k$ and 
let $z = z_{r-1, n} \in \frkz(\frkg_v)$ be
the regularizing differential operator as in \cite[Corollary 5.1.2]{KR1994}, where $\frak{z}(\frak{g}_v)$ is the center of the universal enveloping algebra of $\frak{g}_v$.
By  \cite[Proposition 5.3.1]{KR1994},
the function $g' \mapsto \Theta(g, g'; z \cdot \varphi)$
on $G_1'(k) \backslash G_1'(\AA)$ is rapidly decreasing.
Here $\frkz(\frkg_v)$ acts on $S(V_{r,r}^n(\AA))$ via 
the differential of $\omega$.

Put $s'_0 = ({r-1})/{2}$.
Let $F$ be the ${\bf K}'$-invariant holomorphic section of
$\Ind^{G'(\AA)}_{P'_0(\AA)}(\delta_{P_0'}^{s/(r-1)})$ such that $F(1,s) = 1$. Here $\delta_{P_0'}$ is the modulus character of $P_0'(\A)$.
Define an auxiliary Eisenstein series $E(s)$ by
\[
 E(g'; s) = \sum_{\gamma' \in P'_0(k) \backslash G'(k)} F(\gamma' g', s)
\]
for $\Re(s) \gg 0$,
and by the meromorphic continuation otherwise.
Note that $E(s)$ has a simple pole at $s = s'_0$ with constant residue $\kappa_{r,0}$.
Following \cite[\S 5.5]{KR1994},
we consider the integral
\[
 I^{(n,r)}(g; s, \varphi) = \kappa_{r,0}^{-1}Q_{n,r}(s)^{-1} 
 \int_{G_1'(k) \backslash G_1'(\AA)}
 \Theta(g, g'h; z \cdot \varphi) E(g'h; s) \, dg',
\]
where $\nu(h)=\nu(g)$ and
\[
 Q_{n,r}(s) = \prod_{i=0}^{r-1} ( (s - s_0' + i)^2 - (n+1-r)^2 ).
\]
Let
\[
 I^{(n,r)}(s, \varphi) =
 \sum_{d \gg - \infty} (s - s'_0)^d I^{(n,r)}_d(\varphi)
\]
be the Laurent expansion of $I^{(n,r)}(s, \varphi)$ at $s = s'_0$.
Note that $I^{(n,r)}(s, \varphi)$ has at most a simple (resp.~double) pole
at $s = s'_0$ if $r \le ({n+1})/{2}$ (resp.~$({n+1})/{2} < r \le n$).

Let $\hat{\omega}$ be the Weil representation of ${\rm G}(\Sp_{2n} \times G_1')(\A)$
on $\calS({\rm M}_{r , 2n}(\AA))$ defined via the partial Fourier transform
\[
 \calS({\rm M}_{2r , n}(\AA))
  \longrightarrow \calS({\rm M}_{r , 2n}(\AA)), \quad
 \varphi \longmapsto \hat{\varphi},
\]
where
\[
 \hat{\varphi}(u, v) = \int_{{\rm M}_{r , n}(\AA)} \varphi
 \begin{pmatrix}
  x \\
  u
 \end{pmatrix}
 \psi( \tr(v {}^t \! x)) \, dx
\]
for $u, v \in {\rm M}_{r , n}(\AA)$. Here $dx$ is the Tamagawa measure on ${\rm M}_{r, n}(\A)$. 
For $\Re(s) > s_0'-n + r$,
define a $(\frkg_{\infty}, K_{\infty}) \times \GSp_{2n}(\AA_f)$-intertwining map
\[
 \calF: S(V_{r,r}^n(\AA)) \longrightarrow I_{n,r}(s)
\]
by
\[
 \calF(\varphi)(g, s) = \int_{\GL_r(\AA)} \int_{K'} \hat{\omega}(g, k'g')
 \hat{\varphi} (0_{r \times n}, {}^t \! a, 0_{r \times (n-r)}) F(k'g',s)
 |\det (a)|_{\A}^{s-s_0'+n} \,dk' \, da,
\]
where $\nu(g')=\nu(g)$.
Then, by \cite[\S 5.5]{KR1994} and \cite[\S 7.4]{GI2011},
\[
 I^{(n,r)}(s, \varphi) = E^{(n,r)}(s,\calF(\varphi)).
\]

 We define the spherical Schwartz function $\varphi^o=\bigotimes_v\varphi_v^{o} \in S(V_{r,r}^n(\AA))$ as follows:
\begin{itemize}
\item If $v$ is finite, then 
\begin{align*}
\varphi_v^o\bp x \\ y \ep = q_v^{-\frak{c}_vnr/2}\cdot \mathbb{I}_{{\rm M}_{r, n}(\varpi_v^{-\frak{c}_v}\frak{o}_v)}(x)\mathbb{I}_{{\rm M}_{r, n}(\frak{o}_v)}(y).
\end{align*}
\item If $v$ is real, then
\begin{align*}
\varphi_v^o(x) = e^{-\pi\, {\rm tr}(x{}^t \! x)}.
\end{align*}
\end{itemize}
Note that
\begin{align}\label{E:spherical property}
\omega_v(k,k')\varphi_v^o = \varphi_v^o 
\end{align}
for $(k,k') \in ({\bf K}_v \times {\bf K}_v')\cap {\rm G}(\Sp_{2n}\times G_1')(k_v)$. 
%Let $S_{n,r}^o$ be the $(\frak{g}_{\infty},K_{\infty})\times \Sp_{2n}(\A_f)$-submodule of $S_{n,r}$ generated by $\varphi^o$. 
For $\Re(s) > s_0'-n + r-1$,
define an intertwining map
\[
 \calF_v: S(V_{r,r}^n(k_v)) \longrightarrow I_{n,r,v}(s)
\]
by
\[
 \calF_v(\varphi)(g, s) = \int_{\GL_r(k_v)} \int_{K_v'}\hat{\omega}_v(g, k_v'g')
 \hat{\varphi} (0_{r \times n}, {}^t \! a_v, 0_{r \times (n-r)}) F_v(k_v'g',s)
 |\det (a_v)|_{v}^{s-s_0'+n} \,dk_v'\, da_v,
\]
where $\nu(g')=\nu(g)$ and $F_v$ is the ${\bf K}_v'$-invariant holomorphic section of
$\Ind^{G'(k_v)}_{P'_0(k_v)}(\delta_{P_0'}^{s/(r-1)})$ such that $F_v(1,s) = 1$.
The Haar measure $da_v$ on $\GL_r(k_v)$ is normalized as in (\ref{E:standard measure on GL}). By a direct calculation, 
\begin{align}\label{E:spherical section}
\calF_v(\varphi_v^o)(1,s) = \prod_{j=0}^{r-1}\zeta_v(s-s_0'+n-j).
\end{align}
By (\ref{Tamagawa measure on GL}), for $\varphi = \bigotimes_v\varphi_v \in S(V_{r,r}^n(\AA))$, we have
\begin{align}\label{E:local global ratio}
\calF(\varphi) = \frak{D}^{-r^2/2}\rho^{-1}\prod_{i=2}^r\xi(i)^{-1} \cdot\bigotimes_v \calF_v(\varphi_v).
\end{align}

\subsection{First and second term identities}\label{SS:First and second term identities}

In this section,
we recall certain identities relating the terms in the Laurent expansions
of regularized theta integrals to those of Siegel Eisenstein series,
which were proved by Gan, Qiu, and Takeda in \cite{GQT2014}.

Let $n=4$.
Let $\calR$ be the space of automorphic forms on $\GSp_8(\A)$ spanned by residues
$E^{(4,4)}_{-1}({1}/{2}, F)$
for all holomorphic sections $F$ of $I_{4,4}(s)$.
Define a $(\frkg_{\infty}, K_{\infty}) \times \GSp_8(\AA_f)$-intertwining map
\[
 S(V_{3,3}^4(\AA))  \longrightarrow I_{4,4} \left(\frac{1}{2}\right), \quad
 \varphi  \longmapsto F(\varphi)
\]
by 
\[
 F({\varphi}) \left(g, \frac{1}{2}\right) = \omega(g, g') \varphi(0),
\]
where $\nu(g')=\nu(g)$. We extend $F({\varphi})$ to the holomorphic section $F({\varphi})$ of $I_{4,4}(s)$
such that its restriction to ${\bf K}$ is independent of $s$.

As in \cite[\S 5]{Ikeda1996},
define a $(\frkg_{\infty}, K_{\infty}) \times \Sp_8(\AA_f)$-intertwining map
\[
 \pr: S(V_{3,3}^4(\AA)) \longrightarrow S(V_{2,2}^4(\AA))
\]
by
\[
 \pr(\varphi)
 \begin{pmatrix}
  x \\
  y
 \end{pmatrix}
 = \int_{\AA^4} \varphi_{K'}
 \begin{pmatrix}
  u \\
  x \\
  0 \\
  y
 \end{pmatrix}
 \, du
\]
for $x, y \in {\rm M}_{2 , 4}(\AA)$, where $du$ is the Tamagawa measure on $\A^4$. Then, as a special case of the Siegel-Weil formula proved by Gan, Qiu, and Takeda \cite[Theorem 1.1]{GQT2014}, we have
\begin{align}
 &E^{(4,4)}_{-1} \left(\frac{1}{2}, F({\varphi})\right) = I^{(4,3)}_{-2}(\varphi), \label{E:1st identity}\\
 &E^{(4,4)}_0 \left(\frac{1}{2}, F({\varphi})\right) = I^{(4,3)}_{-1}(\varphi)
 - \kappa_{3,2} \cdot I^{(4,2)}_0(\pr(\varphi))
  \mbox{ \rm mod } \calR \label{E:2nd identity}
\end{align}
for all $\varphi \in S(V_{3,3}^4(\AA))$.

\section{Automorphic $L$-functions}
\label{s:autom}
In this section, we introduce integral representations of two automorphic $L$-functions, namely the standard $L$-function for $\GSp_4$ and the Rankin-Selberg $L$-function for $\GSp_4 \times \GSp_4$. We keep the notation of \S\,\ref{s:sw} with $n=4$. 

\subsection{Preliminaries}  
%Let $k$ be a totally real number field
%and $\AA = \AA_k$ the ring of adeles of $k$.
%If $v$ is a finite place of $k$,
%let $\frko_v$, $\varpi_v$, $q_v$ be
%the maximal compact subring of $k_v$,
%a generator of the maximal ideal of $\frko_v$,
%and the cardinality of $\frko_v / \varpi_v \frko_v$,
%respectively.
%Fix a non-trivial additive character $\psi = \otimes_v \psi_v$ of $\AA / k$.
%Let $\xi(s) = |D|^{\frac{s}{2}} \prod_{v \le \infty} \zeta_v (s)$
%denote the zeta function of $k$,
%where $D$ is the discriminant of $k$.
%It satisfies the functional equation $\xi(1-s) = \xi(s)$
%and has the Laurent expansion of the form
%\[
% \xi(s) = \frac{\rho}{s-1} + \gamma + O(s-1).
%\]

Recall that $G=\GSp_4$. Let $H = \GSp_8$ and
\[
 {\bf G} = \{ (g_1, g_2) \in G \times G \, | \, \nu(g_1) = \nu(g_2) \}.
\]
Let $Z_{H}$ be the center of $H$.
We identify $\bf G$ with its image under the embedding
\[
 {\bf G}  \longrightarrow H, \quad
 \left(
 \begin{pmatrix}
  a_1 & b_1 \\
  c_1 & d_1
 \end{pmatrix},
 \begin{pmatrix}
  a_2 & b_2 \\
  c_2 & d_2
 \end{pmatrix}
 \right)  \longmapsto
 \begin{pmatrix}
  a_1 & 0    & b_1 & 0 \\
  0   & a_2  & 0   & -b_2 \\
  c_1 & 0    & d_1 & 0 \\
  0   & -c_2 & 0   & d_2 
 \end{pmatrix}.
\]

Let $\pi$ be an irreducible globally generic cuspidal automorphic representation of $G(\A)$ with trivial central character and paramodular conductor $\frak{n}$ satisfying the conditions in \S\,\ref{SS:main thm}. %By \cite[Theorem 12.1]{GT2011} (see also \cite{CKPS2004} and \cite{AH2006}),
%$\pi$ has a strong functorial lift $\Pi$ to $\GL_4(\A)$.
%By \cite{GRS2001},
%either $\Pi$ is cuspidal or $\Pi = \sigma_1 \boxplus \sigma_2$ for some irreducible cuspidal automorphic representations $\sigma_1$ and $\sigma_2$
%of $\GL_2(\AA)$ with trivial central character such that $\sigma_1 \neq \sigma_2$.
%We say that $\pi$ is stable (resp.~endoscopic)
%if $\Pi$ is cuspidal (resp.~non-cuspidal).
%Let $S$ be a finite set of places of $k$
%including the archimedean places such that,
%for $v \notin S$, $\pi_v$ is unramified.
%Let $\rm std$ denote the composition of the projection $\GSp_4(\C) \rightarrow \PGSp_4(\C)$ with the standard
%representation of $\PGSp_4(\C) \simeq {\rm SO}_5(\C)$ on $\C^5$, and ${\rm Ad}$ the adjoint representation of $\GSp_4(\C)$ on $\frak{pgsp}_4(\C)$.
%Then
%\begin{align*}
% L(s, \Pi, \wedge^2) & = \zeta(s) L(s, \pi, {\rm std}), \\
% L(s, \Pi, \Sym^2) & = L(s, \pi, \Ad). 
%\end{align*}
%By \cite{GRS2001},
%$L(s, \Pi, \wedge^2)$ has a simple (resp.~double) pole at $s=1$
%if $\pi$ is stable (resp.~endoscopic).
%Hence $L(s, \pi, {\rm std})$ is holomorphic and non-zero
%(resp.~has a simple pole) at $s=1$ if $\pi$ is stable (resp.~endoscopic).
%Moreover, $L(s, \pi, \Ad)$ is holomorphic and non-zero at $s=1$.
When $v \nmid \infty\frak{n}$ or $v\in S({\rm PS})$, there exist $\lambda_{1,v},\lambda_{2,v} \in \C$ %with $|{\rm Re}(\lambda_{1,v})| \geq |{\rm Re}(\lambda_{2,v})|$
such that
\begin{align}\label{E:unram. para.}
\pi_v \vert_{\Sp_4(k_v)} = {\rm Ind}_{\Sp_4(k_v)\cap {\bf B}(k_v)}^{\Sp_4(k_v)}(|\mbox{ }|_v^{\lambda_{1,v}}\boxtimes|\mbox{ }|_v^{\lambda_{2,v}}).
\end{align}
When $v \mid \frak{n}$, there exist $\varepsilon_v \in \{0,1\}$ and $\lambda_v \in \C$ such that
\begin{align}\label{E:IIa para.}
\pi_v = {\rm Ind}_{P_{2,2}(k_v)}^{G(k_v)}\left(({\rm St}_v\otimes |\mbox{ }|_v^{\lambda_v})\boxtimes \eta_v^{\varepsilon_v} |\mbox{ }|_v^{-\lambda_v}\right),
\end{align}
where ${\rm St}_v$ is the Steinberg representation of $\GL_2(k_v)$ and $\eta_v$ is the non-trivial unramified quadratic character of $k_v^\times$. Indeed, by the results in \cite{RS2007}, any irreducible generic admissible representation of $G(k_v)$ with trivial central character and paramodular conductor $\varpi_v\o_v$ is of this form.
Since $\pi_v$ is unitary and generic for all $v$, by the unitarizability criterion in \cite[Theorem 1.1]{LMT2004}, we have
\begin{align}\label{eq:satake}
\begin{cases}
|{\rm Re}(\lambda_{1,v})|+|{\rm Re}(\lambda_{2,v})|<1 & \mbox{ if $v \nmid \infty\frak{n}$},\\
|{\rm Re}(\lambda_{v})|<1/2 & \mbox{ if $v \mid \frak{n}$},\\
|{\rm Re}(\lambda_{1,v})|+|{\rm Re}(\lambda_{2,v})|<1 & \mbox{ if $v \in S({\rm PS})$}.
\end{cases}
\end{align}
%Without lose of generality, we assume
%\[
%\begin{cases}
%{\rm Re}(\lambda) \leq 0 & \mbox{ if $v \mid \frak{n}$},\\
%\min\{{\rm Re}(\lambda_{1,v}),\,{\rm Re}(\lambda_{2,v})\} \leq 0 & \mbox{ if $v \in S({\rm PS})$}.
%\end{cases}
%\]

\subsection{Standard $L$-functions}

In this section,
we review the doubling method of Piatetski-Shapiro and Rallis
\cite{LNM1254}, \cite[\S 6.2]{Harris1993}.

Let $P=P_{4,4}$
be the standard Siegel parabolic subgroup of $H$, and put \[d_P(s) = \xi\left(s + \frac{5}{2}\right) \xi(2s+1) \xi(2s+3).\] Let $I(s) = I_{4,4}(s)$
denote a degenerate principal series representation of $H(\AA)$.
For a holomorphic section $F$ of $I(s)$,
we write $E(s, F)=E^{(4,4)}(s,F).$ 
By \cite[Theorem 1.1]{KR1994},
the Eisenstein series $E(s,F)$ has at most a simple pole at $s = {1}/{2}$.

For $f \in \pi$,
let
\[
 Z(s, f, F) = 
 \int_{Z_H(\AA) {\bf G}(k) \backslash {\bf G}(\AA)}
 E( g; s, F) (f\otimes\bar{f})(g) \, dg.
\]
We assume that $f = \bigotimes_v f_v$ and $F = \bigotimes_v F_v$.
Choose local Hermitian pairings $\langle \ , \ \rangle_v$ on $\pi_v \times \pi_v$ such that
$\langle f, f \rangle = \prod_v \langle f_v, f_v \rangle_v$,
and define a matrix coefficient $\phi_v$ of $\pi_v$ by
\[
 \phi_v(g) = \frac{\langle \pi_v (g) f_v, f_v \rangle_v}{\langle f_v, f_v \rangle_v}.
\]
Note that $\phi_v$ does not depend on the choice of 
$\langle \ , \ \rangle_v$.
Define a local zeta integral $Z_v(s, \phi_v, F_v)$ by
\begin{align}\label{E:local zeta1}
 Z_v(s, \phi_v, F_v)
 = \int_{\Sp_4(k_v)} F_v(\delta (g_v,1), s) \phi_v(g_v) \, dg_v^{\rm Tam},
\end{align}
where
\[
 \delta =
 \begin{pmatrix}
  0 & 0 & -\frac{1}{2} \1_2 & \frac{1}{2} \1_2 \\
  \frac{1}{2} \1_2 & \frac{1}{2} \1_2 & 0 & 0 \\
  \1_2 & - \1_2 & 0 & 0 \\
  0 & 0 & \1_2 & \1_2
 \end{pmatrix}.
\]
Here $dg_v^{\rm Tam}$ is the local Tamagawa measure on $\Sp_4(k_v)$. Then, by \cite{LNM1254}, \cite[\S 6.2]{Harris1993},
\[
 Z(s,f, F) = \langle f, f \rangle \prod_{v} Z_v(s, \phi_v, F_v)
\]
for $\Re(s) \gg 0$.

\begin{lemma}\label{L:unram PSR} 
Let $v$ be a finite place of $k$ satisfying the following conditions: 
\begin{itemize}
\item $\pi_v$ is unramified.
\item $f_v$ is $G(\o_v)$-invariant.
\item $F_v$ is $H(\o_v)$-invariant.
\end{itemize}
We have
\[Z_v(s,\phi_v,F_v) = F_v(1,s)q_v^{-5\c_v}\zeta_v(2)^{-1}\zeta_v(4)^{-1}d_{P,v}(s)^{-1}L\left(s+\frac{1}{2},\pi_v,{\rm std}\right).\]
\end{lemma}

\begin{proof}
Let $dg_v^{\rm std}$ be the Haar measure on $\Sp_4(k_v)$ normalized so that
\[{\rm vol}(\Sp_4(\o_v),dg_v^{\rm std}) = 1. \]
By \cite[Proposition 6.2]{LNM1254}, 
\[
 \int_{\Sp_4(k_v)} F_v(\delta (g_v,1), s) \phi_v(g_v) \, dg_v^{\rm std} = F_v(1,s)d_{P,v}(s)^{-1}L\left(s+\frac{1}{2},\pi_v,{\rm std}\right).
\]
Since
\[
{\rm vol}(\Sp_4(\o_v),dg_v^{\rm Tam}) = q_v^{-5\c_v}\cdot q_v^{-10}\cdot \left| \Sp_4(\mathbb{F}_{q_v})\right| = q_v^{-5\c_v}\zeta_v(2)^{-1}\zeta_v(4)^{-1},
\]
it follows that
\[dg_v^{\rm Tam} = q_v^{-5\c_v}\zeta_v(2)^{-1}\zeta_v(4)^{-1}\cdot dg_v^{\rm std}.\]
This completes the proof.
\end{proof}

\begin{lemma}\label{L:ab doubling}
Let $v$ be a place of $k$. 
The integral $Z_v(s, \phi_v, F_v)$ is absolutely convergent for 
\begin{align*}
\begin{cases}
\Re(s)> -1/2 + \max\{|{\rm Re}(\lambda_{1,v})|, |{\rm Re}(\lambda_{2,v})|\} & \mbox{ if $v \nmid \infty\frak{n}$},\\
\Re(s)> -1/2 + \max\{0,|{\rm Re}(\lambda_v)|-1/2\} & \mbox{ if $v \in S(\frak{n})$},\\
\Re(s) > -1/2 & \mbox{ if $v \in S({\rm DS})$},\\ 
\Re(s)> -1/2 + \max\{|{\rm Re}(\lambda_{1,v})|, |{\rm Re}(\lambda_{2,v})|\} & \mbox{ if $v \in S({\rm PS})$}.
\end{cases} 
\end{align*}
In particular, the integral is absolutely convergent for ${\rm Re}(s) \geq 1/2$.
\end{lemma}

\begin{proof}
The assertions will be proved in Lemmas \ref{L:local zeta Ur}-\ref{L:ab1} below. %(see also Remark \ref{R:ab1}). 
\end{proof}

\subsection{$L$-functions for $\GSp_4 \times \GSp_4$}

In this section,
we review Jiang's integral representation
of the $L$-function for $G \times G$ \cite{Jiang1996}.

Let $\mathcal{P}=P_{4,3}$ be a maximal parabolic subgroup of $H$, and put \[d_{\mathcal{P}}(s) = \xi(s+1) \xi(s+2) \xi(s+3) \xi(2s+2).\] Let $\calI(s)=I_{4,3}(s)$
denote a degenerate principal series representation of $H(\AA)$.
For a holomorphic section $\calF$ of $\calI(s)$,
we write $\calE(s, \calF) = E^{(4,3)}(s,\calF)$. 
By \cite[Chapter 3, Theorem 4.0.1]{Jiang1996},
the Eisenstein series $\calE(s, \calF)$ has at most a double pole at $s = 1$.

For $f \in \pi$,
let
\[
 \calZ(s, f, \calF) = \int_{Z_H(\AA) {\bf G}(k) \backslash {\bf G}(\AA)}
 \calE( g; s, \calF) (f\otimes \bar{f})(g) \, dg .
\]
We assume that $f = \bigotimes_v f_v$ and $\calF = \bigotimes_v \calF_v$.
Let $W = \bigotimes_v W_v$ be the Whittaker function of $f$.
Define a local zeta integral $\calZ_v(s, W_v, \calF_v)$ by
\begin{align}\label{E:local zeta2}
 \calZ_v(s, W_v, \calF_v) = 
 \int_{Z_H(k_v) \tilde{U}(k_v) \backslash {\bf G}(k_v)}
 \calF_v(\eta g_v, s) (W_v\otimes \overline{W_v})(g_v) \, d\bar{g}_v^{\rm Tam},
\end{align}
where
\[
 \tilde{U} = \left\{ \left. \left( u
  \begin{pmatrix}
   1 & 0 & x & 0 \\
   0 & 1 & 0 & 0 \\
   0 & 0 & 1 & 0 \\
   0 & 0 & 0 & 1
 \end{pmatrix},
 u \right) \, \right| \, u \in U, \, x \in \GG_a \right\}
\]
and 
\[
 \eta =
 \begin{pmatrix}
   1 & 0 & 0  & 0  & 0 & 0 & 0 & 0 \\
   0 & 1 & 0  & 0  & 0 & 0 & 0 & 0 \\
   0 & 0 & 0  & 0  & 0 & 0 & 0 & 1 \\
   0 & 0 & 0  & 0  & 0 & 0 & 1 & 0 \\
   0 & 0 & 0  & 0  & 1 & 0 & 1 & 0 \\
   0 & 0 & 0  & 0  & 0 & 1 & 0 & 1 \\
   0 & 1 & 0  & -1 & 0 & 0 & 0 & 0 \\
   1 & 0 & -1 & 0  & 0 & 0 & 0 & 0
 \end{pmatrix}.
\]
Here $d\bar{g}_v^{\rm Tam}$ is the quotient measure defined by the local Tamagawa measures on $Z_H(k_v)\backslash {\bf G}(k_v)$ and $\tilde{U}(k_v)$. Then, by \cite{Jiang1996},
\[
 \calZ(s, f, \calF)
 = 
 \prod_{v} \calZ_v(s, W_v, \calF_v)
\]
for $\Re(s) \gg 0$.

\begin{lemma}\label{L:unram Jiang}
Let $v$ be a finite place of $k$ satisfying the following conditions:
\begin{itemize}
 \item $\pi_v$ is unramified.
 \item $W_v$ is $G(\frak{o}_v)$-invariant.
 \item $\calF_v$ is $H(\frak{o}_v)$-invariant.
\end{itemize}
We have
\begin{align*}
\calZ_v(s, W_v,\calF_v) &= \calF_v(1,s)\vert W_v(\diag( \varpi_v^{-\frak{c}_v}, 1, \varpi_v^{2\frak{c}_v}, \varpi_v^{\frak{c}_v} ))\vert^2\\
&\times q_v^{(3s/2-13)\c_v}\zeta_v(2)^{-2}\zeta_v(4)^{-2} d_{{\mathcal P},v}(s)^{-1}L\left(\frac{s+1}{2},\pi_v\times\pi_v^\vee\right).
\end{align*}
\end{lemma}

\begin{proof}
Let 
\[g_0 = \diag( \varpi_v^{\frak{c}_v}, 1, \varpi_v^{-2\frak{c}_v}, \varpi_v^{-\frak{c}_v} ).\]
Let $\psi_v'(x) = \psi_v(\varpi_v^{-\c_v}x)$ be an additive character of $k_v$ of conductor $\o_v$. Let $W_v'$ be the $G(\o_v)$-invariant Whittaker function of $\pi_v$ with respect to $\psi_v'$ such that $W_v'(1)=1$. Then
\[W_v(g) = W_v(g_0^{-1})W_v'(g_0g)\]
for $g \in G(k_v)$. Let $dg_v^{\rm std}$ and $du_v^{\rm std}$ be the Haar measures on $Z_H(k_v)\backslash {\bf G}(k_v)$ and $\tilde{U}(k_v)$ normalized so that
\[{\rm vol}(Z_H(\o_v)\backslash G(\o_v),dg_v^{\rm std})={\rm vol}(\tilde{U}(\o_v),du_v^{\rm std})=1.\]
Let $d\bar{g}_v^{\rm std}$ be the corresponding quotient measure on $Z_H(k_v) \tilde{U}(k_v) \backslash {\bf G}(k_v)$. By \cite[Theorem 3.3.3]{Jiang1996}, 
\[
\int_{Z_H(k_v) \tilde{U}(k_v) \backslash {\bf G}(k_v)}
 \calF_v(\eta g_v, s) (W_v'\otimes \overline{W_v'})(g_v) \, d\bar{g}_v^{\rm std} = \calF_v(1,s)d_{\mathcal{P},v}(s)^{-1}L\left(\frac{s+1}{2},\pi_v\times\pi_v^\vee\right).
\]
 On the other hand, 
\[
\calF_v(\eta(g_0^{-1},g_0^{-1})g,s) = q_v^{(3s/2+9/2)\c_v}\calF_v(\eta g,s)
\]
 for $g \in {\bf G}(k_v).$ Therefore, 
 \begin{align*}
 &\int_{Z_H(k_v) \tilde{U}(k_v) \backslash {\bf G}(k_v)}
 \calF_v(\eta g_v, s) (W_v\otimes \overline{W_v})(g_v) \, d\bar{g}_v^{\rm std}\\
 &=q_v^{-10\c_v}|W_v(g_0^{-1})|^2\int_{Z_H(k_v) \tilde{U}(k_v) \backslash {\bf G}(k_v)}
 \calF_v(\eta (g_0^{-1},g_0^{-1})g_v, s) (W_v'\otimes \overline{W_v'})(g_v) \, d\bar{g}_v^{\rm std}\\
 &=q_v^{(3s/2-11/2)\c_v}|W_v(g_0^{-1})|^2\calF_v(1,s)d_{\mathcal{P},v}(s)^{-1}L\left(\frac{s+1}{2},\pi_v\times\pi_v^\vee\right).
 \end{align*}
 It suffices to show that
 \[d\bar{g}_v^{\rm Tam} = q_v^{-15\c_v/2}\zeta_v(2)^{-2}\zeta_v(4)^{-2}\cdot d\bar{g}_v^{\rm std}.\]
 Let $dg_{v}^{\rm Tam}$, $dg_{1,v}^{\rm Tam}$, and $dg_{2,v}^{\rm Tam}$ be the local Tamagawa measures on $Z_H(k_v)\backslash{\bf G}(k_v)$, $\PGSp_4(k_v)$, and $\Sp_4(k_v)$, respectively. Under the isomorphism
 \begin{align*}
 \PGSp_4(k_v) \ltimes \Sp_4(k_v)  &\longrightarrow Z_H(k_v)\backslash {\bf G}(k_v),\\ (g_1,g_2) &\longmapsto (g_1g_2,g_1),
 \end{align*}
 we have
 \[dg_{1,v}^{\rm Tam} \times dg_{2,v}^{\rm Tam} = dg_{v}^{\rm Tam}.\]
 Let \[U' = \left\{ \left.\begin{pmatrix}
   1 & 0 & x & 0 \\
   0 & 1 & 0 & 0 \\
   0 & 0 & 1 & 0 \\
   0 & 0 & 0 & 1
 \end{pmatrix}\mbox{ } \right\vert\mbox{ }x \in {\mathbb G}_a\right\}\]
 be a unipotent subgroup of $\Sp_4$. Let $du_{v}^{\rm Tam}$, $du_{1,v}^{\rm Tam}$, and $du_{2,v}^{\rm Tam}$ be the local Tamagawa measures on $\tilde{U}(k_v)$, $U(k_v)$, and $U'(k_v)$, respectively. Then
 \[U(k_v)\ltimes U'(k_v) \simeq \tilde{U}(k_v)\]
 under the above isomorphism and
 \[du_{1,v}^{\rm Tam} \times du_{2,v}^{\rm Tam} = du_{v}^{\rm Tam}.\]
 Note that
 \begin{align*}
 {\rm vol}(Z_H(\o_v)\tilde{U}(\o_v)\backslash {\bf G}(\o_v),d\bar{g}_v^{\rm Tam}) &= \frac{{\rm vol}(\PGSp_4(\o_v),dg_{1,v}^{\rm Tam})}{{\rm vol}(U(\o_v),du_{1,v}^{\rm Tam})} \cdot \frac{{\rm vol}(\Sp_4(\o_v),dg_{2,v}^{\rm Tam})}{{\rm vol}(U'(\o_v),du_{2,v}^{\rm Tam})}\\
 &= \frac{q_v^{-5\c_v}\cdot q_v^{-10}\cdot \left| \PGSp_4(\mathbb{F}_{q_v})\right|}{q_v^{-2\c_v}}\cdot \frac{q_v^{-5\c_v}\cdot q_v^{-10}\cdot \left| \Sp_4(\mathbb{F}_{q_v})\right|}{q_v^{-\c_v/2}}\\
 &=q_v^{-15\c_v/2}\zeta_v(2)^{-2}\zeta_v(4)^{-2}.
 \end{align*}
 This completes the proof.
\end{proof}

\begin{lemma}\label{L:ab2}
Let $v$ be a place of k such that $v \mid \infty\frak{n}$. The integral
$\calZ_{v}(s, W_{v}, \calF_v)$
is absolutely convergent for 
\begin{align*}
\begin{cases}
{\rm Re}(s) > -1 + 4|{\rm Re}(\lambda_v)|  & \mbox{ if $v \in S(\frak{n})$},\\ 
{\rm Re}(s) > -1 & \mbox{ if $v \in S({\rm DS})$},\\
{\rm Re}(s) > -1+ 2(|{\rm Re}(\lambda_{1,v})|+|{\rm Re}(\lambda_{2,v})|)  & \mbox{ if $v \in S({\rm PS})$},
\end{cases}
\end{align*}
In particular, the integral is absolutely convergent for ${\rm Re}(s) \geq 1$.
\end{lemma}

\begin{proof}
The assertion will be proved in Lemma \ref{L:uniform2} below.% (see also Remark \ref{R:ab2}).
\end{proof}

\section{Proof of Main Theorem}
\label{s:proof}

\subsection{Crude form of the formula for Petersson norms}
In this section, we keep the notation of \S\,\ref{s:sw} and \S\,\ref{s:autom}. Let $\pi$ be an irreducible globally generic cuspidal automorphic representation of $G(\A)$ with trivial central character and paramodular conductor $\frak{n}$ satisfying the conditions in \S\,\ref{SS:main thm}.
Let $f \in \pi$ be the cusp form satisfying the conditions (\ref{E:K-type2}) and (\ref{E:normalization}). Let $S$ be the set of places of $k$ defined by
\[S= \{v \mid \infty \frak{n}\}.\]
For an automorphic form $\phi$ on $H(\A)=\GSp_8(\AA)$, let
\[
 \langle \phi, \bar{f} \otimes f \rangle 
 = \int_{Z_H(\A){\bf G}(k) \backslash {\bf G}(\AA)}
 \phi( g) (f \otimes\bar{f})(g) \, dg.
\]

\begin{lemma} \label{lem:vanish_res}
Assume that $\pi$ is stable.
For all $\phi \in \calR$,
\[
 \langle \phi, \bar{f} \otimes f \rangle = 0.
\]
\end{lemma}

\begin{proof}
By Lemma \ref{L:unram PSR}, we have
\[
 \langle E(s, F), \bar{f} \otimes f \rangle = \langle f, f \rangle \cdot
 \xi^T(2)^{-1} \xi^T(4)^{-1} \cdot d_P^T(s)^{-1} \cdot 
 L^T \left( s + \frac{1}{2}, \pi, {\rm std} \right)
 \prod_{v \in T} Z_v(s, \phi_v, F_v)
\]
for any holomorphic section $F = \bigotimes_v F_v$ such that $F_v(1,s) = 1$ for all $v \notin T$, where $T$ is a sufficiently large finite set of places of $k$ depending on $F$.
Since $\pi$ is stable and hence $L^T(s, \pi, {\rm std})$ is holomorphic at $s = 1$, this together with Lemma \ref{L:ab doubling} implies that $\langle E(s, F), \bar{f} \otimes f \rangle$ is holomorphic at $s = 1$.
This completes the proof.
\end{proof}

\begin{lemma} \label{lem:vanish_theta}
Assume that $\pi$ is stable.
For all $d \in \ZZ$ and all $\varphi \in S(V_{2,2}^4(\AA))$,
\[
 \langle I^{(4,2)}_d(\varphi), \bar{f} \otimes f \rangle = 0.
\]
\end{lemma}

\begin{proof}
Let $\theta_{r,r}(\pi)$ be the global theta lift of $\pi$ to ${\rm GO}_{r,r}(\AA)$.
By (\ref{E:unram. para.}), \eqref{eq:satake}, and the local theta corre- spondence for unramified representations, we have
$\theta_{1,1}(\pi) = 0$.
If $\theta_{2,2}(\pi) \ne 0$,
then there exist irreducible cuspidal automorphic representations $\sigma_1$ and $\sigma_2$
of $\GL_2(\AA)$ such that
\[
 L^T(s, \pi, {\rm std}) = \xi^T(s) L^T(s, \sigma_1 \times \sigma_2),
\]
where $T$ is a sufficiently large finite set of places of $k$.
This contradicts the holomorphy of $L^T(s, \pi, {\rm std})$ at $s=1$.
Hence 
\begin{equation} \label{eq:vanish_theta}
 \theta_{2,2}(\pi) = 0. 
\end{equation}

Write $G' = {\rm GO}_{2,2}$ and $G'_1 = {\rm O}_{2,2}$. Let $\mathcal{C} = \A^{\times, 2}k^{\times}\backslash \A^{\times}$. The similitude characters induce isomorphisms
\[
Z_G(\A)\Sp_4(\A)G(k)\backslash G(\A) \simeq \mathcal{C},\quad Z_{G'}(\A)G'_1(\A)G'(k)\backslash G'(\A) \simeq \mathcal{C}.
\]
Fix cross-sections $c \mapsto g_c$ and $c \mapsto g_c'$ of $G(\A) \rightarrow \mathcal{C}$ and $G'(\A) \rightarrow \mathcal{C}$, respectively.
For $\varphi \in S(V_{2,2}^4(\AA))$, write $z \cdot \varphi = \sum_j \varphi_{1j} \otimes \overline{\varphi_{2j}} \in S(V_{2,2}^4(\AA))$
with $\varphi_{ij} \in S(V_{2,2}^2(\AA))$, where $z$ is the regularizing differential operator as in \S\,\ref{SS:Theta integrals}.
Then
\[
 \Theta( (g_1g_c, g_2g_c), g'g_c'; z \cdot \varphi)
 = \sum_j \Theta(g_1g_c, g'g_c'; \varphi_{1j}) \overline{\Theta(g_2g_c, g'g_c'; \varphi_{2j})}
\]
for $c \in \mathcal{C}$, $(g_1,g_2) \in \Sp_4(\A)\times \Sp_4(\A)$, and $g' \in G_1'(\A)$. By \eqref{eq:vanish_theta},
the theta lift $\theta(f, \varphi_{ij})$ is identically zero, where
\[
 \theta(g'g_c'; f, \varphi_{ij})
 = \int_{\Sp_4(k) \backslash \Sp_4(\AA)} \Theta(gg_c, g'g_c'; \varphi_{ij}) f(g) \, dg
\]
for $c \in \mathcal{C}$ and $g' \in G_1'(\AA)$.
Hence $\langle I^{(4,2)}(s, \varphi), \bar{f} \otimes f \rangle$ is equal to
\begin{align*}
& \int_{Z_H(\A){\bf G}(k)\backslash {\bf G}(\A)}
 I^{(4,2)}(g;s,\varphi)(f\otimes \bar{f})(g)\,
 dg\\
 &=  \kappa_{2,0}^{-1}Q_{4,2}(s)^{-1} \int_{\mathcal{C}}\int_{(\Sp_4(k) \backslash \Sp_4(\AA))^2}
 \int_{G_1'(k) \backslash G_1'(\AA)}\Theta( (g_1g_c, g_2g_c), g'g_c'; z \cdot \varphi) E(g'; s)
 f(g_1g_c) \overline{f(g_2g_c)} \, dg' \, dg_1 \, dg_2\, dc \\
 & = \kappa_{2,0}^{-1}Q_{4,2}(s)^{-1} \int_{Z_{G'}(\A)G'(k) \backslash G'(\AA)} \sum_j
 \theta(g'; f, \varphi_{1j}) \overline{\theta(g'; f, \varphi_{2j})} E(g'; s)
 \, dg' \\
 & = 0.
\end{align*}
This completes the proof.
\end{proof}

%\begin{lemma}
%Let $v \in S$. For all $\varphi_v \in S_{4,3, v}$,
%$\calZ_{v}(s, W_{v}, \calF({\varphi_v}))$
%has a meromorphic continuation and is holomorphic at $s=1$.
%\end{lemma}

%\begin{proof}
%By \eqref{eq:first} and Lemma \ref{lem:vanish_res},
%$\langle I^{(4,3)}(s, \Phi), \bar{f} \otimes f \rangle$
%has at most a simple (resp.~double) pole at $s=1$
%if $\pi$ is stable (resp.~endoscopic).
%Since $L_{\fin}(s, \pi \times \pi^{\vee}) = L_{\fin}(s, \Pi \times \Pi^{\vee})$
%has a simple (resp.~double) pole at $s = 1$
%if $\pi$ is stable (resp.~endoscopic),
%the assertion follows.
%\end{proof}

\begin{lemma} \label{lem:nonvanish}
Let $v \in S$. There exists $\varphi_v \in S(V_{3,3}^4(k_v))$ such that 
\[
 Z_{v} \left(\frac{1}{2}, \phi_{v}, F_v({\varphi_v})\right) \ne 0.
\]
Moreover, for $v \mid \frak{n}$ or $v \in S({\rm PS})$, then we can find such $\varphi_v$ which does not depend on $\pi_v$.
\end{lemma}

\begin{proof}
We identify $k_v=\R$ if $v$ is a real place. Let $\pi_0$ be an irreducible component of $\pi_v\vert_{\Sp_4(k_v)}$.
Fix a pairing $\langle \ , \ \rangle$ on $\pi_0$.
For $f_1, f_2 \in \pi_0$,
define a matrix coefficient $\phi_{f_1 \otimes f_2}$ of $\pi_0$ by
\[
 \phi_{f_1 \otimes f_2}(g) = \langle \pi_0 (g) f_1, f_2 \rangle.
\]
Then the local zeta integral $Z_v({1}/{2}, \phi_{f_1 \otimes f_2}, F)$
is absolutely convergent by Lemma \ref{L:ab doubling}, and defines
an $\Sp_4(k_v) \times \Sp_4(k_v)$-intertwining map (resp.~$(\sp_4(\RR), \U(2)) \times (\sp_4(\RR), \U(2))$-intertwining map)
\[
 I_{4,4,v} \left(\frac{1}{2}\right)
  \longrightarrow \pi_0^{\vee} \otimes \pi_0, \quad
 F  \longmapsto \left[ f_1 \otimes f_2 \mapsto
 Z_v \left(\frac{1}{2}, \phi_{f_1 \otimes f_2}, F\right) \right]
\]
if $v$ is finite (resp.~$v$ is real).
By \cite[Corollary 3.2.3]{KRbook} and \cite[Proposition 7.2.1]{KR1994},
for fixed non-zero elements $f_1, f_2 \in \pi_0$,
there exists $F \in I_{4,4,v}(1/2)$ such that 
\begin{align}\label{E:non-vanishing}
Z_v\left(\frac{1}{2}, \phi_{f_1 \otimes f_2}, F\right) \ne 0.
\end{align}
Moreover, if $v \mid \frak{n}$, then by the proof of \cite[Proposition 7.2.1]{KR1994}, we can find such $F$ which depends only on the stabilizers of $f_1$ and $f_2$ in $\Sp_4(k_v)$.

Assume $v \mid \frak{n}$. Let $V_0$ be the split quadratic space of dimension $4$ over $k_v$. Let $V_1$ (resp.~$V_2$) be the split quadratic space (resp.~quaternionic quadratic space) of dimension $6$ over $k_v$ and $\omega_{\psi_v,V_i,4}$ the Weil representation of $\Sp_8(k_v)\times {\rm O}(V_i)(k_v)$ on $S(V_i^4(k_v))$ with respect to $\psi_v$. For $i=1,2$, let $R(V_i)$ be the image of the $\Sp_8(k_v)$-intertwining map
\[
 S(V_i^4(k_v))  \longrightarrow I_{4,4,v} \left(\frac{1}{2}\right), \quad \varphi  \longmapsto F_v({\varphi}),
\]
where $F_v({\varphi})(g, \frac{1}{2}) = \omega_{\psi_v, V_i, 4}(g, 1) \varphi(0)$. By \cite{KR1992}, 
\begin{align}\label{E:deg. prin. finite}
I_{4,4,v}\left(\frac{1}{2}\right) = R(V_1)+R(V_2).
\end{align}
%and 
%$$R(V_1) / (R(V_1) \cap R(V_2)), \quad R(V_2) / (R(V_1) \cap R(V_2))$$
%are irreducible. Let $\varphi^o = \mathbb{I}_{V_1^4(\o)} \in S(V_1^4(k_v))$. Since $R(V_1) / (R(V_1) \cap R(V_2))$ is generated by the image of $F({\varphi^o})$,
%the natural map
%\[
% S(V_1^4(k_v)) \longrightarrow R(V_1) / (R(V_1) \cap R(V_2))
%\]
%is surjective. 
For $i=0,1,2$, let $\theta_i(\pi_0^\vee)$ be the theta lift of $\pi_0^\vee$ to ${\rm O}(V_i)(k_v)$. By our assumption on $\pi_0$  and \cite[Theorem A.10]{GI2011}, $\theta_0(\pi_0^\vee) \neq 0$. Therefore, by the conservation relation \cite{SZ2015} (see also \cite[Theorem 3.8]{KR2005}), $\theta_2(\pi_0^\vee)=0$. By \cite[Proposition 3.1]{HKS1996}, this is equivalent to 
\[
{\rm Hom}_{\Sp_4(k_v)\times \Sp_4(k_v)}(R(V_2),\pi_0^\vee \otimes \pi_0)=0.
\]
From this together with (\ref{E:non-vanishing}) and (\ref{E:deg. prin. finite}), we deduce that $Z_v(1/2,\phi_{f_1\otimes f_2},F) \neq 0$ for some $F \in R(V_1)$ which depends only on the stabilizers of $f_1$ and $f_2$ in $\Sp_4(k_v)$.
This completes the proof for $v \mid \frak{n}$.

Assume $v \in S({\rm DS})$.
For non-negative integers $p,q$, let $V_{p,q}$ denote the quadratic space over $\RR$ of signature $(p,q)$
and $\omega_{\psi_v,V_{p,q},4}$ the Weil representation of 
$\Sp_8(\RR) \times \O_{p,q}(\R)$ on $S(V_{p,q}^4(\R))$ with respect to $\psi_v$.
When $p$, $q$ are positive odd integers such that $p+q = 6$, let $R(p,q)$ be the image of the $(\sp_8(\RR), \U(4))$-intertwining map
\[
 S(V_{p,q}^4(\R))  \longrightarrow I_{4,4,v} \left(\frac{1}{2}\right), \quad \varphi  \longmapsto F_v({\varphi}),
\]
where $F_v({\varphi})(g, \frac{1}{2}) = \omega_{\psi_v, V_{p,q}, 4}(g, 1) \varphi(0)$.
By \cite{KR1990}, \cite{LZ1997},
\begin{align}\label{E:deg. prin. infty}
 I_{4,4,v} \left(\frac{1}{2}\right) = R(5,1) + R(3,3) + R(1,5).
\end{align}
%Let $\varphi^o \in S(V_{3,3}^4(\R))$ be the Gaussian function defined by
%$$\varphi^\circ(x)=e^{-\pi\,{\rm tr}(x{}^tx)}.$$
%Since $R(3,3)$ is generated by $F({\varphi^o})$,
%the natural map
%\[
% S(V_{3,3}^4(\R)) \longrightarrow R(3,3)
%\]
%is surjective. 
Let $\theta_{p,q}(\pi_0^{\vee})$ be the theta lift of $\pi_0^{\vee}$ 
to $\O_{p,q}(\R)$. By our assumption on $\pi_0$ and \cite[Theorem 18]{Paul2005}, $\theta_{2,2}(\pi_0^\vee) \neq 0$. Therefore, 
by the conservation relation \cite{SZ2015} (see also \cite[Proposition 22]{Paul2005}),
$\theta_{5,1}(\pi_0^{\vee}) = \theta_{1,5}(\pi_0^{\vee}) = 0$.
Hence,
as in the proof of \cite[Proposition 3.1]{HKS1996},
\begin{align*}
 \Hom_{(\sp_4(\RR), \U(2)) \times (\sp_4(\RR), \U(2))}
 (R(5,1), \pi_0^{\vee} \otimes \pi_0) & = 0, \\
 \Hom_{(\sp_4(\RR), \U(2)) \times (\sp_4(\RR), \U(2))}
 (R(1,5), \pi_0^{\vee} \otimes \pi_0) & = 0.
\end{align*}
From this together with (\ref{E:non-vanishing}) and (\ref{E:deg. prin. infty}), we deduce that $Z_v(1/2,\phi_{f_1\otimes f_2},F) \neq 0$ for some $F \in R(3,3)$. 
This completes the proof for $v \in S({\rm DS})$.

Assume $v \in S({\rm PS})$. Let $\varphi_v \in S(V_{3,3}^4(\R))$ be the Gaussian function. Proceeding similarly as in the proof of \cite[Proposition 6.2]{LNM1254} or Lemma \ref{L:archi. spherical local zeta integral} below, we see that
\[Z_v\left(\frac{1}{2},\phi_{v},F_v(\varphi_v)\right) \neq 0.\]
This completes the proof.
\end{proof}

\begin{prop}\label{P:main identity1}
We have 
\[
\<f,f\> = 2^c \cdot \frac{L(1,\pi,{\rm Ad})}{\Delta_{{\rm PGSp}_4}}\cdot\prod_v C(\pi_v).
\]
Here $C(\pi_v)$ is a non-zero constant depending only on $\pi_v$ for each place $v$, and 
\begin{align*}
\Delta_{\PGSp_4} & = \xi(2)\xi(4),\\
c & =\begin{cases}
      1 & \mbox{ if $\pi$ is stable},\\
      2 & \mbox{ if $\pi$ is endoscopic}.
     \end{cases}
\end{align*}
In fact, $C(\pi_v)=q_v^{-5\c_v}$ if $v \notin S$ and 
\begin{align}\label{E:local constants1}
C(\pi_v) & =  \zeta_v(1)^{-1}\zeta_v(3)^{-1}\zeta_v(4) L(1,\pi_v,{\rm Ad})^{-1}\cdot\frac{\calZ_v(1, W_v, \calF_v({\varphi_v}))}{Z_v \left(1/2, \phi_{v}, F_v(\varphi_{v})\right)}\cdot \begin{cases}
q_v^{-9\c_v/2} & \mbox{ if $v\mid\frak{n}$},\\
1 & \mbox{ if $v \mid \infty$},
\end{cases}
\end{align}
where $\varphi_v \in S(V_{3,3}^4(k_v))$ is any Schwartz function such that $Z_v(1/2,\phi_v,F_v(\varphi_v)) \neq 0$.
\end{prop}

\begin{proof}
Let $\varphi = \bigotimes_v \varphi_v \in S(V_{3,3}^4(\AA))$
be such that $\varphi_v = \varphi_v^o$ for all places $v \notin S$.
Then the holomorphic section $F(\varphi) = \bigotimes_v F_v(\varphi_v)$ of $I(s)$
satisfies the following conditions:
\begin{itemize}
 \item $F_v(\varphi_v)$ is $H(\o_v)$-invariant and $F_v(\varphi_v)(1,s) = q_v^{-6\c_v}$
 for all places $v \notin S$,
 \item $F_{v}(\varphi_v)$ depends only on $\varphi_v$ for all places $v \in S$.
\end{itemize}
Also,
the holomorphic section $\calF(\varphi) = \frak{D}^{-9/2}\rho^{-1}\xi(2)^{-1}\xi(3)^{-1} \cdot\bigotimes_v \calF_v(\varphi_v)$ of $\calI(s)$
for $\Re(s) > -1$ satisfies the following conditions:
\begin{itemize}
 \item $\calF_v(\varphi_v)$ is $H(\o_v)$-invariant and $\calF_v(\varphi_v)(1,s) = \zeta_v(s+1)\zeta_v(s+2)\zeta_v(s+3)$ for all places $v \notin S$,
 \item $\calF_{v}(\varphi_v)$ depends only on $\varphi_{v}$ for all places $v \in S$.
\end{itemize}
Let $T = \{v \mid \infty\frak{nd}\}.$ By Lemmas \ref{L:unram PSR} and \ref{L:unram Jiang}, we have
\begin{align*}
 \langle E(s, F({\varphi})), \bar{f} \otimes f \rangle 
 & = \langle f, f \rangle \cdot \xi^T(2)^{-1}\xi^T(4)^{-1}\cdot d_{P}^{T}(s)^{-1}\cdot
 L^{T}\left(s+\frac{1}{2}, \pi, {\rm std}\right)
 \prod_{v \in T}Z_{v}(s, \phi_{v}, F_v({\varphi_{v}})), \\
 \langle \calE(s, \calF(\varphi)), \bar{f} \otimes f \rangle
 & = \frak{D}^{-9/2}\rho^{-1}\xi(2)^{-1}\xi(3)^{-1}\cdot \xi^T(2)^{-2}\xi^T(4)^{-2}\cdot d_{\mathcal{P}}^{T}(s)^{-1}\cdot \xi^{T}(s+1)\xi^{T}(s+2)\xi^{T}(s+3)\\
&\times L^{T}\left(\frac{s+1}{2}, \pi \times \pi^{\vee}\right)
 \prod_{v \in T}\calZ_{v}(s, W_{v}, \calF_v({\varphi_{v}})).
\end{align*}
On the other hand, by the first term identity \eqref{E:1st identity},
\[
 \langle \calE_{-2}(1,\calF(\varphi)), \bar{f} \otimes f \rangle
 = \left\langle
 E_{-1} \left(\frac{1}{2}, F(\varphi)\right), \bar{f} \otimes f \right\rangle.
\]
Hence, noting that $d_P^T(1/2) = \xi^T(2)\xi^T(3)\xi^T(4)$, $d_{\mathcal{P}}^T(1) = \xi^T(2)\xi^T(3)\xi^T(4)^2$, and
\[
L^T(s,\pi\times\pi^\vee) = \xi^T(s)L^T(s,\pi,{\rm std})L^T(s,\pi,{\rm Ad}),
\]
we have
\begin{align}\label{E:general iden.}
\begin{split}
 & 4 \frak{D}^{-9/2}\xi(2)^{-1}\xi(3)^{-1}\xi^{T}(4)^{-1} \cdot \Res_{s=1}L^{T}(s, \pi, {\rm std})\cdot \frac{L^{T}(1, \pi, \Ad)}{\xi^T(2)\xi^T(4)}
 \prod_{v \in T}\zeta_v(1)^{-1}\calZ_v(1, W_v, \calF_v({\varphi_{v}})) \\
 & = 
 \langle f, f \rangle  \xi^{T}(2)^{-1}\xi^{T}(3)^{-1}\xi^{T}(4)^{-1}\cdot
 \Res_{s=1}L^{T}(s, \pi, {\rm std}) 
 \prod_{v \in T}Z_v \left(\frac{1}{2}, \phi_{v}, F_v({\varphi_v})\right).
\end{split}
\end{align}
If $\pi$ is endoscopic, then since ${\rm Res}_{s=1}L^T(s,\pi,{\rm std}) \neq 0$ and $\prod_{v \in T}L(s,\pi_v,{\rm Ad})$ is holomorphic and non-zero at $s=1$, we deduce that
\begin{align}\label{E:endoscopic iden.}
\begin{split}
&\langle f, f \rangle
 \prod_{v \in T}Z_v \left(\frac{1}{2}, \phi_{v}, F_v(\varphi_{v})\right)\\
&= 4 \frak{D}^{-9/2}\cdot\frac{L(1, \pi, \Ad)}{\Delta_{\PGSp_4}} \prod_{v \in T} \zeta_v(1)^{-1}\zeta_v(3)^{-1}\zeta_v(4)L(1,\pi_v,{\rm Ad})^{-1}\calZ_v(1, W_v, \calF_v({\varphi_v})).
\end{split}
\end{align}
If $\pi$ is stable, then $L^T(s,\pi,{\rm std})$ is holomorphic and non-zero at $s=1$, so that both sides of (\ref{E:general iden.}) vanish. However, by the second term identity (\ref{E:2nd identity}) and Lemmas \ref{lem:vanish_res}, \ref{lem:vanish_theta},
\[
 \langle \calE_{-1}(1,\calF(\varphi)), \bar{f} \otimes f \rangle
 = \left\langle
 E_{0} \left(\frac{1}{2}, F(\varphi)\right), \bar{f} \otimes f \right\rangle.
\]
Hence, similarly as above, we have
\begin{align*}
 & 2 \frak{D}^{-9/2}\xi(2)^{-1}\xi(3)^{-1}\xi^{T}(4)^{-1} \cdot L^{T}(1, \pi, {\rm std}) \cdot\frac{L^{T}(1, \pi, \Ad)}{\xi^T(2)\xi^T(4)}
 \prod_{v \in T}\zeta_v(1)^{-1}\calZ_v(1, W_v, \calF_v({\varphi_{v}})) \\
 & = 
 \langle f, f \rangle  \xi^{T}(2)^{-1}\xi^{T}(3)^{-1}\xi^{T}(4)^{-1}\cdot
 L^{T}(1, \pi, {\rm std}) 
 \prod_{v \in T}Z_v \left(\frac{1}{2}, \phi_{v}, F_v({\varphi_v})\right),
\end{align*}
and deduce that 
\begin{align}\label{E:stable iden.}
\begin{split}
&\langle f, f \rangle
 \prod_{v \in T}Z_v \left(\frac{1}{2}, \phi_{v}, F_v(\varphi_{v})\right)\\
&= 2 \frak{D}^{-9/2}\cdot\frac{L(1, \pi, \Ad)}{\Delta_{\PGSp_4}} \prod_{v \in T} \zeta_v(1)^{-1}\zeta_v(3)^{-1}\zeta_v(4)L(1,\pi_v,{\rm Ad})^{-1}\calZ_v(1, W_v, \calF_v({\varphi_v})).
\end{split}
\end{align}

Now recall that $\varphi_v=\varphi_v^o$ and hence $Z_v(1/2,\phi_v,F_v(\varphi_v)) \neq 0$ for all $v \notin S$ by Lemma \ref{L:unram PSR}. Also, we may choose $\varphi_v$ such that $Z_v(1/2,\phi_v,F_v(\varphi_v)) \neq 0$ for all $v \in S$ by Lemma \ref{lem:nonvanish}. For each place $v$ of $k$, let
\[
 C(\pi_v) = \zeta_v(1)^{-1}\zeta_v(3)^{-1}\zeta_v(4) L(1,\pi_v,{\rm Ad})^{-1}\cdot\frac{\calZ_v(1, W_v, \calF_v({\varphi_v}))}{Z_v \left(1/2, \phi_{v}, F_v(\varphi_{v})\right)}\cdot \begin{cases}
q_v^{-9\c_v/2} & \mbox{ if $v$ is finite},\\
1 & \mbox{ if $v$ is real}.
\end{cases}
\]
Note that $C(\pi_v)$ is a purely local constant depending only on $\pi_v$ and $\varphi_v$. By Lemmas \ref{L:unram PSR} and \ref{L:unram Jiang}, $C(\pi_v) = q_v^{-5\c_v}$ for all $v \notin S$. In particular, $C(\pi_v) = 1$ for all $v \notin T$. Finally, by (\ref{E:endoscopic iden.}) and (\ref{E:stable iden.}), we have
\[
\<f,f\> = 2^c \cdot\frac{L(1,\pi,{\rm Ad})}{\Delta_{{\rm PGSp}_4}}\cdot\prod_v C(\pi_v).
\]
Since the left-hand side is non-zero and independent of $\varphi_v$, so is $C(\pi_v)$ for all $v \in S$. This completes the proof. 
\end{proof}

In the following proposition, we give an explicit Rallis inner product formula.
\begin{prop}\label{P:main identity2}
Assume $\pi$ is endoscopic. We have 
\[
\<f,f\> = 4 \cdot\frac{L(1,\pi,{\rm Ad})}{\Delta_{{\rm PGSp}_4}}\cdot\prod_v C'(\pi_v),
\]
where 
\begin{align}\label{E:local constants2}
C'(\pi_v) = \begin{cases}
\displaystyle{q_v^{-5\c_v}} & \mbox{ if $v \notin S$},\\
\displaystyle{q_v^{-1-5\c_v}\zeta_v(2)^{-1}\zeta_v(4)} & \mbox{ if $v \mid \frak{n}$},\\
\displaystyle{2^{\lambda_{1,v}-\lambda_{2,v}+5}\pi^{3\lambda_{1,v}-\lambda_{2,v}+5}(1+\lambda_{1,v}-\lambda_{2,v})^{-1}}  & \mbox{ if $v \in S({\rm DS})$},\\
\displaystyle{2^{-4}} & \mbox{ if $v \in S({\rm PS})$}.
\end{cases}
\end{align}
\end{prop}

\begin{proof}
The assertion follows from Propositions \ref{P:W(1)} and \ref{P:Rallis inner product}, whose proofs will be given in \S\,\ref{S:cal1} and \S\,\ref{S:cal2}.
\end{proof}

\begin{corollary}\label{C:endoscopic case}
Assume $\pi$ is endoscopic. We have
\[\prod_{v \in S} C(\pi_v) = \prod_{v \in S}C'(\pi_v).\]
\end{corollary}

\begin{proof}
The assertion follows from Propositions \ref{P:main identity1} and \ref{P:main identity2}.
\end{proof}

\subsection{Families of local representations of $\GSp_4$}\label{SS:analytic family}
In this section, we switch to a local setting. Let $F$ be a local field of characteristic zero. When $F$ is non-archimedean, let $\o$, $\c$, $\varpi$, $q$, $|\mbox{ }|$ be the notation defined as in $\S\,\ref{S:notation}$, and fix a non-trivial additive character $\psi$ of $F$ with conductor $\varpi^{-\c}\o$. When $F=\R$, let $\psi(x)=e^{2\pi\sqrt{-1}x}$.

Consider the split orthogonal similitude group ${\rm GO}_{2,2}$. Put
\[H= {\rm GO}_{2,2},\quad H^\circ = {\rm GSO}_{2,2}.\]
There is an isomorphism (cf.\,\S\,\ref{SS:theta lifts})
\[{\mathbb G}_m \backslash (\GL_2\times\GL_2) \simeq H^\circ.\]
We write $[h_1,h_2]\in H^\circ$ for the image of $(h_1,h_2)\in \GL_2\times\GL_2$ under the above isomorphism. Let ${\bf t} \in H \setminus H^\circ$ be an involution so that ${\rm Ad}({\bf t})[h_1,h_2] = [\det(h_2)^{-1}h_2,\det(h_1)^{-1}h_1]$. Let $\sigma_1$ and $\sigma_2$ be irreducible admissible representations of $\GL_2(F)$ with central characters inverse to each other. Let $\mathcal{V}_{\sigma_i}$ be the space of $\sigma_i$ for $i=1,2$. Let $\sigma_1 \times \sigma_2$ be the representation of $H^\circ(F)$ on $\mathcal{V}_{\sigma_1}\otimes \mathcal{V}_{\sigma_2}$ defined by
\[[h_1,h_2]\cdot (v_1\otimes v_2) = \sigma_1(h_1)v_1\otimes \sigma_2(h_2)v_2\]
for $h_1,h_2 \in \GL_2(F)$ and $v_1\in \mathcal{V}_{\sigma_1}$, $v_2 \in \mathcal{V}_{\sigma_2}$.
 We define an irreducible admissible representation $(\sigma_1\times\sigma_2)^{\sharp}$ of $H(F)$ as follows:
\begin{itemize}
\item If $\sigma_1 \not\simeq \sigma_2^\vee$, then $(\sigma_1\times\sigma_2)^{\sharp} = {\rm Ind}_{H^\circ(F)}^{H(F)}(\sigma_1 \times \sigma_2)$.
\item If $\sigma_1 \simeq \sigma_2^\vee$, then since $\sigma_2^\vee \simeq \sigma_2 \otimes \omega_{\sigma_2}^{-1}$, we may assume $\mathcal{V}_{\sigma_1} = \mathcal{V}_{\sigma_2}$. Let $\mathcal{V}^\sharp = \mathcal{V}_{\sigma_1} \otimes \mathcal{V}_{\sigma_1}$. Let $(\sigma_1\times\sigma_2)^{\sharp}$ be the representation of $H(F)$ on $\mathcal{V}^\sharp$ defined by
\begin{align*}
\begin{split}
[h_1,h_2]\cdot(v_1\otimes v_2) &= \sigma_1(h_1)v_1\otimes \omega_{\sigma_1}(\det(h_2)^{-1})\sigma_1(h_2)v_2,\\
{\bf t}\cdot(v_1\otimes v_2) & = v_2\otimes v_1
\end{split}
\end{align*}
for $h_1,h_2 \in \GL_2(F)$ and $v_1,v_2 \in \mathcal{V}_{\sigma_1}.$
\end{itemize}

We consider three types of generic irreducible admissible representations of $\G(F)$ with trivial central character:
\begin{itemize}
\item[(IIa)] $F$ is non-archimedean and $\pi$ is a  representation of $\G(F)$ with paramodular conductor $\varpi\o$. 
\item[(DS)] $F=\R$ and $\pi$ is a (limit of) discrete series representation of $\G(\R)$.
\item[(PS)] $F=\R$ and $\pi$ is a principal series representation of $\G(\R)$ with non-zero $(\Sp_4(\R)\cap {\rm O}(4))$-invariant vectors.
\end{itemize}
Note that we do not require $\pi$ to be unitary.

Let $\pi$ be a representation of type (IIa). Then, as explained in the proof of \cite[Proposition 7.2.5]{RS2007}, $\pi$ is induced from the following representation of the standard Siegel parabolic subgroup 
\begin{align}\label{E:induced repre IIa}
\bp A & * \\ 0 & \nu {}^t \! A^{-1}\ep \longmapsto
({\rm St}\otimes |\mbox{ }|^{\lambda})(A)\cdot \eta^\varepsilon|\nu|^{-\lambda}
\end{align}
for some $\varepsilon \in \{0,1\}$ and $\lambda \in \C$.
Here ${\rm St}$ is the Steinberg representation of $\GL_2(F)$ and $\eta$ is the non-trivial unramified quadratic character of $F^\times$. Note that $\lambda \, {\rm mod}\, \frac{2\pi\sqrt{-1}}{\log q}\Z$ is determined by $\pi$ up to sign. We call $\varepsilon$ and $\lambda$ the sign and the parameter of $\pi$, respectively. Note that $\pi$ is the theta lift of the representation 
\begin{align}\label{E:theta1}
\left(({\rm St}\otimes {\eta^\varepsilon}
) \times {\rm Ind}_{B(F)}^{\GL_2(F)}(|\mbox{ }|^{\lambda}{\eta^\varepsilon}
\boxtimes|\mbox{ }|^{-\lambda}{\eta^\varepsilon})\right)^\sharp
\end{align}
of $H(F)$ to $\G(F)$.

Let $\pi$ be a representation of type (DS). Then \[\pi\vert_{\Sp_4(\R)} = D_{(\lambda_1,\lambda_2)}\oplus D_{(-\lambda_2,-\lambda_1)}.\]
Here $D_{(\lambda_1,\lambda_2)}$ is the (limit of) discrete series representation of $\Sp_4(\R)$ with Blattner parameter $(\lambda_1,\lambda_2)$ such that $1-\lambda_1 \leq \lambda_2 \leq 0$. We call $\lambda=(\lambda_1,\lambda_2) \in \Z^2$ the parameter of $\pi$. Note that $\pi$ is the theta lift of the representation
\begin{align}\label{E:theta2}
\left({\rm DS}(\lambda_1-\lambda_2)\times{\rm DS}(\lambda_1+\lambda_2)\right)^\sharp
\end{align}
of $H(\R)$ to $\G(\R)$. Here ${\rm DS}(\kappa)$ denotes the  (limit of) discrete series representation of $\GL_2(\R)$ with minimal weight $\kappa \in \Z_{\geq 1}$. Since we assume $\pi$ has trivial central character, $\lambda_1-\lambda_2$ is an even integer.

Let $\pi$ be a representation of type (PS). Then $\pi$ is induced from the following representation of the standard Borel subgroup 
\begin{align}\label{E:induced repre PS}
\bp t_1 & * & *&* \\0&t_2&*&*\\0&0&\nu t_1^{-1}&0\\0&0& *&\nu t_2^{-1}  \ep \longmapsto |t_1|^{\lambda_1}\cdot|t_2|^{\lambda_2}\cdot{\rm sgn}(\nu)^\varepsilon|\nu|^{-(\lambda_1+\lambda_2)/2}
\end{align}
for some $\varepsilon \in \{0,1\}$ and $\lambda = (\lambda_1,\lambda_2) \in \C^2$. Note that $\lambda$ is determined by $\pi$ up to the action of the Weyl group.
We call $\varepsilon$ and $\lambda$ the sign and the parameter of $\pi$, respectively. Note that $\pi$ is the theta lift of the representation
\begin{align}\label{E:theta3}
\left( {\rm Ind}_{B(\R)}^{\GL_2(\R)}(|\mbox{ }|^{(\lambda_1+\lambda_2)/2}{\rm sgn}^\varepsilon\boxtimes|\mbox{ }|^{(-\lambda_1-\lambda_2)/2}{\rm sgn}^\varepsilon) \times {\rm Ind}_{B(\R)}^{\GL_2(\R)}(|\mbox{ }|^{(\lambda_1-\lambda_2)/2}{\rm sgn}^\varepsilon\boxtimes |\mbox{ }|^{(-\lambda_1+\lambda_2)/2}{\rm sgn}^\varepsilon)\right)^\sharp
\end{align}
of $H(\R)$ to $\G(\R)$. 

Let $\mathcal{D}$ be the domain defined by
\begin{align}\label{E:domain}
\mathcal{D} = 
\begin{cases}
\left\{\lambda \in \C \mbox{ }\vert\mbox{ } |{\rm Re}(\lambda) |<1/2 \right\} & \mbox{ in Case (IIa)},\\
\left\{\lambda=(\lambda_1,\lambda_2) \in \C^2 \mbox{ }\vert\mbox{ }|{\rm Re}(\lambda_1)|+|{\rm Re}(\lambda_2)|<1 \right\} & \mbox{ in Case (PS)}.
\end{cases}
\end{align}  
By the unitarizability criterion \cite[Theorem 1.1]{LMT2004}, the domain $\mathcal{D}$ contains the set of parameters of unitary representations of type (IIa) or (PS).
%Given $\varepsilon \in \{0,1\}$ and $\lambda \in \C$ with $\vert{\rm Re}(\lambda)\vert <1/2$, by the unitarizability criterion \cite[Theorem 1.1]{LMT2004}, there exists a unique representation $\pi$ of type (PS) with parameter $\lambda$ and sign $\varepsilon$ defined as above.

Let $\pi$ be a representation of $G(F)$ in one of the three types. When $\pi$ is of type (IIa) or (PS), we assume its parameter is in $\mathcal{D}$. Let $W_\pi$ be the Whittaker function of $\pi$ with respect to $\psi_U$ defined as follows:
\begin{itemize}
\item In Case (IIa), $W_\pi$ is ${\rm K}(\varpi)$-invariant and $W_\pi(\diag( \varpi^{-\frak{c}}, 1, \varpi^{2\frak{c}}, \varpi^{\frak{c}} ))=1$.
\item In Case (DS), $W_\pi$ is a lowest weight vector of the minimal ${\rm U}(2)$-type of $D_{(-\lambda_2,-\lambda_1)}$ in the Whittaker model of $\pi$ with respect to $\psi_U$, and $W_\pi(1)$ is normalized as in (\ref{E:normalization DS}). We refer to \S\,\ref{SS:main thm} for the choice of the lowest weight vector.
\item In Case (PS), $W_\pi$ is $(\Sp_4(\R)\cap {\rm O}(4))$-invariant and $W_{\pi}(1)$ is normalized as in (\ref{E:normalization PS}).
\end{itemize}
Since $\pi$ has trivial central character, we have $\pi \simeq \pi^\vee$. Therefore $W_\pi = W_{\pi^\vee}$ when $\pi$ is of type (IIa) or (PS).

When $\pi$ is of type (IIa) or (PS), let $\phi_\pi$ be a matrix coefficient of $\pi$ defined by
\[\phi_\pi(g) = \frac{\<\pi(g)W_\pi,W_{\pi^\vee}\>}{\<W_\pi,W_{\pi^\vee}\>}.\]
Here $\<\mbox{ },\mbox{ }\>$ is an invariant bilinear pairing on $\pi \times \pi^\vee$. Note that if $\pi$ is unitary, then
\[\phi_\pi(g) = \frac{\<\pi(g)W_\pi,W_{\pi}\>_h}{\<W_\pi,W_{\pi}\>_h}\]
for any Hermitian pairing $\<\mbox{ },\mbox{ }\>_h$ on $\pi \times \pi$, and
\[\overline{W_\pi(g)} = W_\pi(\diag (-1,1,1,-1) g)\]
for all $g \in G(F)$.

When $\pi$ is of type (DS), let $\phi_\pi$ be the matrix coefficient of $\pi$ defined by
\[\phi_\pi(g) = \frac{\<\pi(g)W_\pi,W_{\pi}\>}{\<W_\pi,W_{\pi}\>}\]
for any Hermitian pairing $\<\mbox{ },\mbox{ }\>$ on $\pi \times \pi$.

Let $I(s)=I_{4,4}(s)$ and $\calI(s)=I_{4,3}(s)$ be the degenerate principal series representations of $\GSp_8(F)$ defined as in $\S\,\ref{ss:eisenstein}$. If $F\in I(s)$ and $\calF \in \calI(s)$ are holomorphic sections, let $Z(s,\phi_\pi,F)$ and $\calZ(s,W_\pi,\calF)$ be the local zeta integrals defined as in (\ref{E:local zeta1}) and (\ref{E:local zeta2}), respectively. Note that when $\pi$ is of type (IIa) or (PS), we replace $\overline{W_\pi}$ by the left translation of $W_\pi$ by $\diag(-1,1,1,-1)$.
Recall the intertwining maps 
\begin{align*}
S({\rm M}_{6,4}(F)) &\longrightarrow \calI(s),\quad \varphi \longmapsto \calF(\varphi),\\
S({\rm M}_{6,4}(F)) &\longrightarrow I(s),\quad \varphi \longmapsto F(\varphi),
\end{align*}
defined in \S\,\ref{SS:Theta integrals} and \S\,\ref{SS:First and second term identities}, respectively.
By Lemma \ref{lem:nonvanish}, there exists a Schwartz function $\varphi \in S({\rm M}_{6,4}(F))$ such that
\[Z\left(\frac{1}{2},\phi_\pi,F(\varphi)\right)\neq 0,\]
and such that $\varphi$ does not depend on $\pi$ when $\pi$ is of type (IIa) or (PS).
We fix one such $\varphi$ and define constants $C(\pi)$ and $C'(\pi)$ as in (\ref{E:local constants1}) and (\ref{E:local constants2}), respectively. 

\begin{corollary}\label{C:analytic family}
Assume $\pi$ is of type (IIa) or (PS) with parameter $\lambda$ in $\mathcal{D}$.
The constants $C(\pi)$ and $C'(\pi)$ as functions of $\lambda$ are analytic.
\end{corollary}

\begin{proof}
It is easy to show that the maps $\lambda \mapsto \phi_\pi$ and $\lambda \mapsto W_\pi$ define $K$-finite analytic families of matrix coefficiens and Whittaker functions in the sense introduced in \S\,\ref{SS:Doubling local zeta integrals} and \S\,\ref{SS:local zeta for Rankin-Selberg} below, respectively.
The assertion then follows directly from Lemmas \ref{L:ab1} and \ref{L:uniform2}.
\end{proof}

\begin{prop}\label{P:main ide}
Let $\pi$ be a representation in one of the three types. When $\pi$ is of type (IIa) or (PS), we assume its parameter is in $\mathcal{D}$. We have
\begin{align}\label{E:main ide}
C(\pi) = C'(\pi).
\end{align}
\end{prop}
It is clear that Theorem \ref{T:main thm} follows from Propositions \ref{P:main identity1} and \ref{P:main ide}.
We divide the proof of Proposition \ref{P:main ide} into four steps as follows:
\begin{itemize}
\item[Step 1.] Establish (\ref{E:main ide}) in Case (DS) when $\lambda_1-\lambda_2$ and $\lambda_1+\lambda_2$ are sufficiently large.
\item[Step 2.] Establish (\ref{E:main ide}) in Case (IIa) when $F=\Q_p$.
\item[Step 3.] Establish (\ref{E:main ide}) in Case (DS) for arbitrary parameter $\lambda$.
\item[Step 4.] Establish (\ref{E:main ide}) in Case (PS) and in Case (IIa) for arbitrary non-archimedean $F$.
\end{itemize}
The ingredients are Corollary \ref{C:endoscopic case}, explicit local theta correspondence (\ref{E:theta1}), (\ref{E:theta2}), (\ref{E:theta3}), and the limit multiplicity formula \cite[\S 5]{Shin2012}.

\subsection{Proof of Proposition \ref{P:main ide}}
We need the following result on the existence of automorphic representations of $\GL_2(\A_k)$ for a totally real number field $k \neq \Q$. The proposition is a simple variant of the results in \cite{Shin2012} and is well-known, but we give a proof for completeness. For simplicity, we assume $[k:\Q]$ is odd. In fact, we use the following result only for totally real cubic number fields.

\begin{prop}\label{P:existence}
Assume $k\neq \Q$ is a totally real number field of odd degree. Let $v_0$ be a real place of $k$ and $\chi$ a unitary character of $\R^{\times}$. For each open neighborhood $U$ of $0$ in $\R$, there exist infinitely many irreducible cuspidal automorphic representations $\sigma$ of $\GL_2(\A_k)$ with trivial central character satisfying the following conditions:
\begin{itemize}
\item $\sigma_v$ is unramified for all finite places $v$.
\item $\sigma_v$ is a discrete series representation for all real places $v \neq v_0$.
\item $\sigma_{v_0} = {\rm Ind}_{B(\R)}^{\GL_2(\R)}(\chi|\mbox{ }|^{\sqrt{-1}t}\boxtimes\chi^{-1}|\mbox{ }|^{-\sqrt{-1}t})$ for some $t \in U$.
\end{itemize}
\end{prop}

\begin{proof}
Let $S_0$ be the set of real places $v \neq v_0$.
Since $[k:\Q]$ is odd and $k \ne \Q$, there exists a unique quaternion division algebra $D$ over $k$ which is ramified precisely at the places in $S_0$.
We write $G = D^\times / k^\times$ only in this proof and regard it as an algebraic group over $k$.
Note that $G$ is anisotropic over $k$ and $G(k_v)$ is compact for all $v \in S_0$.
For each finite place $v$, we denote by $K_v = \PGL_2(\mathfrak{o}_v)$ the standard maximal compact subgroup of $G(k_v) = \PGL_2(k_v)$, and put $K_{f} = \prod_{v \nmid \infty} K_v$.
For each $v \in S_0$, fix a sequence 
{$\{ \tau_{v,n} \}_{n \ge 1}$}
of irreducible representations of $G(k_v)$, and put 
{$\tau_n = \bigotimes_{v \in S_0} \tau_{v,n}$}.
We assume that 
\[
{\lim_{n \rightarrow \infty} \dim \tau_n = \infty.}
\]
We denote by $\widehat{G(k_{v_0})}$ the unitary dual of $G(k_{v_0})$ equipped with the Fell topology and by $\mu_{\mathrm{pl}}$ the Plancherel measure on $\widehat{G(k_{v_0})}$.
We define a subset $\mathcal{U}$ of $\widehat{G(k_{v_0})}$ by 
\[
 \mathcal{U} = \left\{ \left. \Ind^{\GL_2(\R)}_{B(\R)}(\chi|\ |^{\sqrt{-1} t} \boxtimes \chi^{-1}|\ |^{-\sqrt{-1} t}) \, \right| \, t \in U \right\}.
\]
Then, by the Jacquet-Langlands correspondence, it suffices to show that for any sufficiently large $n$, there exists an irreducible automorphic representation $\pi$ of $G(\A_k)$ satisfying the following conditions:
\begin{itemize}
\item $\pi_v$ has a non-zero $K_v$-invariant vector for all finite places $v$.
\item {$\pi_v = \tau_{v,n}$}
for all $v \in S_0$.
\item $\pi_{v_0} \in \mathcal{U}$.
\end{itemize}

We now consider a sequence $\{ \mu_n \}_{n \ge 1}$ of positive Borel measures on $\widehat{G(k_{v_0})}$ given by 
\[
 \mu_n = \frac{1}{\vol(G(k) \backslash G(\A_k)) \dim {\tau_n}} \sum_{\rho} m_n(\rho) \delta_\rho,
\]
where $\rho$ runs over irreducible unitary representations of $G(k_{v_0})$, $m_n(\rho)$ is the multiplicity of 
{$\rho \otimes \tau_n$}
in the space of $K_f$-invariant automorphic forms on $G(\A_k)$, and $\delta_\rho$ is the Dirac measure at $\rho$.
Then we are reduced to showing that 
\[
 \lim_{n \rightarrow \infty} \mu_n(\mathbb{I}_{\mathcal{U}}) = \mu_{\mathrm{pl}}(\mathbb{I}_{\mathcal{U}}).
\]
By the density theorem of Sauvageot \cite[Theorem 7.3]{sauvageot1997}, we are further reduced to showing that 
\[
 \lim_{n \rightarrow \infty} \mu_n(\widehat{\phi})
 = \mu_{\mathrm{pl}}(\widehat{\phi})
\]
for all $\phi \in C_c^\infty({G(k_{v_0})})$, where $\widehat{\phi}(\rho) = \tr(\rho(\phi))$ for $\rho \in \widehat{G(k_{v_0})}$.
To prove this, we use the trace formula
\[
 I_{\mathrm{geom}}(f) = I_{\mathrm{spec}}(f).
\]
Here, for a test function $f \in C_c^\infty(G(\A_k))$, the geometric side $I_{\mathrm{geom}}(f)$ is given by 
\[
 I_{\mathrm{geom}}(f) = \sum_{\gamma} \vol(G_\gamma(k) \backslash G_\gamma(\A_k)) O_\gamma(f), 
\]
where $\gamma$ runs over conjugacy classes of $G(k)$, $G_\gamma$ is the centralizer of $\gamma$ in $G$, and $O_\gamma(f)$ is the orbital integral of $f$ at $\gamma$.
Also, the spectral side $I_{\mathrm{spec}}(f)$ is given by 
\[
 I_{\mathrm{spec}}(f) = \sum_{\pi} m(\pi) \tr(\pi(f)),
\]
where $\pi$ runs over irreducible unitary representations of $G(\A_k)$ and $m(\pi)$ is the multiplicity of $\pi$ in $L^2(G(k) \backslash G(\A_k))$.
If we take a test function 
\[
 f_n = \frac{1}{\vol(G(k) \backslash G(\A_k)) \dim {\tau_n}} \cdot \left( \bigotimes_{v \in S_0} f_{v,n} \right) \otimes \left( \bigotimes_{v \notin S_0} f_v \right)
\]
given by 
\begin{itemize}
\item $f_v = \mathbb{I}_{K_v}$ for all finite places $v$,
\item $f_{v,n}(g) = \tr({\tau_{v,n}}(g^{-1}))$ for all $v \in S_0$, 
\item $f_{v_0} = \phi$,
\end{itemize}
then
\[
 I_{\mathrm{spec}}(f_n) = \frac{1}{\vol(G(k) \backslash G(\A_k)) \dim {\tau_n}} \sum_{\rho} m_n(\rho) \tr(\rho(\phi)) = \mu_n(\widehat{\phi}).
\]
On the other hand, we have
\[
 O_\gamma(f_n) = \frac{1}{\vol(G(k) \backslash G(\A_k))} \cdot \prod_{v \in S_0} \frac{\vol(G_{\gamma,v} \backslash G_v) \tr({\tau_{v,n}}(\gamma^{-1}))}{\dim {\tau_{v,n}}} \cdot \prod_{v \notin S_0} O_\gamma(f_v).
\]
If $\gamma \ne 1$, then by \cite[Corollaire 1.12]{CC2009}, we have
\[
 \lim_{n\rightarrow \infty} \prod_{v \in S_0} \frac{\tr({\tau_{v,n}}(\gamma^{-1}))}{\dim {\tau_{v,n}}} = 0.
\]
Hence, noting that the sum in $I_{\mathrm{geom}}(f_n)$ can be taken over a finite set independent of $n$, we have
\[
 \lim_{n \rightarrow \infty} I_{\mathrm{geom}}(f_n) = \lim_{n \rightarrow \infty} \vol(G(k) \backslash G(\A_k)) f_n(1) = \phi(1) = \mu_{\mathrm{pl}}(\widehat{\phi}).
\]
This implies the assertion. 
\end{proof}

Now we begin the proof of Proposition \ref{P:main ide}.
We will use a global-to-local argument repeatedly. When $k$ is a totally real number field, and $\sigma_1$, $\sigma_2$ are irreducible unitary cuspidal automorphic representations of $\GL_2(\A_k)$ with central characters inverse to each other, we put
\[(\sigma_1\times\sigma_2)^\sharp = \bigotimes_v(\sigma_{1,v}\times\sigma_{2,v})^\sharp.\]
By \cite[Proposition 5.4]{Takeda2009}, $(\sigma_1\times\sigma_2)^\sharp$ is an irreducible cuspidal automorphic representation of $H(\A_k)$. Let $\theta((\sigma_1\times\sigma_2)^\sharp)$ be the global theta lift of $(\sigma_1\times\sigma_2)^\sharp$ to $G(\A_k)$. If $\sigma_1 \not\simeq \sigma_2^\vee$, then $\theta((\sigma_1\times\sigma_2)^\sharp)$ is an irreducible unitary endoscopic cuspidal automorphic representation of $G(\A_k)$.

{\bf Step\,1.} Assume $\pi$ is of type (DS) such that $\lambda_1+\lambda_2$ and $\lambda_1-\lambda_2$ are sufficiently large positive even integers. By the dimension formulae for the space of holomorphic elliptic cusp forms of full level, there exist irreducible cuspidal automorphic representations $\sigma_1$ and $\sigma_2$ of $\GL_2(\A_\Q)$ with trivial central character satisfying the following conditions:
\begin{itemize}
\item $\sigma_{1,p}$ and $\sigma_{2,p}$ are unramified for all primes $p$.
\item $\sigma_{1,\infty}={\rm DS}(\lambda_1-\lambda_2)$ and $\sigma_{2,\infty}={\rm DS}(\lambda_1+\lambda_2)$.
\end{itemize} 
Then $\theta((\sigma_1\times\sigma_2)^\sharp)$ is an irreducible globally generic cuspidal automorphic representation of $G(\A_\Q)$ with trivial central character satisfying the following conditions:
\begin{itemize}
\item $\theta((\sigma_1\times\sigma_2)^\sharp)_p$ is unramified for all primes $p$.
\item $\theta((\sigma_1\times\sigma_2)^\sharp)_\infty = \pi$.
\end{itemize}
Therefore, (\ref{E:main ide}) follows from Corollary \ref{C:endoscopic case}.

{\bf Step\,2.} Assume $\pi$ is of type (IIa) with $F=\Q_p$ for some prime $p$. By Corollary \ref{C:analytic family} and the identity theorem for holomorphic functions, we may assume $\pi$ is tempered, that is, ${\rm Re}(\lambda)=0$. Fix an irreducible cuspidal automorphic representation $\sigma_1$ of $\GL_2(\A_\Q)$ with trivial central character such that
\begin{itemize}
\item $\sigma_{1,\ell}$ is unramified for all primes $\ell \neq p$.
\item $\sigma_{1,p} = {\rm St}\otimes {\eta^\varepsilon}$.
\item $\sigma_{1,\infty}={\rm DS}(\kappa_1)$ for some sufficiently large $\kappa_1 \in \Z_{\geq 2}$.
\end{itemize}
Let $U$ be an open neighborhood of $0$ in $\R$. By \cite[Theorem 5.8]{Shin2012}, there exist infinitely many irreducible cuspidal automorphic representations $\sigma_2$ of $\GL_2(\A_\Q)$ with trivial central character such that
\begin{itemize}
\item $\sigma_{2,\ell}$ is unramified for all pimes $\ell \neq p$.
\item $\sigma_{2,p} = {\rm Ind}_{B(\Q_p)}^{\GL_2(\Q_p)}(|\mbox{ }|^{\lambda+\sqrt{-1}t}{\eta^\varepsilon}
\boxtimes|\mbox{ }|^{-\lambda-\sqrt{-1}t}{\eta^\varepsilon})$ for some $t \in U$.
\item $\sigma_{2,\infty}$ is a discrete series representation.
\end{itemize}
We fix one such representation $\sigma_2$ such that $\sigma_{2,\infty}$ has sufficiently large weight $\kappa_2 \in \Z_{\geq 2}$. Then $\theta((\sigma_1\times\sigma_2)^\sharp)$ is an irreducible globally generic cuspidal automorphic representation of $G(\A_\Q)$ with trivial central character satisfying the following conditions:
\begin{itemize}
\item $\theta((\sigma_1\times\sigma_2)^\sharp)_\ell$ is unramified for all primes $\ell \neq p$. 
\item $\theta((\sigma_1\times\sigma_2)^\sharp)_p$ is of type (IIa) with parameter $\lambda+t\sqrt{-1}$ for some $t \in U$ and sign $\varepsilon$.
\item $\theta((\sigma_1\times\sigma_2)^\sharp)_\infty$ is of type (DS) with parameter $((\kappa_1+\kappa_2)/2,-|\kappa_1-\kappa_2|/2)$.
\end{itemize}
By Corollary \ref{C:endoscopic case}, 
\[C(\theta((\sigma_1\times\sigma_2)^\sharp)_p)C(\theta((\sigma_1\times\sigma_2)^\sharp)_\infty) = C'(\theta((\sigma_1\times\sigma_2)^\sharp)_p)C'(\theta((\sigma_1\times\sigma_2)^\sharp)_\infty).\]
On the other hand, since $\kappa_1$ and $\kappa_2$ are sufficiently large,
\[C(\theta((\sigma_1\times\sigma_2)^\sharp)_\infty) = C'(\theta((\sigma_1\times\sigma_2)^\sharp)_\infty)\]
by {\bf Step\,1}. We conclude that
\[C(\theta((\sigma_1\times\sigma_2)^\sharp)_p) = C'(\theta((\sigma_1\times\sigma_2)^\sharp)_p).\]
Therefore, (\ref{E:main ide}) follows from Corollary \ref{C:analytic family}.

{\bf Step\,3.} Assume $\pi$ is of type (DS) for arbitrary parameter $\lambda$. Recall that $\lambda_1+\lambda_2$ and $\lambda_1-\lambda_2$ are positive even integers. By the dimension formulae for the space of holomorphic elliptic cusp forms,  there exist two distinct primes $p_1$, $p_2$, and irreducible cuspidal automorphic representations $\sigma_1$ and $\sigma_2$ of $\GL_2(\A_\Q)$ with trivial central character such that
\begin{itemize}
\item $\sigma_{1,\ell}$ is unramified for all primes $\ell \neq p_1$.
\item $\sigma_{2,\ell}$ is unramified for all primes $\ell \neq p_2$.
\item $\sigma_{1,p_1}$ and $\sigma_{2,p_2}$ are special representations of conductor $p_1\Z_{p_1}$ and $p_2\Z_{p_2}$, respectively.
\item $\sigma_{1,\infty} = {\rm DS}(\lambda_1-\lambda_2)$ and $\sigma_{2,\infty}= {\rm DS}(\lambda_1+\lambda_2)$ are discrete series representations. 
\end{itemize}
Then $\theta((\sigma_1\times\sigma_2)^\sharp)$ is an irreducible globally generic cuspidal automorphic representation of $G(\A_\Q)$ with trivial central character satisfying the following conditions:
\begin{itemize}
\item $\theta((\sigma_1\times\sigma_2)^\sharp)_\ell$ is unramified for all primes $\ell \notin \{p_1,p_2\}$.
\item $\theta((\sigma_1\times\sigma_2)^\sharp)_{p_i}$ is of type (IIa) for $i=1,2$.
\item $\theta((\sigma_1\times\sigma_2)^\sharp)_\infty = \pi$.
\end{itemize}
By Corollary \ref{C:endoscopic case}, 
\[C(\theta((\sigma_1\times\sigma_2)^\sharp)_{p_1})C(\theta((\sigma_1\times\sigma_2)^\sharp)_{p_2})C(\pi) = C'(\theta((\sigma_1\times\sigma_2)^\sharp)_{p_1})C'(\theta((\sigma_1\times\sigma_2)^\sharp)_{p_2})C'(\pi).\]
On the other hand, by {\bf Step\,2},
\[C(\theta((\sigma_1\times\sigma_2)^\sharp)_{p_1})C(\theta((\sigma_1\times\sigma_2)^\sharp)_{p_2}) = C'(\theta((\sigma_1\times\sigma_2)^\sharp)_{p_1})C'(\theta((\sigma_1\times\sigma_2)^\sharp)_{p_2}).\]
Therefore (\ref{E:main ide}) holds for $\pi$.

{\bf Step\,4.} First we assume $\pi$ is of type (IIa). There exist a totally real number field $k$ and a finite place $v_0$ of $k$ such that $k_{v_0}=F$. To prove (\ref{E:main ide}), we proceed as in {\bf Step 2}. By Corollary \ref{C:analytic family} and the identity theorem for holomorphic functions, we may assume ${\rm Re}(\lambda)=0$. By \cite[Theorem 5.8]{Shin2012} (see also \cite[Theorem 1.1]{Weinstein2009}), there exists an irreducible cuspidal automorphic representation $\sigma_1$ of $\GL_2(\A_k)$ with trivial central character satisfying the following conditions:
\begin{itemize}
\item $\sigma_{1,v}$ is unramified for all finite places $v \neq v_0$.
\item $\sigma_{1,v_0} = {\rm St}\otimes {\eta^\varepsilon}$.
\item $\sigma_{1,v}$ is a discrete series representation for all real places $v$.
\end{itemize}
Let $U$ be an open neighborhood of $0$ in $\R$. By \cite[Theorem 5.8]{Shin2012}, there exists an irreducible cuspidal automorphic representation $\sigma_2$ of $\GL_2(\A_k)$ with trivial central character such that
\begin{itemize}
\item $\sigma_{2,v}$ is unramified for all finite places $v \neq v_0$.
\item $\sigma_{2,v_0} = {\rm Ind}_{B(F)}^{\GL_2(F)}(|\mbox{ }|^{\lambda+\sqrt{-1}t}\eta^\varepsilon\boxtimes|\mbox{ }|^{-\lambda-\sqrt{-1}t}{\eta^\varepsilon})$ for some $t \in U$.
\item $\sigma_{2,v}$ is a discrete series representation for all real places $v$.
\end{itemize}
Then $\theta((\sigma_1\times\sigma_2)^\sharp)$ is an irreducible globally generic cuspidal automorphic representation of $G(\A_k)$ with trivial central character satisfying the following conditions:
\begin{itemize}
\item $\theta((\sigma_1\times\sigma_2)^\sharp)_v$ is unramified for all finite places $v \neq v_0$.
\item $\theta((\sigma_1\times\sigma_2)^\sharp)_{v_0}$ is of type (IIa) with parameter $\lambda+t\sqrt{-1}$ for some $t \in U$ and sign $\varepsilon$.
\item $\theta((\sigma_1\times\sigma_2)^\sharp)_v$ is of type (DS) for all real places $v$.
\end{itemize}
By Corollary \ref{C:endoscopic case}, 
\[C(\theta((\sigma_1\times\sigma_2)^\sharp)_{v_0})\prod_{v \mid \infty}C(\theta((\sigma_1\times\sigma_2)^\sharp)_v) = C'(\theta((\sigma_1\times\sigma_2)^\sharp)_{v_0})\prod_{v \mid \infty}C'(\theta((\sigma_1\times\sigma_2)^\sharp)_v).\]
On the other hand, by {\bf Step\,3},
\[\prod_{v \mid \infty}C(\theta((\sigma_1\times\sigma_2)^\sharp)_v) = \prod_{v \mid \infty}C'(\theta((\sigma_1\times\sigma_2)^\sharp)_v).\]
We conclude that
\[C(\theta((\sigma_1\times\sigma_2)^\sharp)_{v_0}) = C'(\theta((\sigma_1\times\sigma_2)^\sharp)_{v_0}).\]
Therefore, (\ref{E:main ide}) follows from Corollary \ref{C:analytic family}. 

Finally we assume $\pi$ is of type (PS). By Corollary \ref{C:analytic family} and the identity theorem for holomorphic functions, we may assume $\pi$ is tempered, that is, ${\rm Re}(\lambda_1)={\rm Re}(\lambda_2)=0$. Let $k$ be a totally real cubic number field. Let $\{\infty_1,\infty_2,\infty_3\}$ be the set of real places of $k$. Let $U$ be an open neighborhood of $0$ in $\R$. By Proposition \ref{P:existence}, there exist irreducible cuspidal automorphic representations $\sigma_1$ and $\sigma_2$ of $\GL_2(\A_k)$ with trivial central character satisfying the following conditions:
\begin{itemize}
\item $\sigma_{i,v}$ is unramified for all finite places $v$ and $i=1,2$.
\item $\sigma_{i,\infty_1}$ and $\sigma_{i,\infty_2}$ are discrete series representations for $i=1,2$.
\item $\sigma_{1,\infty_3} = {\rm Ind}_{B(\R)}^{\GL_2(\R)}(|\mbox{ }|^{(\lambda_1+\lambda_2)/2+\sqrt{-1}t_1}{\rm sgn}^\varepsilon\boxtimes|\mbox{ }|^{(-\lambda_1-\lambda_2)/2-\sqrt{-1}t_1}{\rm sgn}^\varepsilon)$ for some $t_1 \in U$.
\item $\sigma_{2,\infty_3} = {\rm Ind}_{B(\R)}^{\GL_2(\R)}(|\mbox{ }|^{(\lambda_1-\lambda_2)/2+\sqrt{-1}t_2}{\rm sgn}^\varepsilon\boxtimes|\mbox{ }|^{(-\lambda_1+\lambda_2)/2-\sqrt{-1}t_2}{\rm sgn}^\varepsilon)$ for some $t_2 \in U$.
\end{itemize}
Then $\theta((\sigma_1\times\sigma_2)^\sharp)$ is an irreducible globally generic cuspidal automorphic representation of $G(\A_k)$ with trivial central character satisfying the following conditions:
\begin{itemize}
\item $\theta((\sigma_1\times\sigma_2)^\sharp)_v$ is unramified for all finite places $v$.
\item $\theta((\sigma_1\times\sigma_2)^\sharp)_{\infty_i}$ is of type (DS) for $i=1,2$.
\item $\theta((\sigma_1\times\sigma_2)^\sharp)_{\infty_3}$ is of type (PS) with parameter $(\lambda_1+\sqrt{-1}(t_1+t_2), \lambda_2+\sqrt{-1}(t_1-t_2))$ and sign $\varepsilon$.
\end{itemize}
By Corollary \ref{C:endoscopic case}, 
\[\prod_{i=1}^3C(\theta((\sigma_1\times\sigma_2)^\sharp)_{\infty_i}) = \prod_{i=1}^3C'(\theta((\sigma_1\times\sigma_2)^\sharp)_{\infty_i}).\]
On the other hand, by {\bf Step\,3},
\[C(\theta((\sigma_1\times\sigma_2)^\sharp)_{\infty_1})C(\theta((\sigma_1\times\sigma_2)^\sharp)_{\infty_2}) = C'(\theta((\sigma_1\times\sigma_2)^\sharp)_{\infty_1})C'(\theta((\sigma_1\times\sigma_2)^\sharp)_{\infty_2}).\]
We conclude that
\[C(\theta((\sigma_1\times\sigma_2)^\sharp)_{\infty_3}) = C'(\theta((\sigma_1\times\sigma_2)^\sharp)_{\infty_3}).\]
Therefore, (\ref{E:main ide}) follows from Corollary \ref{C:analytic family}. This completes the proof of Proposition \ref{P:main ide}.

\section{Endoscopic lifts}\label{S:Yoshida lifts}
Let $\pi$ be an irreducible globally generic cuspidal automorphic representation of $G(\A)$ with trivial central character and paramodular conductor $\frak{n}$ satisfying the conditions in \S\,\ref{SS:main thm}. In this section, we assume $\pi$ is endoscopic. Then we state
the key ingredients (Propositions \ref{P:W(1)} and \ref{P:Rallis inner product}) of the explicit Rallis inner product formula (Proposition
\ref{P:main identity2}).

\subsection{Theta lifts}\label{SS:theta lifts}
Let $(V,Q)$ be the quadratic space over $k$ defined by $V={\rm M}_{2 , 2}$ and $Q[x]=\det(x)$. Let $\iota$ be the main involution on ${\rm M}_{2 , 2}$ defined by
\[\bp x_1&x_2\\x_3&x_4\ep^{\iota}=\bp x_4 & -x_2 \\ -x_3 & x_1\ep.\] The associated symmetric bilinear form is given by $(x,y) = {\rm tr}(xy^{\iota})$. 
Put
\begin{align*}
H&={\rm GO}(V), \quad H_1 = {\rm O}(V),\\
H^{\circ} &= {\rm GSO}(V), \quad H_1^{\circ} = {\rm SO}(V).
\end{align*}
Let $\nu : H \rightarrow \mathbb{G}_m$ be the similitude character. Let ${\bf t}$ be the involution on $\GL_2\times\GL_2$ defined by 
\[{\rm Ad}({\bf t})(h_1,h_2) = \left((h_2^{\iota})^{-1},(h_1^{\iota})^{-1}\right).\]
Put ${\pmb \mu}_2 = \<{\bf t}\>$. We have an exact sequence
\[1 \longrightarrow   {\mathbb G}_m \stackrel{\Delta}{\longrightarrow} (\GL_2\times \GL_2)\rtimes {\pmb \mu}_2 \stackrel{\rho}{\longrightarrow} H \longrightarrow 1,\]
where $\Delta(a)=(a{\bf 1}_2,a{\bf 1}_2)$, $\rho(h_1,h_2){x}=h_1{ x}h_2^{-1}$, and $\rho({\bf t})x=x^{\iota}$ for $a \in {\mathbb G}_m, h_1, h_2 \in \GL_2$, and $x \in V$. For $h_1,h_2 \in \GL_2$, we write $\rho(h_1,h_2) = [h_1,h_2]$. Note that $\nu([h_1,h_2])=\det(h_1h_2^{-1})$. We denote by ${\bf t}$ the image of $\bf t$ in $H$ by abuse of notation, and by ${\bf t}_v$ the image of $\bf t$ in $H(k_v)$ for each place $v$ of $k$.

Let $\omega=\omega_{\psi,V,2}$ be the Weil representation of $\Sp_4(\A)\times H_1(\A)$ on $\calS(V^2(\A))$ with respect to $\psi$. Let $S(V^2(\A))$ be the subspace of $\calS(V^2(\AA))$
consisting of functions which correspond to polynomials in the Fock model
at the archimedean places. We extend $\omega$ to a representation of ${\rm G}(\Sp_4 \times H_1)(\A)$ as in (\ref{Weil rep on R}).

Let $\varphi \in S(V^2(\A))$. The theta function associated to $\varphi$ is defined by
\[\Theta(g,h;\varphi) = \sum_{x \in V^2(k)}\omega(g,h)\varphi(x)\]
for $(g,h) \in {\rm G}(\Sp_4 \times H_1)(\A)$. Let $f$ be a cusp form on $H(\A)$ and let $\varphi \in S(V^2(\A))$. For $g \in G(\A)$, choose $h \in H(\A)$ such that $\nu(h)=\nu(g)$, and put
\[\theta(f,\varphi)(g) = \int_{H_1(k)\backslash H_1(\A)}f(h_1h)\Theta(g,h_1h;\varphi)\,dh_1.\]
Then $\theta(f,\varphi)$ is an automorphic form on $G(\A)$. For any irreducible cuspidal automorphic representation $\sigma^{\sharp}$ of $H(\A)$ on the space $\mathcal{V}_{\sigma^\sharp}$, we define $\theta(\sigma^{\sharp})$ as the automorphic representation of $G(\A)$ on the space $\mathcal{V}_{\theta(\sigma^{\sharp})}$ generated by $\theta(f,\varphi)$ for all $f \in \mathcal{V}_{\sigma^{\sharp}}$ and $\varphi \in S(V^2(\A))$. 

Since we assume $\pi$ is endoscopic, by \cite[Proposition 2.2]{AH2006}, there exists an irreducible unitary cuspidal automorphic representation $\sigma^{\sharp}$ of $H(\A)$ such that
\[\theta(\sigma^{\sharp}) = \pi.\]
By \cite[Lemma 2]{HST1993}, the tower property, and the cuspidality of $\pi$, we have
\begin{align}\label{E:restriction}
\mathcal{V}_{\sigma^{\sharp}}\vert_{H^\circ(\A)} = \mathcal{V}_\sigma \oplus \mathcal{V}_{\sigma\circ{\rm Ad}({\bf t})}
\end{align}
as spaces of functions on $H^\circ(\A)$ for some irreducible unitary cuspidal automorphic representation $\sigma$ of $H^\circ(\A)$ on the space $\mathcal{V}_\sigma$ with $\sigma \neq \sigma\circ{\rm Ad}({\bf t})$. Via the isomorphism
\[H^\circ \simeq \Delta\mathbb{G}_m\backslash(\GL_2\times\GL_2)\]
induced by $\rho$, we have
\[\sigma = \sigma_1 \times \sigma_2,\quad \mathcal{V}_\sigma = \mathcal{V}_{\sigma_1}\otimes \mathcal{V}_{\sigma_2}\]
for some irreducible unitary cuspidal automorphic representations $\sigma_1$ and $\sigma_2$ of $\GL_2(\A)$ on the spaces $\mathcal{V}_{\sigma_1}$ and $\mathcal{V}_{\sigma_2}$, respectively, with central characters inverse to each other. Let $\frak{n}_1$ and $\frak{n}_2$ be the conductors of $\sigma_1$ and $\sigma_2$, respectively. By our assumptions on $\pi$ in \S\,\ref{SS:main thm} and the condition $\sigma \neq \sigma\circ{\rm Ad}({\bf t})$, $\sigma_{1}$ and $\sigma_2$ satisfy the following properties:
\begin{itemize}
\item $\sigma_1 \neq \sigma_2^\vee$.
\item $\sigma_1$ and $\sigma_2$ have trivial central characters.
\item $\frak{n}_1$ and $\frak{n}_2$ are square-free and coprime, and $\frak{n}_1\frak{n}_2=\frak{n}$.
\item For $v \in S({\rm DS})$, $\sigma_{1,v}$ and $\sigma_{2,v}$ are discrete series representations.
\item For $v \in S({\rm PS})$, $\sigma_{1,v}$ and $\sigma_{2,v}$ are unitary principal series representations with non-zero ${\rm SO}(2)$-invariant vectors.
\end{itemize}
For $v \in S({\rm DS})$, let $\kappa_{1,v} \in \Z_{\geq 1}$ and $\kappa_{2,v}\in \Z_{\geq 1}$ be the minimal weights of $\sigma_{1,v}$ and $\sigma_{2,v}$, respectively. Then
\[
\lambda_{1,v} = \frac{\kappa_{1,v}+\kappa_{2,v}}{2}, \quad \lambda_{2,v} =  -\frac{|\kappa_{1,v}-\kappa_{2,v}|}{2}.
\]
For $v \in S({\rm PS})$, we have 
\[\sigma_{i,v} = {\rm Ind}_{B(k_v)}^{\GL_2(k_v)}(|\mbox{ }|_v^{\mu_{i,v}}{\rm sgn}^\varepsilon\boxtimes |\mbox{ }|_v^{-\mu_{i,v}}{\rm sgn}^\varepsilon)\]
for some $\varepsilon \in \{0,1\}$ and $\mu_{i,v}\in \C$ with $|{\rm Re}(\mu_{i,v})| < 1/2$ for $i=1,2$. Then
\[
\{ \lambda_{1,v} ,\,\lambda_{2,v}\}= \{ \mu_{1,v} + \mu_{2,v},\,\mu_{1,v} - \mu_{2,v}\},% \quad
% \lambda_{2,v} = \mu_{1,v} - \mu_{2,v}, 
\]
after replacing $\lambda_{i,v}$ with $-\lambda_{i,v}$ or $\mu_{i,v}$ with $-\mu_{i,v}$ if necessary.

\subsection{Automorphic forms on ${\rm GO}(V)$}

Consider non-zero automorphic forms ${\bf f}_1 \in \mathcal{V}_{\sigma_1}$ and ${\bf f}_2 \in \mathcal{V}_{\sigma_2}$ satisfying the following conditions:
\begin{itemize}
\item If $v$ is a finite place and $v \nmid \frak{n}$, then $$\sigma_{1,v}(k_1){\bf f}_1 = {\bf f}_1,\quad \sigma_{2,v}(k_2){\bf f}_2 = {\bf f}_2$$
 for $(k_1,k_2) \in \GL_2(\frak{o}_v)\times \GL_2(\frak{o}_v)$.
\item If $v \mid \frak{n}_1$, then
$$\sigma_{1,v}(k_1){\bf f}_1 = {\bf f}_1,\quad \sigma_{2,v}(k_2){\bf f}_2 = {\bf f}_2$$
for $(k_1,k_2) \in K_0(\varpi_v)\times \GL_2(\frak{o}_v)$.
\item If $v \mid \frak{n}_2$, then
$$\sigma_{1,v}(k_1){\bf f}_1 = {\bf f}_1,\quad \sigma_{2,v}(k_2){\bf f}_2 = {\bf f}_2$$
for $(k_1,k_2) \in \GL_2(\o_v)\times K_0(\varpi_v)$.
\item If $v \in S({\rm DS})$, then
$$\sigma_{1,v}(k_{\theta}){\bf f}_1 = e^{\sqrt{-1}\kappa_{1,v} \theta} {\bf f}_1,\quad \sigma_{2,v}(k_{\theta}){\bf f}_2 = e^{\sqrt{-1}\kappa_{2,v} \theta} {\bf f}_2$$
for $k_{\theta} \in {\rm SO}(2)$.
\item If $v \in S({\rm PS})$, then
$$\sigma_{1,v}(k){\bf f}_1 = {\bf f}_1,\quad \sigma_{2,v}(k){\bf f}_2 = {\bf f}_2$$
for $k \in {\rm SO}(2)$.
\end{itemize} 
The conditions above characterize ${\bf f}_1$ and ${\bf f}_2$ up to scalars. We normalize ${\bf f}_1$ and ${\bf f}_2$ so that
\[
W_{{\bf f}_i}\left( \prod_{v \nmid \infty}{\bf a}(\varpi_v^{-\c_v})\right)  = e^{-2\pi|S({\rm DS})|}\prod_{v \in S({\rm PS})}K_{\mu_{i,v}}(2\pi),
\]
where $W_{{\bf f}_i}$ is the Whittaker function of ${\bf f}_i$ defined by 
\[W_{{\bf f}_i}(g) = \int_{k \backslash \A}{\bf f}_i({\bf n}(x)g)\overline{\psi(x)}\,dx.\]
Here $dx$ is the Tamagawa measure on $\A$. Let ${\bf f}\in \mathcal{V}_{\sigma}$ be the automorphic form defined by
\[
{\bf f}(h)={\bf f}_1(h_1){\bf f}_2(h_2)
\]
for $h=[h_1,h_2] \in H^{\circ}(\A)$.

Consider a non-zero automorphic form ${\bf f}^{\sharp} \in \mathcal{V}_{\sigma^{\sharp}}$ satisfying the following conditions:
\begin{itemize}
\item If $v$ is a finite place and $v\nmid \frak{n}$, then
$$\sigma^{\sharp}_v(k'){\bf f}^{\sharp}={\bf f}^{\sharp}$$
for $k' \in H(\o_v)$. 
%\item If $v \mid \frak{d}$ and $v \nmid \frak{n}$, then
%$$\sigma^{\sharp}_v([k_1,k_2]){\bf f}^{\sharp}={\bf f}^{\sharp},$$
%for $(k_1,k_2) \in \GL_2(\o_v)\times {\bf a}(\varpi_v^{-\c_v})\GL_2(\frak{o}_v){\bf a}(\varpi_v^{\c_v})$.
\item If $v \mid \frak{n}_1$, then 
$$\sigma_v^{\sharp}([k_1,k_2]){\bf f}^{\sharp}={\bf f}^{\sharp}$$ 
for $(k_1,k_2) \in K_0(\varpi_v)\times \GL_2(\frak{o}_v)$.
\item If $v \mid \frak{n}_2$, then
$$\sigma^{\sharp}_v([k_1,k_2]){\bf f}^{\sharp}={\bf f}^{\sharp}$$ for $(k_1,k_2) \in \GL_2(\o_v)\times K_0(\varpi_v)$.
\item If $v \in S({\rm DS})$, then 
$$\sigma^{\sharp}_v([k_{\theta_1},k_{\theta_2}]){\bf f}^{\sharp} = e^{\sqrt{-1}(\kappa_{1,v}\theta_1+\kappa_{2,v}\theta_2)}{\bf f}^{\sharp}$$
for $k_{\theta_1},k_{\theta_2} \in {\rm SO}(2)$.
\item If $v \in S({\rm PS})$, then
$$\sigma^{\sharp}_v([k_1,k_2]){\bf f}^{\sharp}={\bf f}^{\sharp}, \quad \sigma^{\sharp}_v({\bf t}_v){\bf f}^{\sharp}={\bf f}^{\sharp}$$
for $k_1,k_2 \in {\rm SO}(2)$.
\end{itemize}
The conditions above characterize ${\bf f}^{\sharp}$ up to scalars. We normalize ${\bf f}^\sharp$ so that
$${\bf f}^{\sharp}\vert_{H^\circ(\A)} = {\bf f}.$$

Let $L$ be the lattice of $V^2(k)$ defined by
\[L=\bp \frak{n}_2 & \frak{o} \\ \frak{n} & \frak{n}_1 \ep \oplus \bp \frak{o} & \frak{o} \\ \frak{o} & \frak{o} \ep.\]
Define $\varphi  = \bigotimes_v \varphi_v \in S(V^2(\A))$ as follows:
If $v$ is a finite place, then 
\begin{align}\label{E:non-archi Schwartz fun}
\varphi_v = \mathbb{I}_{L\otimes_\o \o_v}.
\end{align}
If $v \in S({\rm DS})$ and $\kappa_{1,v}\geq \kappa_{2,v}$, then
\begin{align}\label{E:archi Schwartz fun DS1}
\varphi_v(x,y) =(-\sqrt{-1}x_1-x_2-x_3+\sqrt{-1}x_4)^{\lambda_{1,v}}(y_1+\sqrt{-1}y_2-\sqrt{-1}y_3+y_4)^{-\lambda_{2,v}}e^{-\pi\,{\rm tr}(x{}^t \! x+y{}^t \! y)}.
\end{align}
If $v \in S({\rm DS})$ and $\kappa_{1,v}\leq \kappa_{2,v}$, then
\begin{align}\label{E:archi Schwartz fun DS2}
\varphi_v(x,y) =(-\sqrt{-1}x_1-x_2-x_3+\sqrt{-1}x_4)^{\lambda_{1,v}}(y_1-\sqrt{-1}y_2+\sqrt{-1}y_3+y_4)^{-\lambda_{2,v}}e^{-\pi\,{\rm tr}(x{}^t \! x+y{}^t \! y)}.
\end{align}
If $v \in S({\rm PS})$, then
\begin{align}\label{E:archi Schwartz fun PS}
\varphi_v(x,y) = e^{-\pi\,{\rm tr}(x{}^t \! x+y{}^t \! y)}.
\end{align}
Let $\frak{S}$ be the subset of places of $k$ defined by
\[\frak{S} = \{v \mbox{ }\vert\mbox{ }\sigma_v \not\simeq \sigma_v\circ{\rm Ad}({\bf t}_v)\}.\]
Since $\sigma \neq \sigma \circ {\rm Ad}(\bf t)$, $\frak{S}$ contains infinitely many places of $k$ by the strong multiplicity one theorem for $H^\circ$. Recall $S({\rm DS})$ is the set of real places of type (DS). Let $S(\frak{n}) = \{v \mid \frak{n}\}$. Note that $S(\frak{n})\subseteq \frak{S}$.

\begin{prop}\label{P:W(1)}
Let $f \in \pi$ be the cusp form satisfying the conditions (\ref{E:K-type2}) and (\ref{E:normalization}).
We have
\[\theta({\bf f}^{\sharp},\varphi)(g) = 2^{-|(S({\rm DS})\cup S(\frak{n}))\cap \frak{S}|}\cdot\prod_{v \in S({\rm DS})}2^{-\lambda_{1,v}- 4}\pi^{(-3\lambda_{1,v}+\lambda_{2,v}-5)/2}\cdot\prod_{v \mid \frak{n}}(1+q_v)^{-1}\cdot\frak{D}^{-3/2}\cdot{\xi(2)^{-2}}\cdot f(gg_0)\]
for $g \in G(\A)$. Here $g_0 \in G(\A)$ is defined by
\[
g_{0,v} = \begin{cases}
\diag( 1, 1, \varpi_v^{\frak{c}_v}, \varpi_v^{\frak{c}_v} ) & \mbox{ if $v \nmid \infty$},\\
1 & \mbox{ if $v \mid \infty$}.
\end{cases}
\]
\end{prop}

\begin{proof}
The proof will be given in $\S\,\ref{S:proof 1}$ below.
\end{proof}

\begin{prop}\label{P:Rallis inner product}
We have
\begin{align*}
\left\<\theta({\bf f}^{\sharp},\varphi),\theta({\bf f}^{\sharp},\varphi)\right\> &= 2^{-2|(S({\rm DS})\cup S(\frak{n}))\cap \frak{S}|}\cdot\frak{D}^{-8}\cdot\frac{4}{\xi(2)^4}\cdot\frac{L(1,\pi,{\rm Ad})}{\Delta_{{\rm PGSp}_4}}\cdot \prod_{v\mid \frak{n}}q_v^{-3}\zeta_v(1)^{-2}\zeta_v(2)\zeta_v(4)\\
&\times \prod_{v \in S({\rm DS})} 2^{-\lambda_{1,v}-\lambda_{2,v}-3}(1+\lambda_{1,v}-\lambda_{2,v})^{-1}\cdot \prod_{v \in S({\rm PS})} 2^{-4}.
\end{align*}
\end{prop}

\begin{proof}
The proof will be given in $\S\,\ref{S:proof 2}$ below.
\end{proof}

\section{Calculation of Whittaker functions}\label{S:cal1}
We keep the notation of $\S\,\ref{S:Yoshida lifts}$. The aim of this section is to prove Proposition \ref{P:W(1)}. We prove  it by comparing the Whittaker function of $\theta({\bf f}^\sharp,\varphi)$ and the normalization of $f$ in (\ref{E:normalization}). It boils down to the calculation of certain local Whittaker functions of $\pi_v$ in terms of local integrals involving Whittaker functions of $\sigma_v$ and the Weil representation $\omega_v$ for each place $v$.

\subsection{Measures}\label{SS:measure} Let $v$ be a place of $k$. Let $K_v$ be the maximal compact subgroup of $H_1^{\circ}(k_v)$ defined by
\[
K_v = \begin{cases}
[{\bf d}(\varpi_v^\c),{\bf d}(\varpi_v^\c)]H_1^{\circ}(\o_v)[{\bf d}(\varpi_v^{-\c}),{\bf d}(\varpi_v^{-\c})]   & \mbox{ if $v$ is finite},\\
H_1^{\circ}(k_v) \cap ({\rm O}(2) \times {\rm O}(2))/\{\pm 1\} & \mbox{ if $v$ is real}.
\end{cases}
\]
We fix a Haar measure $dh_{1,v}$ on $H_1^{\circ}(k_v)$ as follows:
\begin{itemize}
\item If $v$ is finite, the measure is normalized so that 
\begin{align}\label{E:non-archi. integration1}
{\rm vol}(K_v,dh_{1,v})=1.
\end{align} For $\phi \in L^1(H_1^{\circ}(k_v))$, we have
\begin{align}\label{E:non-archi. integration2}
\begin{split}
&\int_{H_1^{\circ}(k_v)}\phi(h_{1,v})\,dh_{1,v}\\
& = q_v^{-2\c_v}\int_{k_v}\int_{k_v}\int_{k_v^{\times}}\int_{k_v^{\times}}\int_{K_v}\phi([{\bf n}(x_1),{\bf n}(x_2)][{\bf m}(y_1),{\bf m}(y_2)]k)|y_1|_{v}^{-2}|y_2|_{v}^{-2}\,dk\,d^{\times}y_1\,d^{\times}y_2\,dx_1\,dx_2\\
&+ q_v^{-2\c_v}\int_{k_v}\int_{k_v}\int_{k_v^{\times}}\int_{k_v^{\times}}\int_{K_v}\phi([{\bf a}(\varpi_v),{\bf a}(\varpi_v)][{\bf n}(x_1),{\bf n}(x_2)][{\bf m}(y_1),{\bf m}(y_2)]k)|y_1|_{v}^{-2}|y_2|_{v}^{-2}\,dk\,d^{\times}y_1\,d^{\times}y_2\,dx_1\,dx_2.
\end{split}
\end{align}
\item If $v$ is real, the measure is defined so that for all $\phi\in L^1(H_1^{\circ}(\R))$, we have
\begin{align}\label{E:archi. integration}
\begin{split}
\int_{H_1^{\circ}(\R)}\phi(h_{1,\infty})\,dh_{1,\infty}&=\int_{H_1^{\circ}(\R)^0}\phi(h_{1,\infty})dh_{1,\infty}+\int_{H_1^{\circ}(\R)^0}\phi([{\bf a}(-1),{\bf a}(-1)]h_{1,\infty})\,dh_{1,\infty},\\
\int_{H_1^{\circ}(\R)^0}\phi(h_{1,\infty})\,dh_{1,\infty} &=4\int_{\R}\int_{\R}\int_{0}^{\infty}\int_{0}^{\infty}\int_{K_{\infty}^0}\phi([{\bf n}(x_1),{\bf n}(x_2)][{\bf m}(y_1),{\bf m}(y_2)]k)y_1^{-2}y_2^{-2}\,dk\,d^{\times}y_1\,d^{\times}y_2\,dx_1\,dx_2
\\
&= 16 \pi^2 \int_{K_{\infty}^0}\int_{0}^{\infty}\int_{0}^{\infty}\int_{K_{\infty}^0}\phi(k[{\bf m}(e^{t_1}),{\bf m}(e^{t_2})]k')\sinh(2t_1)\sinh(2t_2)\,dk\,dt_1\,dt_2\,dk'.
\end{split}
\end{align} 
Here $K_{\infty}^0$ and $H_1^{\circ}(\R)^0$ are the topological identity components of $K_{\infty}$ and $H_1^{\circ}(\R)$, respectively, and ${\rm vol}(K_{\infty}^0,dk)={\rm vol}(K_{\infty}^0,dk')=1$. We refer to \cite[\S 12]{IchinoIkeda2010} for the measure with respect to the Cartan decomposition.
\end{itemize}
Let $d\epsilon_v$ be the Haar measure on ${\pmb \mu}_2(k_v)$ such that ${\rm vol}({\pmb \mu}_2(k_v),d\epsilon_v)=1$. We extend the measure on $H_1^{\circ}(k_v)$ to a measure on $H_1(k_v)$ by
\[
\int_{H_1(k_v)}\phi(h_{1,v})\,dh_{1,v}  =  \int_{{\pmb \mu}_2(k_v)}\int_{H_1^{\circ}(k_v)}\phi(h_{1,v}\epsilon_v)\,dh_{1,v}\,d\epsilon_v
\]
for $\phi \in L^1(H_1(k_v))$.

Let $dh_1$ be the Tamagawa measure on $H_1(\A)$. Note that ${\rm vol}(H_1(k)\backslash H_1(\A))=1$ and 
\begin{align}\label{E:ratio of measures}
dh_1 = \frak{D}^{-3}\cdot\xi(2)^{-2}\cdot\prod_{v}dh_{1,v}.
\end{align}
Let $d\epsilon =\prod_v d\epsilon_v$ be the Tamagawa measure on ${\pmb \mu}_2(\A)$.

\subsection{Whittaker functions}
Recall $\frak{S}$ is the subset of places of $k$ defined by
\[\frak{S} = \{v \mbox{ }\vert\mbox{ }\sigma_v \not\simeq \sigma_v\circ{\rm Ad}({\bf t}_v)\}.\]
Let $v$ be a place of $k$. Let $\mathcal{V}_{1,v}$ and $\mathcal{V}_{2,v}$ be the spaces of $\sigma_{1,v}$ and $\sigma_{2,v}$, respectively. Then $\mathcal{V}_v = \mathcal{V}_{1,v}\otimes\mathcal{V}_{2,v}$ is the space of $\sigma_v$. We define an irreducible admissible representation $\sigma_v^{\sharp}$ of $H(k_v)$ as follows:
\begin{itemize}
\item $v \notin \frak{S}$: In this case, $\sigma_{1,v} \simeq \sigma_{2,v}$. We take $\mathcal{V}_{1,v} = \mathcal{V}_{2,v}$ and define $\mathcal{V}_v^{\sharp} = \mathcal{V}_v$. The representation $\sigma_v^{\sharp}$ of $H(k_v)$ on $\mathcal{V}_v^{\sharp}$ is defined by
\begin{align*}
\begin{split}
\sigma_v^{\sharp}([h_1,h_2])(v_1\otimes v_2) &= \sigma_{1,v}(h_1)v_1\otimes \sigma_{2,v}(h_2)v_2,\\
\sigma_v^{\sharp}({\bf t}_v)(v_1\otimes v_2) & = v_2\otimes v_1
\end{split}
\end{align*}
for $h_1,h_2 \in \GL_2(k_v)$ and $v_1\otimes v_2 \in \mathcal{V}_v^{\sharp}$.
\item $v \in \frak{S}$: In this case, define $\sigma_v^{\sharp}= {\rm Ind}_{H^{\circ}(k_v)}^{H(k_v)}\sigma_v$. The induced representation is realized on $\mathcal{V}_v^{\sharp} = \mathcal{V}_v \oplus \mathcal{V}_v$, where the action is defined by
\begin{align*}
\begin{split}
\sigma_v^{\sharp}(h)(v_1, v_2) &= (\sigma_v(h)v_1 ,\sigma_v({\rm Ad}({\bf t}_v)h)v_2),\\
\sigma_v^{\sharp}({\bf t}_v)(v_1, v_2) &= (v_2, v_1)
\end{split}
\end{align*}
for $h \in H^{\circ}(k_v)$ and $(v_1, v_2) \in \mathcal{V}_v^{\sharp}.$
\end{itemize}
 
\subsubsection{Local models} 
Let $v$ be place of $k$. For $i=1,2$, let $\mathcal{V}_{i,v}=\mathcal{W}(\sigma_{i,v},\psi_v)$ be the space of Whittaker functions of $\sigma_{i,v}$ with respect to $\psi_v$. Consider a non-zero Whittaker function $W_{1,v} \in \mathcal{V}_{1,v}$ and $W_{2,v} \in \mathcal{V}_{2,v}$ satisfying the following conditions:
\begin{itemize}
\item If $v$ is a finite place and $v \nmid \frak{n}$, then $$\sigma_{1,v}(k_1){W}_{1,v} = {W}_{1,v},\quad \sigma_{2,v}(k_2){W}_{2,v} = {W}_{2,v}$$
 for $(k_1,k_2) \in \GL_2(\frak{o}_v)\times \GL_2(\frak{o}_v)$.
\item If $v \mid \frak{n}_1$, then
$$\sigma_{1,v}(k_1){W}_{1,v} = {W}_{1,v},\quad \sigma_{2,v}(k_2){W}_{2,v} = {W}_{2,v}$$
for $(k_1,k_2) \in K_0(\varpi_v)\times \GL_2(\frak{o}_v)$.
\item If $v \mid \frak{n}_2$, then
$$\sigma_{1,v}(k_1){W}_{1,v} = {W}_{1,v},\quad \sigma_{2,v}(k_2){W}_{2,v} = {W}_{2,v}$$
for $(k_1,k_2) \in \GL_2(\o_v)\times K_0(\varpi_v)$.
\item If $v \in S({\rm DS})$, then
$$\sigma_{1,v}(k_{\theta}){W}_{1,v} = e^{\sqrt{-1}\kappa_{1,v} \theta} {W}_{1,v},\quad \sigma_{2,v}(k_{\theta}){W}_{2,v} = e^{\sqrt{-1}\kappa_{2,v} \theta} {W}_{2,v}$$
for $k_{\theta} \in {\rm SO}(2)$.
\item If $v \in S({\rm PS})$, then
$$\sigma_{1,v}(k){W}_{1,v} = {W}_{1,v},\quad \sigma_{2,v}(k){W}_{2,v} = {W}_{2,v}$$
for $k \in {\rm SO}(2)$.
\end{itemize}
The conditions above characterize $W_{1,v}$ and $W_{2,v}$ up to scalars. We normalize $W_{1,v}$ and $W_{2,v}$ as follows:
\begin{itemize}
\item If $v$ is a finite place, then
$$W_{1,v}({\bf a}(\varpi_v^{-\c_v}))=W_{2,v}({\bf a}(\varpi_v^{-\c_v}))=1.$$
\item If $v \in S({\rm DS})$, then
$$W_{1,v}(1)=W_{2,v}(1)=e^{-2\pi}.$$
\item If $v \in S({\rm PS})$, then
$$W_{1,v}(1)=K_{\mu_{1,v}}(2\pi),\quad W_{2,v}(1)=K_{\mu_{2,v}}(2\pi).$$
\end{itemize}
Let $W_v$ be the Whittaker function on $H^{\circ}(k_v)$ defined by 
\begin{align}\label{E:Whittaker newform}
W_v([h_1,h_2])=W_1(h_1)W_2(h_2)
\end{align}
for $h_1,h_2 \in \GL_2(k_v)$.

Fix isomorphisms $\sigma \simeq \bigotimes_v \sigma_{v}$ and $\sigma^{\sharp}\simeq \bigotimes_v \sigma_v^{\sharp}$ such that
\begin{align}\label{E:iso}
\begin{split}
{\bf f} &\longmapsto \bigotimes_vW_v,\\
{\bf f}^{\sharp} & \longmapsto \left(\bigotimes_{\scriptstyle{v \in S({\rm DS})\cup S(\frak{n})} \atop \scriptstyle{v\in\frak{S}} }(W_v,0)\right)\otimes \left(\bigotimes_{\scriptstyle{v \notin S({\rm DS})\cup S(\frak{n})}\atop\scriptstyle{ v \in \frak{S}} }(W_v,W_v)\right) \otimes \left(\bigotimes_{v \notin \frak{S}} W_v\right).
\end{split}
\end{align}
\subsubsection{Global Whittaker functions}
If $R$ is a set of places of $k$, denote by ${\bf t}_R \in {\pmb \mu}_2(\A)$ the element such that
\[
 {\bf t}_{R,v}=
 \begin{cases}  
{\bf t}_v   & \mbox{ if $v \in R$},\\
1 & \mbox{ if $v \notin R$}.
 \end{cases} 
\]
For an automorphic form $f$ on $H(\A)$ and a set $R$ of places of $k$, let $f_R$ be the automorphic form on $H^\circ(\A)$ defined by
\[f_R(h)=f(h{\bf t}_R)\]
for $h \in H^\circ(\A)$.

Let $f$ be a cusp form on $H^\circ(\A)$ and let $\varphi \in S(V^2(\A))$. Define Whittaker functions $W_f$ and $W_{\theta(f,\varphi)}$ by
\begin{align*}
W_{f}(h) &= \int_{(k\backslash \A)^2}f([{\bf n}(x),{\bf n}(y)]h)\overline{\psi(x+y)}\,dx\,dy,\\
 W_{\theta(f, \varphi)}(g) &= \int_{U(k) \backslash U(\A)} \theta(f, \varphi)(ug)
 \overline{\psi_U(u)} \, du
\end{align*}
for $g \in G(\A)$ and $h \in H^\circ(\A)$.
Here $dx,dy,du$ are the Tamagawa measures.

Let
\[
\Delta N = \{ [{\bf n}(x),{\bf n}(-x)] \in H_1^{\circ} \mbox{ }\vert\mbox{ }x \in \mathbb{G}_a\}.
\]
In the following lemma, we establish a formula for the Whittaker functions of theta lifts. The result is a variant of \cite[(5.18)]{HPS1983}, which considers Whittaker functions of theta lifts to $\Sp_{2n}(\A)$.

\begin{lemma}\label{L:Whittaker function} Let $\varphi \in S(V^2(\A))$ and let $f$ be a cusp form on $H(\A)$. For $(g,h) \in {\rm G}(\Sp_4 \times H_1)(\A)$ with $h \in H^{\circ}(\A)$, we have
\[
 W_{\theta(f, \varphi)}(g) = 2^{-1-|T|}\sum_{R\subseteq T}
 \int_{\Delta N(\A) \backslash H_1^{\circ}(\A)} W_{f_R}(h_1\cdot{\rm Ad}({\bf t}_R)h)
 \omega(g, h_1{\bf t}_R h) \varphi ({\bf x}_0,{\bf y}_0) \, dh_1,
\]
where $T$ is a sufficiently large finite set of places of $k$, and
\[
 {\bf x}_0 =
 \begin{pmatrix}
  0 & -1 \\
  0 & 0
 \end{pmatrix},\quad {\bf y}_0={\bf a}(-1).
\]
\end{lemma}

\begin{proof}
Note that
\[
 \theta(f, \varphi)(g) = 2^{-1} 
 \int_{{\pmb \mu}_2(\A)}\int_{H_1^{\circ}(k) \backslash H_1^{\circ}(\A)} f(h_1\epsilon h) \Theta(g, h_1\epsilon h; \varphi) \, dh_1\,d\epsilon.
\]
Since 
\[
 \omega(u, 1) \varphi(x, y) = \varphi(x, xu_0 + y)
 \psi(\det(x) u_1 + (x, y) u_2 + \det(y) u_3)
\]
for 
\[
u = \begin{pmatrix}
 1 & 0 & u_1 & u_2 \\
 0 & 1 & u_2 & u_3 \\
 0 & 0 & 1 & 0 \\
 0 & 0 & 0 & 1
\end{pmatrix}
\begin{pmatrix}
 1 & u_0 & 0 & 0 \\
 0 & 1 & 0 & 0 \\
 0 & 0 & 1 & 0 \\
 0 & 0 & -u_0 & 1
\end{pmatrix},
\]
we have
\begin{align*}
W_{\theta(f, \varphi)}(g) &= 2^{-1} \int_{{\pmb \mu}_2(\A)}\int_{H_1^{\circ}(k) \backslash H_1^{\circ}(\A)} \int_{(k \backslash \A)^4}
 f(h_1\epsilon h ) \sum_{x, y \in V(k)} \omega(g, h_1\epsilon h) \varphi(x, x u_0 + y) \\
 & \quad\quad\quad\quad\quad\quad\quad\times \psi(\det(x) u_1 + (x, y) u_2 + (\det(y) + 1) u_3 + u_0) \, du \, dh_1\,d\epsilon.
\end{align*}
Hence, putting
\begin{align*}
 X_0 & = \left\{ (0, y) \in V^2 \, | \, \det(y) = -1 \right\}, \\
 X_1 & = \left\{ (x, y) \in V^2 \, | \, \det(x) = 0, x \ne 0, 
 (x, y) = 0, \det(y) = -1 \right\},
\end{align*}
and
\[
 I_i = \int_{{\pmb \mu}_2(\A)}\int_{H_1^{\circ}(k) \backslash H_1^{\circ}(\A)} \int_{k \backslash \A}
 f(h_1\epsilon h) \sum_{(x, y) \in X_i(k)} \omega(g, h_1\epsilon h) \varphi(x, x u_0 + y) \psi(u_0)
 \, du_0 \, dh_1\,d\epsilon,
\]
we have
\[
 W_{\theta(f, \varphi)}(g) = 2^{-1}(I_0+I_1).
\]
In fact, we have
\[
 I_0 = \int_{{\pmb \mu}_2(\A)}\int_{H_1^{\circ}(k) \backslash H_1^{\circ}(\A)} \int_{k \backslash \A}
 f(h_1\epsilon h) \sum_{y \in \SL_2(k)} \omega(g, h_1\epsilon h) \varphi(0, y{\bf a}(-1)) \psi(u_0)
 \, du_0 \, dh_1\,d\epsilon = 0
\]
since
\[
\int_{k \backslash \A}\psi(u_0)\,du_0=0.
\]
It remains to compute $I_1$.

We claim that the map $h_1 \mapsto h_1 ({\bf x}_0, {\bf y}_0)$ induces an isomorphism
\[
 \Delta N \backslash H_1^{\circ} \longrightarrow X_1.
\]
Obviously, $H_1^{\circ}$ acts on $X_1$.
We first show that this action is transitive.
Let $(x, y) \in X_1$ and write \[y = \bp y_1 & y_2 \\ y_3 & y_4\ep.\]
Since the rank of $x$ is $1$, there exists $h_1 \in H_1^{\circ}$ such that 
$h_1x = {\bf x}_0$.
Hence we may assume that $x = {\bf x}_0$.
Then $(x, y) = y_3 = 0$.
Thus we have
\[
 y =
 \begin{pmatrix}
  y_1 & y_2 \\
  0 & -y_1^{-1}
 \end{pmatrix}.
\]
Put 
\[
 h_1 = \left[
 \begin{pmatrix}
  1 &  y_1 y_2 \\
  0 & -y_1
 \end{pmatrix},
 \begin{pmatrix}
  -y_1 & 0 \\
  0 & 1
 \end{pmatrix}
 \right].
\]
Then we have $h_1({\bf x}_0, y) = ({\bf x}_0,{\bf y}_0)$.
We next show that the stabilizer of this action is 
$ \Delta N$ in $H_1^{\circ}$.
Obviously, $\Delta N$ fixes $({\bf x}_0,{\bf y}_0)$.
Assume that $h_1({\bf x}_0,{\bf y}_0) = ({\bf x}_0,{\bf y}_0)$ with $h_1 = [g_1, g_2] \in H_1^{\circ}$.
Then $(g_1 {\bf x}_0 g_2^{-1}, g_1 {\bf y}_0g_2^{-1}) = ({\bf x}_0,{\bf y}_0)$.
Write
\[
 g_1 =
 \begin{pmatrix}
  a & b \\
  c & d
 \end{pmatrix}.
\]
It follows from $g_1 \mathbf{y}_0 = \mathbf{y}_0 g_2$ that \[g_2 = \begin{pmatrix} a & -b \\ -c & d \end{pmatrix},\] which combined with $g_1 \mathbf{x}_0 = \mathbf{x}_0 g_2$ implies that $a=d$, $c=0$.
Hence $h_1 \in \Delta N$.
This completes the proof of the claim.

If follows from the claim that
\[
 I_1 = \int_{{\pmb \mu}_2(\A)}\int_{\Delta N(\A) \backslash H_1^{\circ}(\A)}
 \int_{\Delta N(k) \backslash \Delta N(\A)} \int_{k \backslash \A}
 f(n h_1\epsilon h) \omega(g, h_1\epsilon h) \varphi({\bf x}_0, {\bf x}_0 u_0 + {\bf y}_0) \psi(u_0)
 \, du_0 \, dn \, dh_1\,d\epsilon. 
\]
Since 
\[
 {\bf x}_0 u_0 + {\bf y}_0 = 
 {\bf n}(-u_0){\bf y}_0,
\]
we have
\begin{align*}
 & \int_{\Delta N(k) \backslash \Delta N(\A)} \int_{k \backslash \A}
 f(n h_1\epsilon h) \omega(g, h_1\epsilon h) \varphi({\bf x}_0, {\bf x}_0 u_0 + {\bf y}_0) \psi(u_0)
 \, du_0 \, dn \\
 & = \int_{(k \backslash \A)^2}
 f([{\bf n}(x), {\bf n}(-x)] h_1\epsilon h) \omega(g, h_1\epsilon h) \varphi({\bf x}_0, {\bf n}(-u_0){\bf y}_0) \psi(u_0) \, du_0 \, dx.
\end{align*}
Note that
\[
 \omega(g, h_1\epsilon h) \varphi({\bf x}_0, {\bf n}(-u_0){\bf y}_0)
 = \omega(g, [{\bf n}(u_0), 1] h_1\epsilon h) \varphi({\bf x}_0,{\bf y}_0).
\]
Hence $I_1$ is equal to 
\begin{align*}
 & \int_{{\pmb \mu}_2(\A)}\int_{\Delta N(\A) \backslash H_1^{\circ}(\A)} \int_{(k \backslash \A)^2}
 f([{\bf n}(x-u_0), {\bf n}(-x)] h_1\epsilon h) \omega(g, h_1\epsilon h) \varphi({\bf x}_0,{\bf y}_0) \psi(u_0)
 \, du_0 \, dx  \, dh_1\,d\epsilon \\
 &= \int_{{\pmb \mu}_2(\A)}\int_{\Delta N(\A) \backslash H_1^{\circ}(\A)} \int_{(k \backslash \A)^2}
 f([{\bf n}(-u_0), {\bf n}(-x)] h_1\epsilon h) \omega(g, h_1\epsilon h) \varphi({\bf x}_0,{\bf y}_0) \psi(u_0 + x)
 \, du_0 \, dx  \, dh_1\,d\epsilon \\
 &= \int_{{\pmb \mu}_2(\A)}\int_{\Delta N(\A) \backslash H_1^{\circ}(\A)} \int_{(k\backslash \A)^2}f([{\bf n}(x),{\bf n}(y)]h_1\epsilon h)\overline{\psi(x+y)} \omega(g, h_1\epsilon h) \varphi({\bf x}_0,{\bf y}_0) 
 \,dx\,dy\, dh_1\, d\epsilon.
\end{align*}
This completes the proof.
\end{proof}

\begin{lemma}\label{L:global Whittaker}
Let $\varphi=\bigotimes_v\varphi_v \in S(V^2(\A))$ be the Schwartz function defined in (\ref{E:non-archi Schwartz fun})-(\ref{E:archi Schwartz fun PS}). Let $g \in \Sp_4(\A)$. We have
\[
W_{\theta({\bf f}^{\sharp},\varphi)}(g) = 2^{-1-|(S({\rm DS})\cup S(\frak{n}))\cap \frak{S}|}\cdot \frak{D}^{-5/2}\cdot \xi(2)^{-2}\cdot\left[\prod_v \mathcal{W}_v^+(g_v) + \prod_v \mathcal{W}_v^-(g_v) \right],
\]
where  
\begin{align}\label{E:local Whi.}
\begin{split}
\mathcal{W}_v^+(g_v) &= \int_{\Delta N(k_v)\backslash H_1^{\circ}(k_v)}W_v(h_{1,v})\omega_v(g_v,h_{1,v})\varphi_v({\bf x}_0,{\bf y}_0)\,d\bar{h}_{1,v},\\
 \mathcal{W}_v^-(g_v) &= \int_{\Delta N(k_v)\backslash H_1^{\circ}(k_v)}W_v({\rm Ad}({\bf t}_v)h_{1,v})\omega_v(g_v,h_{1,v}{\bf t}_{v})\varphi_v({\bf x}_0,{\bf y}_0)\,d\bar{h}_{1,v}.
\end{split}
\end{align}
Here $d\bar{h}_{1,v}$ is the quotient measure defined by the measures on $\Delta N(k_v)$ and $H_1^\circ(k_v)$ in \S\,\ref{SS:measures} and \S\,\ref{SS:measure}, respectively.
\end{lemma}

\begin{proof}
Put $S' = (S({\rm DS})\cup S(\frak{n}))\cap \frak{S}$. Let $R$ be a set of places of $k$. By (\ref{E:restriction}) and the normalization of isomorphisms in (\ref{E:iso}), 
%if $S' \cap R\neq  \varnothing$ and $S' \nsubseteq R$, then ${\bf f}_R^\sharp = 0$. %Also
%\begin{align*}
%{\bf f}_R^\sharp = \begin{cases}
%{\bf f} &\mbox{ if $S' \cap R\ = \varnothing$},\\
%{\bf f}\circ {\rm Ad}({\bf t}) & \mbox{ if $S' \subseteq R$}.
%\end{cases}
%\end{align*}
%If $S' \cap R\ =   \varnothing$, then ${\bf f}_R^\sharp = {\bf f}$.
%then ${\bf f}_R^\sharp \in \mathcal{V}_\sigma$ and 
%\begin{align*}
%{\bf f}_R^\sharp \longmapsto \left(\bigotimes_{\scriptstyle{v \in S({\rm DS})\cup S(\frak{nd}) }\atop \scriptstyle{v \notin R}} W_v\right) \otimes \left( \bigotimes_{\scriptstyle{v \in S({\rm DS})\cup S(\frak{nd}) }\atop \scriptstyle{v \in R}} W_v \circ {\rm Ad}({\bf t}_v)\right) \otimes \left(\bigotimes_{v \notin S({\rm DS})\cup S(\frak{nd})}W_v\right).
%\end{align*}
%If $S' \subseteq R$, then ${\bf f}_R^\sharp = {\bf f}\circ{\bf f}$.
\[
{\bf f}_R^\sharp = \begin{cases}
{\bf f} & \mbox{ if $S' \cap R= \varnothing$},\\
{\bf f}\circ{\rm Ad}({\bf t}) & \mbox{ if $S' \subseteq R$},\\
0 & \mbox{ otherwise}.
\end{cases}
\]
%then ${\bf f}_R^\sharp \in \mathcal{V}_{\sigma\circ {\rm Ad}({\bf t})}$ and 
%\begin{align*}
%{\bf f}_R^\sharp \longmapsto \left(\bigotimes_{\scriptstyle{v \in S({\rm DS})\cup S(\frak{nd}) }\atop \scriptstyle{v \notin R}} W_v\right) \otimes \left( \bigotimes_{\scriptstyle{v \in S({\rm DS})\cup S(\frak{nd}) }\atop \scriptstyle{v \in R}} W_v \circ {\rm Ad}({\bf t}_v)\right) \otimes \left(\bigotimes_{v \notin S({\rm DS})\cup S(\frak{nd})}W_v\circ {\rm Ad}({\bf t}_v)\right).
%\end{align*}
%Since
%$$W_{\bf f} = \prod_v W_v,$$
It follows that
\begin{align}\label{E:restriction of Whi.}
W_{{\bf f}_R^\sharp} = 
\begin{cases}
\displaystyle{\prod_v W_v} & \mbox{ if $S' \cap R =   \varnothing$},\\
\displaystyle{\prod_v W_v\circ{\rm Ad}({\bf t}_v)} & \mbox{ if $S' \subseteq R$},\\
0 & \mbox{ otherwise}.
\end{cases}
%\begin{cases}
%\displaystyle{\prod_{\scriptstyle{v \in S({\rm DS})\cup S(\frak{nd}) }\atop \scriptstyle{v \notin R}}W_v \prod_{\scriptstyle{v \in S({\rm DS})\cup S(\frak{nd}) }\atop \scriptstyle{v \in R}}W_v\circ{\rm Ad}({\bf t}_v)\prod_{v\notin S({\rm DS})\cup S(\frak{nd})}W_v} & \mbox{ if $S' \cap R\ =   \varnothing$},\\
%\displaystyle{\prod_{\scriptstyle{v \in S({\rm DS})\cup S(\frak{nd}) }\atop \scriptstyle{v \notin R}}W_v \prod_{\scriptstyle{v \in S({\rm DS})\cup S(\frak{nd}) }\atop \scriptstyle{v \in R}}W_v\circ{\rm Ad}({\bf t}_v)\prod_{v\notin S({\rm DS})\cup S(\frak{nd})}W_v\circ{\rm Ad}({\bf t}_v)} & \mbox{ if $S' \subseteq R$},\\
%0 & \mbox{ otherwise}.
%\end{cases}
\end{align}
Note that $\omega_v(1,{\bf t}_v)\varphi_v = \varphi_v$ for $v \notin S'$. Hence we conclude from (\ref{E:restriction of Whi.}) that for $g \in \Sp_4(\A)$, we have
\[
\int_{\Delta N(\A) \backslash H_1^\circ(\A)}W_{{\bf f}_R^\sharp}(h_1)\omega(g,h_1{\bf t}_R)\varphi({\bf x}_0,{\bf y}_0)\,dh_1 = \frak{D}^{-5/2}\xi(2)^{-2}\cdot\begin{cases}
\displaystyle{\prod_v \mathcal{W}_v^+(g_v)}  & \mbox{ if $S'\cap R = \varnothing$},\\
\displaystyle{\prod_v \mathcal{W}_v^-(g_v)} & \mbox{ if $S' \subseteq R$},\\
0 & \mbox{ otherwise}.
\end{cases} 
\]
%Note that ${\bf f}^\sharp$ is right ${\bf t}_v$-invariant for all $v \notin S({\rm DS})\cup S(\frak{nd})$. 
Here the extra factor $\frak{D}^{-5/2}\xi(2)^{-2}$ is due to the ratio of the Tamagawa measures on $\Delta N(\A)$ and $H_1^\circ(\A)$ to the corresponding standard measures defined in \S\,\ref{SS:measures} and \S\,\ref{SS:measure}, respectively.
The assertion then follows from Lemma \ref{L:Whittaker function}. This completes the proof.
\end{proof}

\subsection{Calculation of local Whittaker functions}

\subsubsection{Non-archimedean cases}
Let $v$ be a finite place of $k$ and $\varphi_v \in S(V^2(k_v))$ the Schwartz function defined in (\ref{E:non-archi Schwartz fun}). 
If $v \nmid \frak{n}$, then 
\begin{align}\label{E:equivariant property of test vector case v=p not divide N_1N_2}
\omega_v(k',[k_1,k_2])\varphi_v=\varphi_v
\end{align}
for $k' \in {\rm diag}(1,1,\varpi_v^{\c_v},\varpi_v^{\c_v})G(\o_v){\rm diag}(1,1,\varpi_v^{-\c_v},\varpi_v^{-\c_v})$ and $(k_1,k_2) \in \GL_2(\frak{o}_v)\times \GL_2(\frak{o}_v)$ such that $\nu(k')=\det(k_1k_2^{-1})$.
If $v \mid \frak{n}_1$, then
\begin{align}\label{E:equivariant property of test vector case v=p divide N_1}
\omega_v(k',[k_1,k_2])\varphi_v = \varphi_v
\end{align}
for $k' \in {\rm diag}(1,1,\varpi_v^{\c_v},\varpi_v^{\c_v}){\rm K}(\varpi_v){\rm diag}(1,1,\varpi_v^{-\c_v},\varpi_v^{-\c_v})$ and $(k_1,k_2) \in K_0(\varpi_v)\times \GL_2(\frak{o}_v)$ such that $\nu(k')=\det(k_1k_2^{-1})$. 
If $v\mid \frak{n}_2$, then 
\begin{align}\label{E:equivariant property of test vector case v=p divide N_2}
\omega_v(k',[k_1,k_2])\varphi_v = \varphi_v
\end{align}
for $k' \in {\rm diag}(1,1,\varpi_v^{\c_v},\varpi_v^{\c_v}){\rm K}(\varpi_v){\rm diag}(1,1,\varpi_v^{-\c_v},\varpi_v^{-\c_v})$ and $(k_1,k_2) \in \GL_2(\o_v)\times K_0(\varpi_v)$ such that $\nu(k')=\det(k_1k_2^{-1})$.

Let $\mathcal{W}^{\pm}_v$ be the local Whittaker functions defined in (\ref{E:local Whi.}).
\begin{lemma}\label{L:local Whittaker finite 1}
Assume $v$ is a finite place of $k$ and $ v \nmid \frak{n}$. We have
\[\mathcal{W}_v^+(\diag( \varpi_v^{-\frak{c}_v}, 1, \varpi_v^{\frak{c}_v}, 1 ))=\mathcal{W}_v^-(\diag( \varpi_v^{-\frak{c}_v}, 1, \varpi_v^{\frak{c}_v}, 1 ))=q_v^{\c_v}.\]
\end{lemma}

\begin{proof}
We drop the subscript $v$ for brevity. Let 
\[g = \diag( \varpi^{-\frak{c}}, 1, \varpi^{\frak{c}}, 1 ) \in \Sp_4(k_v),\quad h = [{\bf d}(\varpi^{\c}),{\bf d}(\varpi^{\c})] \in H_1^\circ(k_v).\]
By (\ref{E:equivariant property of test vector case v=p not divide N_1N_2}) and our normalization of the Haar measure on $H_1^{\circ}(k_v)$ in (\ref{E:non-archi. integration2}), we have
\[
\mathcal{W}^+(g) = q^{-2\c}\left(Z^{(1)}+ q\cdot Z^{(2)}\right),
%\mathcal{W}^-(g) &= q^{-2\c}\left(Z^{(3)}+q\cdot Z^{(4)}\right),
\]
where
\begin{align*}
Z^{(1)} &= \int_{k_v}\int_{k_v^\times}\int_{k_v^\times}W([1,{\bf n}(x)][{\bf m}(y_1),{\bf m}(y_2)]h)\\
&\times\omega(g,[1,{\bf n}(x)][{\bf m}(y_1),{\bf m}(y_2)]h)\varphi({\bf x}_0,{\bf y}_0)|y_1|^{-2}|y_2|^{-2}\,d^{\times}y_1\,d^\times y_2\,dx,\\
Z^{(2)} &= \int_{k_v}\int_{k_v^\times}\int_{k_v^\times}W([{\bf a}(\varpi),{\bf a}(\varpi)][1,{\bf n}(x)][{\bf m}(y_1),{\bf m}(y_2)]h)\\
&\times\omega(g,[{\bf a}(\varpi),{\bf a}(\varpi)][1,{\bf n}(x)][{\bf m}(y_1),{\bf m}(y_2)]h)\varphi({\bf x}_0,{\bf y}_0)|y_1|^{-2}|y_2|^{-2}\,d^{\times}y_1\,d^\times y_2\,dx.\\
%Z^{(3)} &= \int_{k_v}\int_{k_v^\times}\int_{k_v^\times}\psi(x)W([{\bf m}(y_2),{\bf m}(y_1)]h)\\
%&\times\omega(g,[1,{\bf n}(x)][{\bf m}(y_1),{\bf m}(y_2)]{\bf t}h)\varphi({\bf x}_0,{\bf y}_0)|y_1|^{-2}|y_2|^{-2}\,d^{\times}y_1\,d^\times y_2\,dx,\\
%Z^{(4)} &= \int_{k_v}\int_{k_v^\times}\int_{k_v^\times}\psi(\varpi x)W([{\bf a}(\varpi),{\bf a}(\varpi)][{\bf m}(y_2),{\bf m}(y_1)]h)\\
%&\times\omega(g,[{\bf a}(\varpi),{\bf a}(\varpi)][1,{\bf n}(x)][{\bf m}(y_1),{\bf m}(y_2)]{\bf t}h)\varphi({\bf x}_0,{\bf y}_0)|y_1|^{-2}|y_2|^{-2}\,d^{\times}y_1\,d^\times y_2\,dx.
\end{align*} 
For $h_1 \in H_1(k_v)$, we have
\[
\omega(g,h_1)\varphi({\bf x}_0,{\bf y}_0) = q^{2\frak{c}}\cdot\varphi(h_1^{-1}\cdot (\varpi^{-\frak{c}} {\bf x}_0,{\bf y}_0)).
\]
%\begin{align*}
%\omega(g,h_1h)\varphi({\bf x}_0,{\bf y}_0) = q^{4\frak{c}}\cdot\varphi(h^{-1}h_1^{-1}\cdot (\varpi^{-\frak{c}} {\bf x}_0,{\bf y}_0)).
%\end{align*}
Let $h_1 = [1,{\bf n}(x)][{\bf m}(y_1),{\bf m}(y_2)]h$ with $x \in k_v$, $y_1,y_2 \in k_v^{\times}$. Then
\begin{align*}
\omega(g,h_1)\varphi({\bf x}_0,{\bf y}_0)&=q^{2\frak{c}}\cdot\varphi\left(\bp 0 & -y_1^{-1}y_2^{-1}\\ 0 & 0 \ep, \bp -y_1^{-1}y_2 & -y_1^{-1}y_2^{-1}x\varpi^{\c} \\ 0 & y_1y_2^{-1}\ep \right),\\
\omega(g,[{\bf a}(\varpi),{\bf a}(\varpi)]h_1)\varphi({\bf x}_0,{\bf y}_0)&=q^{2\frak{c}}\cdot\varphi\left(\bp 0 & -y_1^{-1}y_2^{-1}\varpi^{-1}\\ 0 & 0 \ep, \bp -y_1^{-1}y_2 & -y_1^{-1}y_2^{-1}x\varpi^{\c} \\ 0 & y_1y_2^{-1}\ep \right).
\end{align*}
Therefore, 
\begin{align*}
\omega(g,h_1)\varphi({\bf x}_0,{\bf y}_0) \neq 0 & \mbox{ if and only if }y_1y_2^{-1} \in \frak{o}^{\times}, y_1^{-1} \in \frak{o} ,\mbox{ and }x \in y_1y_2\varpi^{-\frak{c}}\frak{o},\\
\omega(g,[{\bf a}(\varpi),{\bf a}(\varpi)]h_1)\varphi({\bf x}_0,{\bf y}_0) \neq 0 &\mbox{ if and only if }y_1y_2^{-1} \in \o^{\times},
 y_1^{-1} \in \varpi\o,\mbox{ and }x \in y_1y_2\varpi^{-\c}\o.
\end{align*}
On the other hand, the functions $W_{1}({\bf a}(y))$ and $W_{2}({\bf a}(y))$ are both supported in $\varpi^{-\c}\o$. We conclude that 
\begin{align*}
&W(h_1)\omega(g,h_1)\varphi({\bf x}_0,{\bf y}_0) \neq 0 \mbox{ if and only if }y_1,y_2\in\o^{\times}\mbox{ and }x\in \varpi^{-\c}\o,\\
&W([{\bf a}(\varpi),{\bf a}(\varpi)]h_1)\omega(g,[{\bf a}(\varpi),{\bf a}(\varpi)]h_1)\varphi({\bf x}_0,{\bf y}_0)=0.
\end{align*}
Hence 
\begin{align*}
Z^{(1)} &= q^{2\c}\cdot W([{\bf a}(\varpi^{-\c}),{\bf a}(\varpi^{-\c})])\cdot{\rm vol}(\varpi^{-\c}\o){\rm vol}(\o^\times)^2 = q^{3\c},\\ Z^{(2)} &=0.
\end{align*}
A similar calculation shows that $\mathcal{W}^-(g)=\mathcal{W}^+(g).$ This completes the proof.
% Similarly we have 
% $$Z^{(3)} = q^{5\c},\quad Z^{(4)} =0.$$ This completes the proof.
\end{proof}

\begin{lemma}\label{L:local Whittaker finite 2}
Assume $v \mid \frak{n}$. We have 
\[\mathcal{W}_v^+(\diag( \varpi_v^{-\frak{c}_v}, 1, \varpi_v^{\frak{c}_v}, 1 ))=\mathcal{W}_v^-(\diag( \varpi_v^{-\frak{c}_v}, 1, \varpi_v^{\frak{c}_v}, 1 ))=q_v^{\c_v}(1+q_v)^{-1}.\]
\end{lemma}

\begin{proof}
The calculations for $v \mid \frak{n}_1$ and $v\mid \frak{n}_2$ are similar and we assume $v \mid \frak{n}_1$. We drop the subscript $v$ for brevity. Let $\mathcal{U}$ be the open compact subgroup of $H_1^{\circ}(k_v)$ defined by
\[\mathcal{U}=H_1^{\circ}(k_v)\cap (K_0(\varpi)\times \GL_2(\o))/\o^{\times}.\]
Note that
\[
\mathcal{C}=\left\{[k(a),1],[w,1] \mbox{ }\vert\mbox{ }a \in \o / \varpi\o           \right\}
\]
is a complete set of coset representatives for $H_1^\circ(\o)/\mathcal{U}$,
where
\[k(a) = \bp 1 & 0 \\ a & 1 \ep.\]
Let
\[g = \diag( \varpi^{-\frak{c}}, 1, \varpi^{\frak{c}}, 1 ) \in \Sp_4(k_v),\quad h = [{\bf d}(\varpi^{\frak{c}}),{\bf d}(\varpi^{\frak{c}})] \in H_1^\circ(k_v).\]
By (\ref{E:equivariant property of test vector case v=p divide N_1}) and our normalization of the Haar measure on $H_1^{\circ}(k_v)$ in (\ref{E:non-archi. integration2}), we have
\[
\mathcal{W}^+(g) = q^{-2\c}(1+q)^{-1}\left( Z^{(1)}+ q\cdot Z^{(2)}\right),
%\mathcal{W}^-(g) &= q^{-2\c}(1+q)^{-1}\left( Z^{(3)}+ q\cdot Z^{(4)}\right),
\]
where
\begin{align*}
Z^{(1)} &= \sum_{k \in \mathcal{C}}\int_{k_v}\int_{k_v^\times}\int_{k_v^\times}W([1,{\bf n}(x)][{\bf m}(y_1),{\bf m}(y_2)]hk)\\
&\times\omega(g,[1,{\bf n}(x)][{\bf m}(y_1),{\bf m}(y_2)]hk)\varphi({\bf x}_0,{\bf y}_0)|y_1|^{-2}|y_2|^{-2}\,d^{\times}y_1\,d^\times y_2\,dx,\\
Z^{(2)} &= \sum_{k \in \mathcal{C}}\int_{k_v}\int_{k_v^\times}\int_{k_v^\times}W([{\bf a}(\varpi),{\bf a}(\varpi)][1,{\bf n}(x)][{\bf m}(y_1),{\bf m}(y_2)]hk)\\
&\times\omega(g,[{\bf a}(\varpi),{\bf a}(\varpi)][1,{\bf n}(x)][{\bf m}(y_1),{\bf m}(y_2)]hk)\varphi({\bf x}_0,{\bf y}_0)|y_1|^{-2}|y_2|^{-2}\,d^{\times}y_1\,d^\times y_2\,dx.\\
%Z^{(3)} &= \sum_{k \in \mathcal{C}}\int_{k_v}\int_{k_v^\times}\int_{k_v^\times}\psi(x)W([{\bf m}(y_2),{\bf m}(y_1)]kh)\\
%&\times\omega(g,[1,{\bf n}(x)][{\bf m}(y_1),{\bf m}(y_2)]k{\bf t}h)\varphi({\bf x}_0,{\bf y}_0)|y_1|^{-2}|y_2|^{-2}\,d^{\times}y_1\,d^\times y_2\,dx,\\
%Z^{(4)} &= \sum_{k \in \mathcal{C}}\int_{k_v}\int_{k_v^\times}\int_{k_v^\times}\psi(\varpi x)W([{\bf a}(\varpi),{\bf a}(\varpi)][{\bf m}(y_2),{\bf m}(y_1)]kh)\\
%&\times\omega(g,[{\bf a}(\varpi),{\bf a}(\varpi)][1,{\bf n}(x)][{\bf m}(y_1),{\bf m}(y_2)]k{\bf t}h)\varphi({\bf x}_0,{\bf y}_0)|y_1|^{-2}|y_2|^{-2}\,d^{\times}y_1\,d^\times y_2\,dx.
\end{align*}
For $h_1 \in H_1(k_v)$, we have
\[
\omega(g,h_1)\varphi({\bf x}_0,{\bf y}_0) = q^{2\frak{c}}\cdot\varphi(h_1^{-1}\cdot (\varpi^{-\frak{c}} {\bf x}_0,{\bf y}_0)).
\]
Let $h_1 = [1,{\bf n}(x)][{\bf m}(y_1),{\bf m}(y_2)]hk$ with $x \in k_v$, $y_1,y_2 \in k_v^{\times}$, and $k \in \mathcal{C}$. If $k=[k(a),1]$, then
\begin{align*}
\omega(g,h_1)\varphi({\bf x}_0,{\bf y}_0)&=q^{2\frak{c}}\cdot\varphi\left(\bp 0 & -y_1^{-1}y_2^{-1}\\ 0 & ay_1^{-1}y_2^{-1} \ep, \bp -y_1^{-1}y_2 & -y_1^{-1}y_2^{-1}x\varpi^\c \\ ay_1^{-1}y_2 & y_1y_2^{-1}+ay_1^{-1}y_2^{-1}x\varpi^\c\ep \right),\\
\omega(g,[{\bf a}(\varpi),{\bf a}(\varpi)]h_1)\varphi({\bf x}_0,{\bf y}_0)&=q^{2\frak{c}}\cdot\varphi\left(\bp 0 & -y_1^{-1}y_2^{-1}\varpi^{-1}\\ 0 & ay_1^{-1}y_2^{-1}\varpi^{-1} \ep, \bp -y_1^{-1}y_2 & -y_1^{-1}y_2^{-1}x\varpi^\c \\ ay_1^{-1}y_2 & y_1y_2^{-1}+ay_1^{-1}y_2^{-1}x\varpi^\c\ep \right).
\end{align*}
If $k=[w,1]$, then
\begin{align*}
\omega(g,h_1)\varphi({\bf x}_0,{\bf y}_0) &= q^{2\c}\cdot\varphi \left( \bp 0 & 0 \\ 0 & -y_1^{-1}y_2^{-1} \ep,\bp 0 & -y_1y_2^{-1} \\ -y_1^{-1}y_2 & -y_1^{-1}y_2^{-1}x\varpi^\c  \ep \right),\\
 \omega(g,[{\bf a}(\varpi),{\bf a}(\varpi)]h_1)\varphi({\bf x}_0,{\bf y}_0) &= q^{2\c}\cdot\varphi \left( \bp 0 & 0 \\ 0 & -y_1^{-1}y_2^{-1}\varpi^{-1} \ep,\bp 0 & -y_1y_2^{-1} \\ -y_1^{-1}y_2 & -y_1^{-1}y_2^{-1}x\varpi^\c  \ep \right).
\end{align*}
Therefore, if $k=1$, then
\begin{align*}
\omega(g,h_1)\varphi({\bf x}_0,{\bf y}_0) \neq 0 & \mbox{ if and only if }y_1y_2^{-1} \in \frak{o}^{\times}, y_1^{-1} \in \frak{o} ,\mbox{ and }x \in y_1y_2\varpi^{-\frak{c}}\frak{o},\\
\omega(g,[{\bf a}(\varpi),{\bf a}(\varpi)]h_1)\varphi({\bf x}_0,{\bf y}_0) \neq 0 &\mbox{ if and only if }y_1y_2^{-1} \in \o^{\times},
 y_1^{-1} \in \varpi\o,\mbox{ and }x \in y_1y_2\varpi^{-\c}\o,
\end{align*}
and if $k \neq 1$, then
\begin{align*}
\omega(g,h_1)\varphi({\bf x}_0,{\bf y}_0) \neq 0 & \mbox{ if and only if }y_1y_2^{-1} \in \frak{o}^{\times}, y_1^{-1} \in \varpi\frak{o} ,\mbox{ and }x \in y_1y_2\varpi^{-\frak{c}}\frak{o},\\
\omega(g,[{\bf a}(\varpi),{\bf a}(\varpi)]h_1)\varphi({\bf x}_0,{\bf y}_0) \neq 0 &\mbox{ if and only if }y_1y_2^{-1} \in \o^{\times},
 y_1^{-1} \in \varpi\o,\mbox{ and }x \in y_1y_2\varpi^{-\c}\o.
\end{align*}
On the other hand, the functions $W_{1}({\bf a}(y))$ and $W_{2}({\bf a}(y))$ are both supported in $\varpi^{-\c}\o$, whereas the functions $W_{1}({\bf a}(y)k(a))$ with $a \in \o^\times$ and $W_{1}({\bf a}(y)w)$ are both supported in $\varpi^{-\c-1}\o$. We conclude that if $k=1$, then
\begin{align*}
&W(h_1)\omega(g,h_1)\varphi({\bf x}_0,{\bf y}_0) \neq 0 \mbox{ if and only if }y_1,y_2\in\o^{\times}\mbox{ and }x\in \varpi^{-\c}\o,\\
&W([{\bf a}(\varpi),{\bf a}(\varpi)]h_1)\omega(g,[{\bf a}(\varpi),{\bf a}(\varpi)]h_1)\varphi({\bf x}_0,{\bf y}_0)=0,
\end{align*}
and if $k \neq 1$, then
\begin{align*}
W(h_1)\omega(g,h_1)\varphi({\bf x}_0,{\bf y}_0) &= 0 ,\\
W([{\bf a}(\varpi),{\bf a}(\varpi)]h_1)\omega(g,[{\bf a}(\varpi),{\bf a}(\varpi)]h_1)\varphi({\bf x}_0,{\bf y}_0)&=0.
\end{align*}
Hence 
\begin{align*}
Z^{(1)} &= q^{2\c}\cdot W([{\bf a}(\varpi^{-\c}),{\bf a}(\varpi^{-\c})])\cdot{\rm vol}(\varpi^{-\c}\o){\rm vol}(\o^\times)^2 = q^{3\c},\\ Z^{(2)} &=0.
\end{align*}
A similar calculation shows that $\mathcal{W}^-(g)=\mathcal{W}^+(g).$ This completes the proof.
%Similarly we have 
%$$Z^{(3)} = q^{5\c},\quad Z^{(4)} =0.$$
%This completes the proof.
\end{proof}

\subsubsection{Archimedean cases}
Let $v$ be a real place of $k$. We identify $k_v$ with $\R$. Let $\varphi_v \in S(V^2(\R))$ be the Schwartz function defined in (\ref{E:archi Schwartz fun DS1})-(\ref{E:archi Schwartz fun PS}). 
If $v \in S({\rm DS})$, then
\begin{align}\label{E:equivariant property d.s.}
\omega_{v}(1,[k_{\theta_1},k_{\theta_2}])\varphi_{v} = e^{-\sqrt{-1}(\kappa_{1,v}\theta_1+\kappa_{2,v}\theta_2)}\varphi_{v}
\end{align}
for $k_{\theta_1},k_{\theta_2} \in {\rm SO}(2)$, and
\begin{align}\label{E:lowest weight condition}
Z\cdot \varphi_{v} = -\lambda_1\cdot\varphi_{v},\quad Z'\cdot \varphi_{v} = -\lambda_2\cdot\varphi_{v},\quad N_{-}\cdot\varphi_{v} =0. 
\end{align}
Here $Z, Z', N_-$ are elements in $\frak{sp}_4(\R)\otimes_\R\C$ defined by
\[
Z=-\sqrt{-1} \bp 0&0&1&0 \\ 0&0&0&0 \\ -1&0&0&0\\ 0&0&0&0 \ep,\quad Z' = -\sqrt{-1} \bp 0&0&0&0 \\ 0&0&0&1 \\ 0&0&0&0\\ 0&-1&0&0 \ep,\quad N_- = \frac{1}{2} \bp 0&1&0&\sqrt{-1} \\ -1&0&\sqrt{-1}&0 \\ 0&-\sqrt{-1}&0&1 \\ -\sqrt{-1}& 0 & -1 & 0\ep.
\]
If $v \in S({\rm PS})$, then
\begin{align}\label{E:equivariant property PS}
\omega_v(k,[k_1,k_2])\varphi_v = \varphi_v
\end{align}
for $k \in G(\R)\cap {\rm O}(4)$ and $k_1,k_2 \in {\rm O}(2)$ with $\nu(k) = \det(k_1k_2^{-1})$.

For $n \in \Z_{\geq 0}$, let
\[H_n(x) = (-1)^n e^{x^2}\frac{d^n}{dx^n}\left(e^{-x^2}\right)\]
denote the Hermite polynomial.

\begin{lemma}\label{L:4.5}
Let $r_1,r_2,r_3 \in \R$ with $r_1>0$ and $r_2^2+r_3>0$. For $n \in \Z_{\geq 0}$, put
\[J_n(r_1,r_2,r_3) = \int_{0}^{\infty}y^{n-2}H_n(\sqrt{\pi}(r_1y+r_2y^{-1}))e^{-\pi(r_1y+r_2y^{-1})^2-\pi r_3y^{-2}}\,dy.\]
Then
\[
J_n(r_1,r_2,r_3) = 2^{n-1}\pi^{n/2}\left((r_2^2+r_3)^{1/2}+r_2\right)^n(r_2^2+r_3)^{-1/2}e^{-2\pi r_1\left((r_2^2+r_3)^{1/2}+r_2\right)}.
\]
\end{lemma}

\begin{proof}
The assertion is a variant of \cite[Lemma 7.5]{Ichino2005}. In a small enough neighborhood of $x=0$, we have
\begin{align*}
\sum_{n=0}^{\infty}\frac{1}{n!}(-\sqrt{\pi}x)^nJ_n(r_1,r_2,r_3) &= \int_{0}^{\infty}y^{-2}e^{-\pi(xy+r_1y+r_2y^{-1})^2-\pi r_3y^{-2}}\,dy\\
&=2^{-1}(r_2^2+r_3)^{-1/2}e^{-2\pi r_2(x+r_1)}e^{-2\pi (r_2^2+r_3)^{1/2}|x+r_1|}\\
&=2^{-1}(r_2^2+r_3)^{-1/2}e^{-2\pi \left((r_2^2+r_3)^{1/2}+r_2\right)(x+r_1)}.
\end{align*}
This completes the proof.
\end{proof}

\begin{lemma}\label{L:Mellin transform}
Let $n \in \Z_{\geq0}$. Put
\begin{align*}
h_{n}(a_1,a_2) &= a_1a_2^{n+1}\int_{0}^{\infty}y^{-n-2}\left((a_1^2y^{-2}+a_2^2y^2)^{1/2}-a_2y\right)^{n}(a_1^2y^{-2}+a_2^2y^2)^{-1/2}\\
&\quad\quad\quad\quad\quad\quad\quad\quad\quad\quad\quad\times e^{-2\pi\left(y^2-a_2^2+(a_2y^{-1}+a_2^{-1}y)(a_1^2y^{-2}+a_2^2y^2)^{1/2}\right)}\,dy
\end{align*}
for $a_1,a_2>0.$ Let $s_1,s_2 \in \C$ satisfy
${\rm Re}(s_1+s_2+1)>0$ and ${\rm Re}(s_1)>0>{\rm Re}(s_2).$
Then 
\begin{align*}
&\int_{0}^{\infty}\int_{0}^{\infty}x_1^{s_1-1}x_2^{s_2-1}h_n(x_1,x_2)\,dx_1\,dx_2 \\
& = 2^{-n-4}\pi^{-n-1}(4\pi^{3})^{-s_1/2}(4\pi)^{-s_2/2}\Gamma\left(\frac{s_1+s_2+2n+1}{2}\right)\Gamma\left(\frac{s_1+s_2+1}{2}\right)\Gamma\left(\frac{s_1}{2}\right)\Gamma\left(\frac{-s_2}{2}\right).
\end{align*}
\end{lemma}

\begin{proof}
We make a change of variable from $x_2$ to $x_1y^{-2}x_2$. Then
\begin{align*}
&\int_{0}^{\infty}\int_{0}^{\infty}x_1^{s_1-1}x_2^{s_2-1}h_{n}(x_1,x_2)\,dx_1\,dx_2 \\
&=\int_{0}^{\infty}\int_{0}^{\infty}\int_{0}^{\infty}x_1^{s_1+s_2+2n}x_2^{s_2+n}y^{-2s_2-4n-3}\left((1+x_2^2)^{1/2}-x_2\right)^{n}(1+x_2^2)^{-1/2}\\
&\quad\quad\quad\quad\quad\quad\quad\times e^{-2\pi x_1^2x_2y^{-4}\left((1+x_2^2)^{1/2}-x_2\right)}e^{-2\pi y^2x_2^{-1}\left((1+x_2^2)^{1/2}+x_2\right)}\,dx_1\,dy\,dx_2\\
&=2^{-1}(2\pi)^{-(s_1+s_2+2n+1)/2}\Gamma\left(\frac{s_1+s_2+2n+1}{2}\right)\\
&\times\int_{0}^{\infty}\int_{0}^{\infty}x_2^{(-s_1+s_2-1)/2}y^{2s_1-1}\left((1+x_2^2)^{1/2}-x_2\right)^{-(s_1+s_2+1)/2}(1+x_2^2)^{-1/2}\\
&\quad\quad\quad\quad\quad\quad\quad\quad\quad\quad\quad\quad\quad\quad\quad\quad\quad\quad\quad\times e^{-2\pi y^2x_2^{-1}\left((1+x_2^2)^{1/2}+x_2\right)}\,dy\,dx_2\\
&=2^{-2}(2\pi)^{(-3s_1-s_2-2n-1)/2}\Gamma\left(\frac{s_1+s_2+2n+1}{2}\right)\Gamma\left(s_1\right)\\
&\times \int_{0}^{\infty}x_2^{(s_1+s_2-1)/2}\left((1+x_2^2)^{1/2}+x_2\right)^{(-s_1+s_2+1)/2}(1+x_2^2)^{-1/2}\,dx_2\\
&=2^{-n-4}\pi^{-n-1}(4\pi^{3})^{-s_1/2}(4\pi)^{-s_2/2}\Gamma\left(\frac{s_1+s_2+2n+1}{2}\right)\Gamma\left(\frac{s_1+s_2+1}{2}\right)\Gamma\left(\frac{s_1}{2}\right)\Gamma\left(\frac{-s_2}{2}\right).
\end{align*}
Here the last equality follows from \cite[p.\,311, (28)]{BE1954} and the duplication formula
\[\Gamma\left(\frac{s_1}{2}\right)\Gamma\left(\frac{s_1+1}{2}\right)= 2^{1-s_1}\sqrt{\pi}\,\Gamma\left(s_1\right).\] This completes the proof.
\end{proof}

\begin{lemma}\label{L:local Whittaker archi. 1}
Let $v \in S({\rm DS})$. For $a_1,a_2>0$, we have 
\begin{align*}
&\mathcal{W}_{v}^+({\rm diag}(a_1,a_2,a_1^{-1},a_2^{-1}))\\
 &= \mathcal{W}_{v}^-({\rm diag}(a_1,a_2,a_1^{-1},a_2^{-1})) \\
&= 2^{-\lambda_{1,v}- 4}\pi^{(-3\lambda_{1,v}+\lambda_{2,v}-5)/2}\\
&\times e^{-2\pi a_2^2}\int_{c_1-\sqrt{-1}\infty}^{c_1+\sqrt{-1}\infty}\frac{ds_1}{2\pi \sqrt{-1}}\,\int_{c_2-\sqrt{-1}\infty}^{c_2+\sqrt{-1}\infty}\frac{ds_2}{2\pi \sqrt{-1}}\,(4\pi^{3}a_1^2)^{(-s_1+\lambda_{1,v}+1)/2}(4\pi a_2^2)^{(-s_2+\lambda_{2,v})/2}\\
&\quad\quad\quad\quad\quad\quad\quad\quad\quad\quad\quad\quad\times \Gamma\left(\frac{s_1+s_2-2\lambda_{2,v}+1}{2}\right)\Gamma\left(\frac{s_1+s_2+1}{2}\right)\Gamma\left(\frac{s_1}{2}\right)\Gamma\left(\frac{-s_2}{2}\right).
\end{align*}
Here $c_1,c_2 \in \R$ satisfy 
\[
c_1+c_2+1>0,\quad c_1>0>c_2.
\]
\end{lemma}

\begin{proof}
We identify $k_v=\R$ and drop the subscript $v$ for brevity. Without loss of generality, we assume $\kappa_1 \geq \kappa_2$. Let $g = {\rm diag}(a_1,a_2,a_1^{-1},a_2^{-1})$. Note that $W\vert_{H_1^\circ(\R)}$ is supported in $H_1^{\circ}(\R)^0$ and 
\[W([{\bf a}(y_1),{\bf a}(y_2)]) = y_1^{\kappa_1/2}y_2^{\kappa_2/2}e^{-2\pi y_1-2\pi y_2}\]
for $y_1,y_2 >0$. By (\ref{E:archi. integration}) and (\ref{E:equivariant property d.s.}), we have
\begin{align*}
&\mathcal{W}^+(g)= 4a_1^{\lambda_1+2}a_2^{-\lambda_2+2}\int_{0}^{\infty}\int_{0}^{\infty} y_1^{-\lambda_2-2}y_2^{\lambda_2-2}e^{-\pi(a_1^2y_1^{-2}y_2^{-2}+a_2^2y_1^{-2}y_2^2+a_2^{2}y_1^2y_2^{-2}+2y_1^2+2y_2^2)}\\
&\quad\quad\quad\quad\quad\quad\quad\quad\quad\times \int_{\R}(y_1y_2^{-1}-y_1^{-1}y_2-\sqrt{-1}x)^{-\lambda_2}e^{-\pi a_2^2 x^2}e^{2\pi \sqrt{-1}xy_1y_2}\,dx\,dy_1\,dy_2.
\end{align*}
By \cite[Lemma 7.4]{Ichino2005}, we have
\begin{align*}
&\int_{\R}(y_1y_2^{-1}-y_1^{-1}y_2-\sqrt{-1}x)^{-\lambda_2}e^{-\pi a_2^2 x^2}e^{2\pi \sqrt{-1}xy_1y_2}\,dx\\
&=(2\sqrt{\pi})^{\lambda_2}a_2^{\lambda_2-1}H_{-\lambda_2}(\sqrt{\pi} ( a_2 y_1 y_2^{-1}-a_2 y_1^{-1}y_2+a_2^{-1}y_1y_2))e^{-\pi a_2^{-2}y_1^2y_2^2}.
\end{align*}
Therefore, by Lemma \ref{L:4.5},
\begin{align*}
\mathcal{W}^+(g)&= 4(2\sqrt{\pi})^{\lambda_2}a_1^{\lambda_1+2}a_2\int_{0}^{\infty}\int_{0}^{\infty}y_1^{-\lambda_2-2}y_2^{\lambda_2-2}e^{-\pi(a_1^2y_1^{-2}y_2^{-2}+a_2^2y_1^{-2}y_2^2+a_2^{2}y_1^2y_2^{-2}+a_2^{-2}y_1^2y_2^2+2y_1^2+2y_2^2)}\\
&\quad\quad\quad\quad\quad\quad\quad\quad\quad\quad\quad\quad\quad\quad\quad\quad\times H_{-\lambda_2}(\sqrt{\pi} ( a_2 y_1 y_2^{-1}-a_2 y_1^{-1}y_2+a_2^{-1}y_1y_2)) \,dy_1\,dy_2\\
&=4(2\sqrt{\pi})^{\lambda_2}a_1^{\lambda_1+2}a_2\int_{0}^{\infty}y^{\lambda_2-2}e^{-4\pi y^2-2\pi a_2^2}J_{-\lambda_2}(a_2y^{-1}+a_2^{-1}y,\,-a_2y,\,a_1^2y^{-2})\,dy\\
&= 2a_1^{\lambda_1+1}a_2^{\lambda_2}e^{-2\pi a_2^2}h_{-\lambda_2}(a_1,a_2).
\end{align*}
A similar calculation shows that $\mathcal{W}^-(g)=\mathcal{W}^+(g).$ The assertion then follows from Lemma \ref{L:Mellin transform} and the Mellin inversion formula. This completes the proof.
\end{proof}

\begin{lemma}\label{L:local Whittaker archi. 2}
Let $v \in S({\rm PS})$. For $a_1,a_2>0$, we have
\begin{align*}
&\mathcal{W}_v^+({\rm diag}(a_1,a_2,a_1^{-1},a_2^{-1})) \\
&= \mathcal{W}_v^-({\rm diag}(a_1,a_2,a_1^{-1},a_2^{-1})) \\&= 2^{-4}a_1^2a_2\\
&\times\int_{c_1-\sqrt{-1}\infty}^{c_1+\sqrt{-1}\infty}\frac{ds_1}{2\pi \sqrt{-1}}\,\int_{c_2-\sqrt{-1}\infty}^{c_2+\sqrt{-1}\infty}\frac{ds_2}{2\pi \sqrt{-1}}\,\left(\pi a_1a_2^{-1}\right)^{-s_1}\left(\pi a_2^{2}\right)^{-s_2}\\
&\times\Gamma\left(\frac{s_1+\lambda_{1,v}}{2}\right)\Gamma\left(\frac{s_1-\lambda_{1,v}}{2}\right)\Gamma\left(\frac{s_1+\lambda_{2,v}}{2}\right)\Gamma\left(\frac{s_1-\lambda_{2,v}}{2}\right)\\
&\times\Gamma\left(\frac{s_2}{2}+\frac{\lambda_{1,v}+\lambda_{2,v}}{4}\right)\Gamma\left(\frac{s_2}{2}-\frac{\lambda_{1,v}+\lambda_{2,v}}{4}\right)\Gamma\left(\frac{s_2}{2}+\frac{\lambda_{1,v}-\lambda_{2,v}}{4}\right)\Gamma\left(\frac{s_2}{2}-\frac{\lambda_{1,v}-\lambda_{2,v}}{4}\right)\\
&\times\Gamma\left(\frac{s_1+s_2}{2}+\frac{\lambda_{1,v}+\lambda_{2,v}}{4}\right)^{-1}\Gamma\left(\frac{s_1+s_2}{2}-\frac{\lambda_{1,v}+\lambda_{2,v}}{4}\right)^{-1}\\
&\times {}_3F_2\left(\frac{s_1}{2},\frac{s_2}{2}+\frac{\lambda_{1,v}-\lambda_{2,v}}{4}, \frac{s_2}{2}-\frac{\lambda_{1,v}-\lambda_{2,v}}{4} ;\, \frac{s_1+s_2}{2}+\frac{\lambda_{1,v}+\lambda_{2,v}}{4},\frac{s_1+s_2}{2}-\frac{\lambda_{1,v}+\lambda_{2,v}}{4};\,1           \right).
\end{align*}
Here $c_1,c_2 \in \R$ satisfy 
\[
c_1>\max\left\{ |{\rm Re}(\lambda_{1,v})|,\,|{\rm Re}(\lambda_{2,v})|\right\},\quad c_2>\max\left\{\left\vert{\rm Re}\left(\frac{\lambda_{1,v}+\lambda_{2,v}}{2}\right)\right\vert,\,\left\vert{\rm Re}\left(\frac{\lambda_{1,v}-\lambda_{2,v}}{2}\right)\right\vert\right\}.
\]
\end{lemma}

\begin{proof}
We identify $k_v=\R$ and drop the subscript $v$ for brevity. Let $g = {\rm diag}(a_1,a_2,a_1^{-1},a_2^{-1})$. Note that 
\[W([{\bf a}(y_1),{\bf a}(y_2)]) = {\rm sgn}^\varepsilon(y_1y_2)|y_1y_2|^{1/2}K_{\mu_1}(2\pi |y_1|)K_{\mu_2}(2\pi |y_2|)\]
for $y_1,y_2 \in \R^{\times}$. By (\ref{E:archi. integration}) and (\ref{E:equivariant property PS}), we have
\begin{align*}
\mathcal{W}^+(g)& =8 a_1^2a_2^2 \int_{0}^{\infty}\int_{0}^{\infty}y_1^{-2}y_2^{-2}K_{\mu_1}(2\pi y_1^2)K_{\mu_2}(2\pi y_2^2)e^{-\pi(a_1^2y_1^{-2}y_2^{-2}+a_2^2y_1^{-2}y_2^2+a_2^2y_1^2y_2^{-2})}\\
&\quad\quad\quad\quad\quad\quad\quad\quad\quad\quad\quad\quad\quad\quad\quad\quad\quad\times\int_{\R}e^{-\pi a_2^2y_1^{-2}y_2^{-2}x^2}e^{2\pi \sqrt{-1}x}\,dx\,dy_1\,dy_2\\
&=8a_1^2a_2\int_{0}^{\infty}\int_{0}^{\infty}y_1^{-1}y_2^{-1}K_{\mu_1}(2\pi y_1^2)K_{\mu_2}(2\pi y_2^2)e^{-\pi(a_1^2y_1^{-2}y_2^{-2}+a_2^2y_1^{-2}y_2^2+a_2^2y_1^2y_2^{-2}+a_2^{-2}y_1^2y_2^2)}\,dy_1\,dy_2.
\end{align*}
Therefore, in the notation of \cite[Theorem 3.2]{Ishii2005}, we have
\[
\mathcal{W}^+(g)=W_{(\lambda_1,\lambda_2)}^{(0,0)}((a_1a_2^{-1},a_2^2)).
\]
A similar calculation shows that $\mathcal{W}^-(g)=\mathcal{W}^+(g).$ The assertion then follows from \cite[Proposition 3.5]{Ishii2005}.
\end{proof}

\subsection{Proof of Proposition \ref{P:W(1)}}\label{S:proof 1}
Let $g_0, g_1 \in G(\A)$ be the elements defined by
\begin{align*}
g_{0,v} &= \begin{cases}
\diag( 1, 1, \varpi_v^{\frak{c}_v}, \varpi_v^{\frak{c}_v} ) & \mbox{ if $v \nmid \infty$},\\
1 & \mbox{ if $v \mid \infty$},
\end{cases}\\
g_{1,v} &= \begin{cases}
\diag( \varpi_v^{-\frak{c}_v}, 1, \varpi_v^{\frak{c}_v}, 1 ) & \mbox{ if $v \nmid \infty$},\\
1 & \mbox{ if $v \mid \infty$}.
\end{cases}
\end{align*}
Let $\varphi = \bigotimes_v \varphi_v \in S(V^2(\A))$ be the Schwartz function defined in (\ref{E:non-archi Schwartz fun})-(\ref{E:archi Schwartz fun PS}). By (\ref{E:equivariant property of test vector case v=p not divide N_1N_2})-(\ref{E:equivariant property PS}), there exists a constant $C$ such that
\[\theta({\bf f}^\sharp,\varphi)(g) = C\cdot f(gg_0)\]
for $g \in G(\A)$.
%Let $g \in \G(\A)$ be defined by
%$$g_v = \begin{cases}
%\diag( \varpi_v^{-\frak{c}_v}, 1, \varpi_v^{2\frak{c}_v}, \varpi_v^{\frak{c}_v} ) & \mbox{ if $v \nmid \infty$},\\
%1 & \mbox{ if $v \mid \infty$}.
%\end{cases}$$
To determine $C$, by the normalization of $f$ in (\ref{E:normalization}), it suffices to compute $W_{\theta({\bf f}^\sharp,\varphi)}(g_1).$ The assertion then follows from Lemmas \ref{L:global Whittaker}-\ref{L:local Whittaker finite 2}, \ref{L:local Whittaker archi. 1} and \ref{L:local Whittaker archi. 2}. This completes the proof of Proposition \ref{P:W(1)}.

\section{Explicit Rallis inner product formula}\label{S:cal2}
We keep the notation of $\S\,\ref{S:Yoshida lifts}$ and $\S\,\ref{S:cal1}$. The aim of this section is to prove Proposition \ref{P:Rallis inner product}. We prove it by specializing the Rallis inner product formula in \cite{GQT2014} to the theta lift $\theta({\bf f}^\sharp,\varphi)$. It boils down to the calculation of certain doubling local zeta integrals involving matrix coefficients of $\sigma_v^\sharp$ and the Weil representation $\omega_v$ for each place $v$.

\subsection{Matrix coefficients}

\subsubsection{Hermitian pairings on $\sigma_v$ and $\sigma_v^\sharp$}\label{SS:local pairing}

Let $v$ be a place of $k$. Let $\<\mbox{ },\mbox{ }\>_{i,v} : \mathcal{V}_{i,v}\otimes\overline{\mathcal{V}}_{i,v} \rightarrow \C$ be the $\GL_2(k_v)$-invariant Hermitian pairing defined by
\[
\<W_1,W_2\>_{i,v}=\int_{k_v^{\times}}W_1({\bf a}(t))\overline{W_2({\bf a}(t))}\,d^{\times}t
\]
for $W_1,W_2 \in \mathcal{V}_{i,v}$. 
Let $\mathcal{B}_v : \mathcal{V}_v\otimes \overline{\mathcal{V}}_v\rightarrow \C$ be the $H^{\circ}(k_v)$-invariant Hermitian pairing defined by
\[
\mathcal{B}_v(W_1\otimes W_2,W_3\otimes W_4) = \frac{\zeta_{v}(2)^2}{\zeta_{v}(1)^2}\cdot\frac{\<W_1,W_3\>_{1,v}\<W_2,W_4\>_{2,v}}{L(1,\sigma_{1,v},{\rm Ad})L(1,\sigma_{2,v},{\rm Ad})}
\]
for $W_1,W_3 \in \mathcal{V}_{1,v}$ and $W_2,W_4 \in \mathcal{V}_{2,v}$.
Let $\mathcal{B}_v^{\sharp} : \mathcal{V}_v^{\sharp}\otimes\overline{\mathcal{V}}_v^{\sharp}\rightarrow \C$ be the $H(k_v)$-invariant Hermitian pairing defined as follows:
\begin{itemize}
\item If $v \notin \frak{S}$, then $\mathcal{B}_v^{\sharp}=\mathcal{B}_v$.
\item If $v \in \frak{S}$, then
$$\mathcal{B}_v^{\sharp}((W_1,W_2),(W_3,W_4))=\frac{1}{2}\left[\mathcal{B}_v(W_1,W_3)+\mathcal{B}_v(W_2,W_4)\right]$$
for $(W_1,W_2),\,(W_3,W_4) \in \mathcal{V}_{v}^{\sharp}$.
\end{itemize}

\subsubsection{Matrix coefficients of $\sigma_v$}
Let $v$ be a place of $k$.
Let $W_v$ be the Whittaker function on $H^{\circ}(k_v)$ defined in (\ref{E:Whittaker newform}).

\begin{lemma}\label{E:unramified matrix coefficient}
Let $v \nmid \infty \frak{n}$. We have
\[
\mathcal{B}_v(W_v,W_v)=1.
\]
\end{lemma}

\begin{proof}
The calculation was carried out in \cite[Lemma 6.9]{IP2018}.
\end{proof}

\begin{lemma}\label{E:matrix coefficients}
Let $v \mid \frak{n}_1$. Write
\[\sigma_{1,v} = {\rm St}\otimes {\eta},\quad \sigma_{2,v}  = {\rm Ind}_{B(k_v)}^{\GL_2(k_v)}(\chi\boxtimes\chi^{-1}),\]
where ${\rm St}$ is the Steinberg representation of $\GL_2(k_v)$, and $\chi$  and {$\eta$} are unramified characters of $k_v^{\times}$ such that $\chi= |\mbox{ }|_v^s $ with $|{\rm Re}(s)|<1/2$ and ${\eta^2}=1$. For
\[h_{n,m} = [{\bf a}(\varpi_v^n){\bf d}(\varpi_v^m),{\bf a}(\varpi_v^{n+m})],\quad h_{n,m}' = [w{\bf a}(\varpi_v^n){\bf d}(\varpi_v^m),{\bf a}(\varpi_v^{n+m})]\]
with $m,n \in \Z$ such that $m+n \geq 0$, we have
\begin{align*}
%\begin{split}
\mathcal{B}_{ v}(\sigma_{ v}(h_{n,m})W_{v },W_{ v}) &= \frac{\zeta_v(2)}{\zeta_v(1)}\cdot\varepsilon^{n-m}\cdot\frac{q_v^{-|n-m|-(n+m)/2}}{1+q_v^{-1}}\cdot\left(\alpha^{n+m}\cdot\frac{1-\alpha^{-2}q_v^{-1}}{1-\alpha^{-2}}+\alpha^{-(n+m)}\cdot\frac{1-\alpha^2q_v^{-1}}{1-\alpha^2}\right),\\
\mathcal{B}_{ v}(\sigma_{v }(h_{n,m}')W_{v },W_{ v}) &= -\frac{\zeta_v(2)}{\zeta_v(1)}\cdot\varepsilon^{n-m}\cdot\frac{q_v^{-|n-m-1|-(n+m)/2}}{1+q_v^{-1}}\cdot\left(\alpha^{n+m}\cdot\frac{1-\alpha^{-2}q_v^{-1}}{1-\alpha^{-2}}+\alpha^{-(n+m)}\cdot\frac{1-\alpha^2q_v^{-1}}{1-\alpha^2}\right),
%\end{split}
\end{align*}
where
\begin{align}\label{E:symbols}
\alpha = \chi(\varpi_v),\quad \varepsilon = {\eta}(\varpi_v).\end{align}
\end{lemma}

\begin{proof}
By \cite[Lemmas 6.9 and 6.11]{IP2018},
\[\mathcal{B}_{ v}(W_{v },W_{ v}) = \frac{\zeta_v(2)}{\zeta_v(1)}.\]
The assertion then follows from the formulae for matrix coefficients in \cite{Macdonald1971} and \cite[\S 7]{GJ1972}.
\end{proof}

\begin{lemma}\label{E:coefficient}
Let $v \in S({\rm DS})$. For $h = [k_{\theta_1},k_{\theta_2}][{\bf m}(e^{t_1}),{\bf m}(e^{t_2})][k_{\theta_3},k_{\theta_4}] \in H_1^\circ(\R)^0$ with $k_{\theta_1},k_{\theta_2},k_{\theta_3},k_{\theta_4} \in {\rm SO}(2)$ and $t_1, t_2 \in \R$, we have
\[
\mathcal{B}_{ v}(\sigma_{v }(h)W_{ v},W_{v })  = 2^{-2\lambda_{1,v}-2}e^{\sqrt{-1}(\kappa_{1,v}\theta_1+\kappa_{1,v}\theta_3+\kappa_{2,v}\theta_2+\kappa_{2,v}\theta_4)}\cosh(t_1)^{-\kappa_{1,v}}\cosh(t_2)^{-\kappa_{2,v}}.
\]
For $h \in H^\circ(\R) \setminus H^\circ(\R)^0$, we have $\mathcal{B}_{ v}(\sigma_{v }(h)W_{ v},W_{v }) =0$.
\end{lemma}

\begin{proof}
The assertion follows from the formulae for the Whittaker functions that
\[
W_{i,v}({\bf a}(y)k_{\theta})=\begin{cases} e^{\sqrt{-1}\kappa_{i,v}\theta}y^{\kappa_{i,v}/2}e^{-2\pi y}&\mbox{ if }y>0,\\
0&\mbox{ if }y<0,\end{cases}
\]
for $i=1,2$, $y \in \R^{\times}$ and $ k_{\theta} \in {\rm SO}(2)$. 
\end{proof}

\subsubsection{Matrix coefficients of the Weil representations}
Let $v$ be a place of $k$ and $\varphi_v \in S(V^2(k_v))$ the Schwartz function defined in (\ref{E:non-archi Schwartz fun})-(\ref{E:archi Schwartz fun PS}). Let $\Phi_v$ be the matrix coefficient of the Weil representation $\omega_v$ defined by
\[
\Phi_v(h_{1,v}) = \int_{V^2(k_v)}\omega_v(1,h_{1,v})\varphi_v(x_v)\overline{\varphi_v(x_v)}\,dx_v
\]
for $h_{1,v} \in H_1(k_v)$. 

\begin{lemma}\label{L:matrix coefficient archimedean case}
Let $v \in S({\rm PS})$. We have
\[
\mathcal{B}_v(W_v,W_v) = 2^{-4}.
\]
\end{lemma}

\begin{proof}
By \cite[6.576.4]{Table2000}, we have
\[
\<W_{i,v}, W_{i,v}\>_{i,v} = \int_\R K_{\mu_{i,v}}(2 \pi |t|)^2 \, dt = 2^{-2} \Gamma \left( \frac{1}{2} + \mu_{i,v} \right) \Gamma \left( \frac{1}{2} - \mu_{i,v} \right)
\]
for $i=1,2$.
Therefore
\[
\mathcal{B}_v(W_v,W_v) = \frac{\zeta_v(2)^2}{\zeta_v(1)^2}\cdot\frac{\<W_{1,v}, W_{1,v}\>_{1,v}\<W_{2,v}, W_{2,v}\>_{2,v}}{L(1,\sigma_{1,v},{\rm Ad})L(1,\sigma_{2,v},{\rm Ad})} = 2^{-4}.
\]

\end{proof}

\begin{lemma}\label{L:matrix coefficient nonarchimedean case}
Let $v$ be a finite place of $k$. Let $n,m \in \Z$. If $v \nmid \frak{n}$, then 
$$\Phi_{v}([{\bf a}(\varpi_v^n){\bf d}(\varpi_v^m),{\bf a}(\varpi_v^{n+m})])= q_v^{-2|n|-2|m|}.$$
If $v \mid \frak{n}_1$, then
\begin{align*}
\Phi_{v}([{\bf a}(\varpi_v^n){\bf d}(\varpi_v^m),{\bf a}(\varpi_v^{n+m})]) & =  q_v^{-2|n|-2|m|-2},\\
\Phi_{v}([w{\bf a}(\varpi_v^n){\bf d}(\varpi_v^m),{\bf a}(\varpi_v^{n+m})]) & =  q_v^{-|n|-|n-1|-|m|-|m+1|-2}.
\end{align*}
If $v\mid \frak{n}_2$, then
\begin{align*}
\Phi_{v}([{\bf a}(\varpi_v^{n+m}),{\bf a}(\varpi_v^n){\bf d}(\varpi_v^m)]) & =  q_v^{-2|n|-2|m|-2},\\
\Phi_{v}([{\bf a}(\varpi_v^{n+m}),w{\bf a}(\varpi_v^n){\bf d}(\varpi_v^m)]) & = q_v^{-|n|-|n-1|-|m|-|m+1|-2}.
%\Phi_{v}([{\bf a}(\varpi_v^{n+m}),w{\bf a}(\varpi_v^n){\bf d}(\varpi_v^m)]) & = q_v^{-|n-\c_v|-|n-\c_v-1|-|m+\c_v|-|m+\c_v+1|-2}.
\end{align*}
\end{lemma}

\begin{proof}
The verification is straightforward and we leave it to the readers.
\end{proof}
 
\begin{lemma}\label{L:matrix coeff. of Weil rep d.s}
Let $v \in S({\rm DS})$. For $a,b>0$, we have
\begin{align*}
\Phi_{v}([{\bf m}(a),{\bf m}(b)]) &= \pi^{-\lambda_{1,v}+\lambda_{2,v}}\Gamma(\lambda_{1,v}+1)\Gamma(-\lambda_{2,v}+1)(ab+a^{-1}b^{-1})^{-2}(ab^{-1}+a^{-1}b)^{-2}\\
&\times\left[(ab+a^{-1}b^{-1})^{-1}+(ab^{-1}+a^{-1}b)^{-1}\right]^{\lambda_{1,v}-\lambda_{2,v}}.
\end{align*}
\end{lemma}

\begin{proof}
We identify $k_v=\R$ and drop the subscript $v$ for brevity. Without loss of generality, we assume $\kappa_1 \geq \kappa_2$. For $n \in \Z_{\geq 0}$ and $z=(z_1,z_2,z_3,z_4) \in \C^4$, let 
\begin{align*}
f_n(z) &= \int_{\R^4}(a^{-1}x_1-a^{-1}\sqrt{-1}x_2-a\sqrt{-1}x_3-ax_4)^n(b^{-1}x_1+b\sqrt{-1}x_2+b^{-1}\sqrt{-1}x_3-bx_4)^{n}\\
&\quad\times e^{-\pi\left[ (a^{-2}+b^{-2})x_1^2+ (a^{-2}+b^{2})x_2^2+(a^{2}+b^{-2})x_3^2+(a^{2}+b^{2})x_4^2        \right]}e^{2\pi \sqrt{-1}(z_1x_1+z_2x_2+z_3x_3+z_4x_4)}\,dx_1\,dx_2\,dx_3\,dx_4.
\end{align*}
Then \[f_{0}(z) = (ab+a^{-1}b^{-1})^{-1}(ab^{-1}+a^{-1}b)^{-1}e^{-\pi\left[ (a^{-2}+b^{-2})^{-1}z_1^2+ (a^{-2}+b^{2})^{-1}z_2^2+(a^{2}+b^{-2})^{-1}z_3^2+(a^{2}+b^{2})^{-1}z_4^2        \right]}\]
and $f_n(z) = \nabla^n f_{0}(z)$, where  
\begin{align*}
\nabla = \frac{1}{(2\pi\sqrt{-1})^2}&\left(a^{-1}\frac{\partial}{\partial z_1}-a^{-1}\sqrt{-1}\frac{\partial}{\partial z_2}-a\sqrt{-1}\frac{\partial}{\partial z_3}-a\frac{\partial}{\partial z_4}\right)\\
\times&\left(b^{-1}\frac{\partial}{\partial z_1}+b\sqrt{-1}\frac{\partial}{\partial z_2}+b^{-1}\sqrt{-1}\frac{\partial}{\partial z_3}-b\frac{\partial}{\partial z_4}\right).
\end{align*}
We claim that
\begin{align}\label{E:1}
f_n(0) = \pi^{-n}\Gamma(n+1)(ab+a^{-1}b^{-1})^{-1}(ab^{-1}+a^{-1}b)^{-1}\left[(ab+a^{-1}b^{-1})^{-1}+(ab^{-1}+a^{-1}b)^{-1}\right]^n.
\end{align}

Assume the claim holds. We have
\[
\Phi ([{\bf m}(a),{\bf m}(b)])= \int_{V^2(\R)}\varphi\left({\bf m}(a)^{-1}x,{\bf m}(a)^{-1}y\right)\overline{\varphi(x{\bf m}(b)^{-1},y{\bf m}(b)^{-1})}\,dx\,dy.
\]
Making a change of variables from $(y_1,y_3)$ to $(-y_1,-y_3)$, we see that
\begin{align*}
\Phi([{\bf m}(a),{\bf m}(b)]) &= f_{\lambda_1}(0)f_{-\lambda_2}(0)\\
&=\pi^{-\kappa_1}\Gamma(\lambda_1+1)\Gamma(-\lambda_2+1)(ab+a^{-1}b^{-1})^{-2}(ab^{-1}+a^{-1}b)^{-2}\\
&\times\left[(ab+a^{-1}b^{-1})^{-1}+(ab^{-1}+a^{-1}b)^{-1}\right]^{\kappa_1}.
\end{align*}
Here the last equality follows from (\ref{E:1}).

It remains to prove the claim (\ref{E:1}). Let
\[
D= \{(a,b) \in \C^2 \mbox{ }\vert\mbox{ }ab(ab+a^{-1}b^{-1})(ab^{-1}+a^{-1}b) \neq 0 \}.
\]
Let $(a,b) \in D$ be such that $(1+ab^{-1})(1+a^{-1}b^{-1})(1-a^{-1}b)(1-ab) \neq 0$.
We make the following change of variables:
\begin{align*}
&\begin{cases}
z_1 = a^{-1}u_1+b^{-1}u_4,\\
z_4 = -bu_1+au_4,
\end{cases}
\quad 
\begin{cases}
z_2 = a^{-1}u_2+bu_3,\\
z_3 =  -b^{-1}u_2+au_3,
\end{cases}\\
&\begin{cases}
Z = (1+ab^{-1})^{-1}u_1 + (1+a^{-1}b^{-1})^{-1}\sqrt{-1}u_3 ,\\
\overline{Z} = (1+ab^{-1})^{-1}u_1 - (1+a^{-1}b^{-1})^{-1}\sqrt{-1}u_3,
\end{cases}
\quad
\begin{cases}
Z' = (1-a^{-1}b)^{-1}u_4 + (1-ab)^{-1}\sqrt{-1}u_2,\\
\overline{Z'} = (1-a^{-1}b)^{-1}u_4 - (1-ab)^{-1}\sqrt{-1}u_2,
\end{cases}\\
&\begin{cases}
Z = u+ u',\\
\overline{Z'} = -u+u',
\end{cases}
\quad
\begin{cases}
\overline{Z} = v+v',\\
Z' = v-v'.
\end{cases}
\end{align*}
Then
\begin{align*}
f_0(z) &= (ab+a^{-1}b^{-1})^{-1}(ab^{-1}+a^{-1}b)^{-1} e^{-4\pi\left[(ab+a^{-1}b^{-1})^{-1}+(ab^{-1}+a^{-1}b)^{-1}\right](uv+u'v')}\\
&\times e^{-4\pi\left[(ab^{-1}+a^{-1}b)^{-1}-(ab+a^{-1}b^{-1})^{-1}\right](uu'+vv') - 4\pi(uv'+u'v)},\\
\nabla &= \frac{1}{(2\pi \sqrt{-1})^2}\cdot\frac{\partial^2}{\partial u \partial v}.
\end{align*}
It follows that 
\begin{align*}
f_n(0) &= (-4\pi^2)^{-n}(n!)^2(ab+a^{-1}b^{-1})^{-1}(ab^{-1}+a^{-1}b)^{-1}\\
&\times \left(\mbox{the coefficient of $u^nv^n$ in the Taylor expansion of $e^{-4\pi\left[(ab+a^{-1}b^{-1})^{-1}+(ab^{-1}+a^{-1}b)^{-1}\right]uv}$ at (0,0)}\right)\\
&=\pi^{-n}\Gamma(n+1)(ab+a^{-1}b^{-1})^{-1}(ab^{-1}+a^{-1}b)^{-1}\left[(ab+a^{-1}b^{-1})^{-1}+(ab^{-1}+a^{-1}b)^{-1}\right]^n.
\end{align*}
Therefore (\ref{E:1}) holds in this case.
Since both sides of (\ref{E:1}) are holomorphic functions on $D$, we conclude that they are equal on $D$. In particular, (\ref{E:1}) holds for $a,b>0$. This completes the proof.
\end{proof}

\begin{lemma}\label{L:matrix coeff. of Weil rep PS}
Let $v \in S({\rm PS})$. For $a,b>0$, we have
\[
\Phi_v([{\bf m}(a),{\bf m}(b)]) = (ab+a^{-1}b^{-1})^{-2}(ab^{-1}+a^{-1}b)^{-2}.
\]
\end{lemma}

\begin{proof}
The calculation is similar to the one in Lemma \ref{L:matrix coeff. of Weil rep d.s}. We leave the details to the readers.
\end{proof}

\subsection{Rallis inner product formula}
Let $\mathcal{B}_{\sigma} : \mathcal{V}_{\sigma}\otimes \overline{\mathcal{V}}_{\sigma} \rightarrow \C$ and $\mathcal{B}_{\sigma^{\sharp}} : \mathcal{V}_{\sigma^{\sharp}}\otimes \overline{\mathcal V}_{\sigma^{\sharp}} \rightarrow \C$ be the Petersson pairings defined by
\begin{align*}
\mathcal{B}_{\sigma}(f_1,f_2) &= \int_{Z_{H}(\A)H^{\circ}(k)\backslash H^{\circ}(\A)}f_1(h_0)\overline{f_2(h_0)}\,dh_0,\\
\mathcal{B}_{\sigma^{\sharp}}(f_3,f_4) &= \int_{Z_{H}(\A)H(k)\backslash H(\A)}f_3(h)\overline{f_4(h)}\,dh
\end{align*}
for $f_1,f_2 \in \mathcal{V}_{\sigma}$ and $f_3,f_4 \in \mathcal{V}_{\sigma^{\sharp}}$. Here $Z_H$ is the center of $H$, and $dh_0$ and $dh$ are the Tamagawa measures on $Z_H(\A)\backslash H^{\circ}(\A)$ and $Z_H(\A)\backslash H(\A)$, respectively. Note that ${\rm vol}(Z_H(\A)H^{\circ}(k)\backslash H^{\circ}(\A))=4$ and ${\rm vol}(Z_H(\A)H(k)\backslash H(\A))=2$. 
By \cite[Proposition 6]{Wald1985} and \cite[Lemma 2.3]{GI2011}, we have
\begin{align}
\mathcal{B}_{\sigma} &= \frak{D}^{-1}\cdot\frac{4}{\xi(2)^2}\cdot L(1,\sigma,{\rm Ad})\cdot \prod_v\mathcal{B}_v,\\
\mathcal{B}_{\sigma^{\sharp}} & = \frak{D}^{-1}\cdot\frac{4}{\xi(2)^2}\cdot L(1,\sigma,{\rm Ad})\cdot \prod_v\mathcal{B}_v^{\sharp}\label{E:ratio of measures2}.
\end{align}
Here $\mathcal{B}_v$ and $\mathcal{B}_v^{\sharp}$ are the local Hermitian pairings defined in \S\,\ref{SS:local pairing}. 

Let $\mathcal{B}_{\omega} : \mathcal{S}(V^2(\A))\otimes \overline{\mathcal{S}(V^2(\A))} \rightarrow \C$ be the canonical pairing defined by
\[\mathcal{B}_{\omega}(\varphi_1,\varphi_2) = \int_{V^2(\A)}\varphi_1(x)\overline{\varphi_2(x)}\,dx,\]
where $dx$ is the Tamagawa measure on $V^2(\A)$. For each place $v$ of $k$, let $\mathcal{B}_{\omega_v} : \mathcal{S}(V^2(k_v))\otimes \overline{\mathcal{S}(V^2(k_v))} \rightarrow \C$ be the canonical pairing defined by
\[\mathcal{B}_{\omega_v}(\varphi_1,\varphi_2) = \int_{V^2(k_v)}\varphi_1(x_v)\overline{\varphi_2(x_v)}\,dx_v,\]
where $dx_v$ is defined by the Haar measure on $k_v$ in \S\,\ref{SS:measures}. Then we have
\begin{align}\label{E:ratio of measures3}
\mathcal{B}_\omega(\varphi_1,\varphi_2) = \frak{D}^{-4}\cdot \prod_v \mathcal{B}_{\omega_v}(\varphi_{1,v},\varphi_{2,v})
\end{align}
for $\varphi_1= \bigotimes_v \varphi_{1,v}$, $\varphi_2= \bigotimes_v \varphi_{2,v} \in \mathcal{S}(V^2(\A))$.

\begin{thm}[Rallis inner product formula]\label{T:RIF}
Let $f = \bigotimes_v f_v \in \sigma^\sharp$ and $\varphi=\bigotimes_v\varphi_v \in \calS(V^2(\A))$. We have
\[
\left\<\theta({f},\varphi),\theta({f},\varphi) \right\> = \frak{D}^{-8}\cdot\frac{4}{\xi(2)^4}\cdot\frac{L(1,\pi,{\rm Ad})}{\Delta_{\PGSp_4}}\cdot \prod_{v}Z_v(f_v,\varphi_v),
\]
where 
\[
Z_v(f_v,\varphi_v) = \frac{\zeta_{v}(2)\zeta_{v}(4)}{L(1,\sigma_v,{\rm std})}\cdot \int_{H_1(k_v)}\mathcal{B}_{\omega_v}(\omega_v(1,h_{1,v})\varphi_v,\varphi_v)\mathcal{B}_v^{\sharp}(\sigma_v^{\sharp}(h_{1,v}){f}_v,{f}_v)\,dh_{1,v}.
\]
\end{thm}

\begin{proof}
Since 
\[
 L(s, \pi, \Ad) = L(s, \sigma, {\rm std}) \cdot L(s, \sigma, {\rm Ad}), 
\]
the assertion is a special case of the Rallis inner product formula proved in \cite[Theorem 11.3]{GQT2014} (see also \cite[\S 7]{GI2011}).
Here the extra factor $4 \mathfrak{D}^{-8} \xi(2)^{-4} L(1, \sigma, {\rm Ad})$ is due to the ratio of the Haar measures in (\ref{E:ratio of measures}), (\ref{E:ratio of measures2}), and (\ref{E:ratio of measures3}).
\end{proof}

\subsection{Calculation of local integrals}
Let $v$ be a place of $k$.
Let ${\bf f}_v^\sharp \in \mathcal{V}_v^\sharp$ be the local component of ${\bf f}^\sharp$ with respect to the isomorphism in (\ref{E:iso}), and 
$\varphi_v \in S(V^2(k_v))$ the Schwartz function defined in (\ref{E:non-archi Schwartz fun})-(\ref{E:archi Schwartz fun PS}).
Let $\Psi_v$ and $\Phi_v$ be the matrix coefficients of $\sigma_v^{\sharp}$ and $\omega_v$, respectively, defined by
\begin{align*}
\Psi_v(h_v) &= \mathcal{B}_v^{\sharp}(\sigma_v^{\sharp}(h_v){\bf f}_v^{\sharp},{\bf f}_v^{\sharp}),\\
\Phi_v(h_{1,v}) &= \mathcal{B}_{\omega_v}(\omega_v(1,h_{1,v})\varphi_v,\varphi_v)
\end{align*}
for $h_v \in H(k_v)$ and $h_{1,v} \in H_1(k_v)$.  Let $Z_v$ be the local integral defined by
\[
Z_v = \frac{\zeta_{v}(2)\zeta_{v}(4)}{L(1,\sigma_v,{\rm std})}\cdot \int_{H_1(k_v)}\Phi_v(h_{1,v})\Psi_v(h_{1,v})\,dh_{1,v}.
\]

\begin{lemma}\label{L:zeta integral1}
Let $v$ be a finite place of $k$. If $v \nmid \frak{n}$, then
\[Z_v = 1.\]
%$$Z_v = q_v^{4\c_v}\cdot \begin{cases} 1 & \mbox{ if }v \notin S(\frak{d})\mbox{ or }v \notin \frak{S},
%\\2^{-2} & \mbox{ if }v \in S(\frak{d})\cap\frak{S}.\end{cases}$$
If $v \mid \frak{n}$, then 
\[Z_v  = 2^{-2}\cdot q_v^{-3}\frac{\zeta_v(2)\zeta_v(4)}{\zeta_v(1)^2}.\]
\end{lemma}

\begin{proof}
We drop the subscript $v$ for brevity. First we assume $v\nmid \frak{n}$. %Let $\varphi' \in S(V^2(k_v))$ be defined by
%$$\varphi' = \mathbb{I}_{V^2(\o)}.$$
%Let 
%$$g = {\rm diag}(1,1,\varpi^{-\c},\varpi^{-\c}) \in G(k_v),\quad h = [1,{\bf d}(\varpi^\c)] \in H^\circ(k_v).$$
%We have
%$$\omega(g,h)\varphi' = q^{-2\c}\cdot \varphi.$$
%If $v \notin S(\frak{d})$, then $\omega(1,{\bf t})\varphi = \varphi$ and $\sigma^\sharp({\bf t}){\bf f}^\sharp = {\bf f}^\sharp$. 
%If $v \notin \frak{S}$, then $\sigma_1=\sigma_2$ and $W_2 = \sigma_1(d(\varpi^\c))W_1$. Therefore, for $h_1 \in H_1(k_v)$, we have
%\begin{align*}
%\Phi(h_1) = q^{4\c}\cdot\mathcal{B}_\omega(\omega(1,h^{-1}h_1h)\varphi',\varphi'),\quad \Psi(h_1) = \mathcal{B}(\sigma(h^{-1}h_1h)W_1\otimes W_1,W_1\otimes W_1).
%\end{align*}
%If $v \in S(\frak{d})\cap \frak{S}$, then $\Psi(h) = 2^{-1}\mathcal{B}(\sigma(h)W,W)$ for $h \in H^\circ(k_v)$, and 
%$$\Psi_{ }(h_{1}{\bf t}_{ }) = \mathcal{B}_{ }^{\sharp}((0,\sigma_{ }({\rm Ad}({\bf t}_{ })h_{1})W_{ }),(W_{ },0))=0$$
%for $h_{1} \in H_1^{\circ}(k_v)$. We conclude that
%Since ${\bf f}^\sharp$ and $\varphi$ are both ${\bf t}$-invariant%and $\Psi(h) = \mathcal{B}(\sigma(h)W,W)$ for $h \in H^\circ(k_v)$
%, we have
%\begin{align*}
%Z &= \frac{1}{2}\frac{\zeta(2)\zeta(4)}{L(1,\sigma,{\rm std})}\cdot\int_{H_1^\circ(k_v)}\Phi(h_1)[\mathcal{B}(\sigma(h_1)W,W)+\mathcal{B}(\sigma({\rm Ad}({\bf t})h_1)W,W)]\,dh_1\\
%  &= \frac{\zeta(2)\zeta(4)}{L(1,\sigma,{\rm std})}\cdot\int_{H_1^\circ(k_v)}\Phi(h_1)\mathcal{B}(\sigma(h_1)W,W)\,dh_1.
%\end{align*}
By Lemmas \ref{E:unramified matrix coefficient}, \ref{L:matrix coefficient nonarchimedean case}, (\ref{E:equivariant property of test vector case v=p not divide N_1N_2}), and \cite[Proposition 6.2]{LNM1254} (see also \cite[Proposition 3]{LR2005}, \cite[Proposition 6.3]{HN2017}), we have
\[
\int_{H_1(k_v)}\Phi(h_{1})\Psi(h_{1})\,dh_{1}=\frac{L(1,\sigma,{\rm std})}{\zeta(2)\zeta(4)}.
\]

Now we assume $v \mid \frak{n}$. The calculations for $v\mid \frak{n}_1$ and $v\mid \frak{n}_2$ are similar and we assume $v \mid \frak{n}_1$. Since $v \in \frak{S}$ and ${\bf f}^\sharp = (W,0) \in \mathcal{V}^{\sharp}$, we have \begin{align*}
\Psi(h) &= 2^{-1}\mathcal{B}(\sigma(h)W,W),\\
\Psi_{ }(h{\bf t}_{ }) &= \mathcal{B}_{ }^{\sharp}((0,\sigma_{ }({\rm Ad}({\bf t}_{ })h)W_{ }),(W_{ },0))=0
\end{align*}
for $h \in H^{\circ}(k_v)$. Therefore
\[
Z = 2^{-2}\cdot\frac{\zeta(2)\zeta(4)}{L(1,\sigma,{\rm std})}\cdot\int_{H_1^\circ(k_v)}\Phi(h_1)\mathcal{B}(\sigma(h_1)W,W)\,dh_1.
\]
%In this case, we have
%$$\sigma_{1} = {\rm St}\otimes \mu,\quad \sigma_{2}  = {\rm Ind}_{B(k_v)}^{\GL_2(k_v)}(\chi\boxtimes\chi^{-1})$$ for some unramified characters $\chi$  and $\mu$ of $k_v^{\times}$ such that $\chi= |\mbox{ }|^s $ with $|{\rm Re}(s)|<1/2$ and $\mu^2=1$. Here ${\rm St}$ is the Steinberg representation of $\GL_2(k_v)$. 
%Put
%$$\alpha = \chi(\varpi),\quad \varepsilon = \mu(\varpi).$$
Let $\mathcal{U}$ be the open compact subgroup of $H_1^{\circ}(k_v)$ defined by
\[\mathcal{U}=H_1^{\circ}(k_v)\cap (K_0(\varpi)\times \GL_2(\o))/\o^{\times}.\]
For $n,m \in \Z$, let 
\[
h_{n,m} = [{\bf a}(\varpi^n){\bf d}(\varpi^m),{\bf a}(\varpi^{n+m})],\quad h_{n,m}' = [w{\bf a}(\varpi^n){\bf d}(\varpi^m),{\bf a}(\varpi^{n+m})]
\] 
with the notation of \S\,\ref{S:notation}.
Note that
\[\left\{h_{n,m},h_{n,m}' \mbox{ }\vert\mbox{ }n,m \in \Z,\, n+m \geq 0\right\}\]
is a complete set of coset representatives for $\mathcal{U} \backslash H_1^{\circ}(k_v) / \mathcal{U}$, and we have
\begin{align}\label{E:cardinalities for double cosets}
\begin{split}
|\,\mathcal{U}h_{n,m}\mathcal{U} / \mathcal{U}\,| &=  \begin{cases} q^{|n-m|} & \mbox{ if }n+m=0,\\  q^{n+m+|n-m|}(1+q^{-1}) & \mbox{ if }n+m \geq 1,   \end{cases}\\
|\,\mathcal{U}h_{n,m}'\mathcal{U} / \mathcal{U}\,| &=\begin{cases} q^{|n-m-1|} & \mbox{ if }n+m=0,\\  q^{n+m+|n-m-1|}(1+q^{-1}) & \mbox{ if }n+m \geq 1.   \end{cases}
\end{split}
\end{align}
%By formulae for matrix coefficients in \cite{Macdonald1971} and \cite[\S 7]{GJ1972}, we have
%\begin{align}%\label{E:matrix coefficients}
%\begin{split}
%\mathcal{B}_{ }(\sigma_{ }(h_{n,m})W_{ },W_{ }) &= \varepsilon^{n-m}\cdot\frac{q^{-|n-m|-(n+m)/2}}{1+q^{-1}}\cdot\left(\alpha^{n+m}\cdot\frac{1-\alpha^{-2}q^{-1}}{1-\alpha^{-2}}+\alpha^{-(n+m)}\cdot\frac{1-\alpha^2q^{-1}}{1-\alpha^2}\right),\\
%\mathcal{B}_{ }(\sigma_{ }(h_{n,m}')W_{ },W_{ }) &= -\varepsilon^{n-m}\cdot\frac{q^{-|n-m-1|-(n+m)/2}}{1+q^{-1}}\cdot\left(\alpha^{n+m}\cdot\frac{1-\alpha^{-2}q^{-1}}{1-\alpha^{-2}}+\alpha^{-(n+m)}\cdot\frac{1-\alpha^2q^{-1}}{1-\alpha^2}\right).
%\end{split}
%\end{align}
By Lemmas \ref{E:matrix coefficients}, \ref{L:matrix coefficient nonarchimedean case}, (\ref{E:equivariant property of test vector case v=p divide N_1}), and (\ref{E:cardinalities for double cosets}), we have
\begin{align*}
{\rm vol}(\mathcal{U},dh_1)^{-1} \int_{H_1^{\circ}(k_v)}\Phi_{ }(h_{1})\mathcal{B}_{ }(\sigma_{ }(h_{1})W_{ },W_{ })\,dh_{1} & = \sum_{n,m \in \Z, \, n+m\geq 0}\Phi(h_{n,m})\mathcal{B}_{ }(\sigma_{ }(h_{n,m})W_{ },W_{ }) \cdot|\,\mathcal{U}h_{n,m}\mathcal{U} / \mathcal{U}\,|\\
& + \sum_{n,m \in \Z, \, n+m\geq 0}\Phi(h_{n,m}')\mathcal{B}_{ }(\sigma_{ }(h_{n,m}')W_{ },W_{ }) \cdot|\,\mathcal{U}h_{n,m}'\mathcal{U} / \mathcal{U}\,|\\
&=\frac{\zeta(2)}{\zeta(1)}\cdot(Z^{(1)}+Z^{(2)}+Z^{(3)}+Z^{(4)}),
\end{align*}
where
\begin{align*}
Z^{(1)} &= \sum_{n\in\Z}\left(q^{-4|n|-2}-q^{-2|n|-2|n-1|-2}\right),\\
Z^{(2)} &= (q^{-2}-q^{-4})\sum_{m=1}^{\infty}\varepsilon^mq^{-3m/2}\left(\alpha^{m}\cdot\frac{1-\alpha^{-2}q^{-1}}{1-\alpha^{-2}}+\alpha^{-m}\cdot\frac{1-\alpha^2q^{-1}}{1-\alpha^2}\right),\\
Z^{(3)}& = ({q^{-2}-q^{-4}})\sum_{n=1}^{\infty}\sum_{m=n+1}^{\infty}\varepsilon^{-n-m}q^{-5n/2-3m/2}\left(\alpha^{-n+m}\cdot\frac{1-\alpha^{-2}q^{-1}}{1-\alpha^{-2}}+\alpha^{n-m}\cdot\frac{1-\alpha^2q^{-1}}{1-\alpha^2}\right),\\
Z^{(4)}& = (q^{-2}-1)\sum_{m=1}^{\infty}\sum_{n=m+1}^{\infty}\varepsilon^{n+m}q^{-3n/2-5m/2}\left(\alpha^{n-m}\cdot\frac{1-\alpha^{-2}q^{-1}}{1-\alpha^{-2}}+\alpha^{-n+m}\cdot\frac{1-\alpha^2q^{-1}}{1-\alpha^2}\right),
\end{align*}
with $\alpha$ and $\varepsilon$ as in (\ref{E:symbols}).
By direct calculations, we have
\begin{align*}
Z^{(1)} & = \frac{(q-q^{-1})^2}{q^4-1},\\
Z^{(2)} & = (q^{-2}-q^{-4})\left[\varepsilon q^{-3/2}(\alpha+\alpha^{-1})-q^{-3}-q^{-4}\right](1-\varepsilon\alpha q^{-3/2})^{-1}(1-\varepsilon\alpha^{-1} q^{-3/2})^{-1},\\
Z^{(3)} & = \frac{q^{-2}-q^{-4}}{q^4-1}\cdot\left[\varepsilon q^{-3/2}(\alpha+\alpha^{-1})-q^{-3}-q^{-4}\right](1-\varepsilon\alpha q^{-3/2})^{-1}(1-\varepsilon\alpha^{-1} q^{-3/2})^{-1},\\
Z^{(4)} & = \frac{q^{-2}-1}{q^4-1}\cdot\left[\varepsilon q^{-3/2}(\alpha+\alpha^{-1})-q^{-3}-q^{-4}\right](1-\varepsilon\alpha q^{-3/2})^{-1}(1-\varepsilon\alpha^{-1} q^{-3/2})^{-1}.
\end{align*}
Therefore, noting that ${\rm vol}(\mathcal{U},dh_1) = (1+q)^{-1}$ by the normalization (\ref{E:non-archi. integration1}) and that
\[
 L(s,\sigma,{\rm std}) = (1 - \varepsilon \alpha q^{-s-1/2})^{-1} (1 - \varepsilon \alpha^{-1} q^{-s-1/2})^{-1},
\]
we have
\begin{align*}
\int_{H_1^{\circ}(k_v)}\Phi_{ }(h_{1})\mathcal{B}_{ }(\sigma_{ }(h_{1})W_{ },W_{ })dh_{1} & = q^{-3} (1-q^{-1})^2  (1 - \varepsilon \alpha q^{-3/2})^{-1} (1 - \varepsilon \alpha^{-1} q^{-3/2})^{-1}\\
& = \frac{L(1,\sigma,{\rm std})}{\zeta(2)\zeta(4)}\cdot q^{-3}\frac{\zeta(2)\zeta(4)}{\zeta(1)^2} .
\end{align*}
This completes the proof.
\end{proof}

\begin{lemma}\label{L:archi. d.s local zeta integral}
Let $v \in S({\rm DS})$. We have
\[Z_{v} = 2^{-\lambda_{1,v}-\lambda_{2,v}-3}(1+\lambda_{1,v}-\lambda_{2,v})^{-1}\cdot\begin{cases} 1 & \mbox{ if }v \notin \frak{S},
\\2^{-2} & \mbox{ if }v \in \frak{S}. \end{cases}\]
\end{lemma}

\begin{proof}
We identify $k_v=\R$ and drop the subscript $v$ for brevity. Without loss of generality, we assume $\kappa_1 \ge \kappa_2$. 
If $v \notin \frak{S}$, then $\kappa_1=\kappa_2$ and we have  
$\omega_{ }(1,{\bf t}_{ })\varphi_{ }=(-1)^{\lambda_1}\varphi_{ }=\varphi_{ }$, $\sigma_{ }^{\sharp}({\bf t}_{ })W=W$, and $\Psi(h) = \mathcal{B}(\sigma(h)W, W)$ for $h \in H^\circ(\R)$. If $v \in \frak{S}$, then  
\begin{align*}
\Psi(h) &= 2^{-1}\mathcal{B}(\sigma(h)W,W),\\
\Psi_{ }(h{\bf t}_{ }) &= \mathcal{B}_{ }^{\sharp}((0,\sigma_{ }({\rm Ad}({\bf t}_{ })h)W_{ }),(W_{ },0))=0
\end{align*}
for $h \in H^{\circ}(\R)$. Therefore, we have
\[
Z_{ } = \frac{\zeta(2)\zeta(4)}{L(1,\sigma,{\rm std})}\cdot\int_{H_1^{\circ}(\R)}\Phi_{ }(h_{1})\mathcal{B}_{ }(\sigma_{ }(h_{1})W_{ },W_{ })\,dh_{1}\cdot \begin{cases} 1 & \mbox{ if }v \notin \frak{S},
\\2^{-2} & \mbox{ if }v \in \frak{S}. \end{cases}
\]
%Note that
%\begin{align*}
%W_{i}({\bf a}(y)k_{\theta})=\begin{cases} e^{\sqrt{-1}\kappa_i\theta}y^{\kappa_i/2}e^{-2\pi y}&\mbox{ if }y>0,\\
%0&\mbox{ if }y<0,\end{cases}
%\end{align*}
%for $i=1,2$, $y \in \R^{\times}$ and $ k_{\theta} \in {\rm SO}(2)$. From the formulae we can deduce that the matrix coefficient $\Psi$ is supported in $H(\R)^0$ and 
%\begin{align}\label{E:coefficient}
%\mathcal{B}_{ }(\sigma_{ }(h_{1})W_{ },W_{ }) & = 2^{-2\lambda_1-2}e^{\sqrt{-1}(\kappa_1\theta_1+\kappa_1\theta_3+\kappa_2\theta_2+\kappa_2\theta_4)}\cosh(t_1)^{-\kappa_1}\cosh(t_2)^{-\kappa_2}
%\end{align}
%for $h_{1} = [k_{\theta_1},k_{\theta_2}][{\bf m}(e^{t_1}),{\bf m}(e^{t_2})][k_{\theta_3},k_{\theta_4}]$ with $k_{\theta_1},k_{\theta_2},k_{\theta_3},k_{\theta_4} \in {\rm SO}(2)$ and $t_1, t_2 >0$.
By (\ref{E:archi. integration}), (\ref{E:equivariant property d.s.}), and Lemmas \ref{E:coefficient}, \ref{L:matrix coeff. of Weil rep d.s}, we have
\begin{align*}
&\int_{H_1^{\circ}(\R)}\Phi_{ }(h_{1})\mathcal{B}_{ }(\sigma_{ }(h_{1})W_{ },W_{ })\,dh_{1} \\
&= 2^{-2-2\kappa_1-\kappa_2}\pi^{2-\kappa_1}\Gamma(\lambda_1+1)\Gamma(-\lambda_2+1)\\
&\times \int_{0}^{ \infty}\int_{0}^{\infty } \cosh(t_1+t_2)^{-2}\cosh(t_1-t_2)^{-2}\left(\cosh(t_1+t_2)^{-1}+\cosh(t_1-t_2)^{-1}\right)^{\kappa_1}\\
&\quad\quad\quad\quad\quad\quad\quad\quad\quad\quad\quad\quad\times\cosh(t_1)^{-\kappa_1}\cosh(t_2)^{-\kappa_2}\sinh(2t_1)\sinh(2t_2)\,dt_1\,dt_2\\
&=2^{-\kappa_2+1}\pi^{2-\kappa_1}\Gamma(\lambda_1+1)\Gamma(-\lambda_2+1)\\
&\times \int_{0}^{\infty }\int_{0}^{\infty } \frac{\cosh(t_2)^{-2\lambda_2+1}\sinh(2t_1)\sinh(t_2)}{(\cosh(2t_1)+\cosh(2t_2))^{2+\kappa_1}}\,dt_1\,dt_2.
\end{align*}
Here the last equality follows from the formulae
\begin{align*}
\cosh(t_1+t_2)\cosh(t_1-t_2) & = 2^{-1}(\cosh(2t_1)+\cosh(2t_2)),\\
\cosh(t_1+t_2)+\cosh(t_1-t_2) & = 2\cosh(t_1)\cosh(t_2),\\
\sinh(2 t_2) &= 2 \sinh(t_2) \cosh(t_2).
\end{align*}
Moreover, the above integral is equal to 
\begin{align*}
&\int_{0}^{\infty }\cosh(t_2)^{-2\lambda_2+1}\sinh(t_2)\left(\int_{0}^{ \infty} \frac{\sinh(2t_1)}{(\cosh(2t_1)+\cosh(2t_2))^{2+\kappa_1}}\,dt_1\right)\,dt_2\\
&=2^{-1}(1+\kappa_1)^{-1}\int_{0}^{\infty }\frac{\cosh(t_2)^{-2\lambda_2+1}\sinh(t_2)}{(1+\cosh(2t_2))^{1+\kappa_1}}\,dt_2\\
&=2^{-2-\kappa_1}(1+\kappa_1)^{-1}\int_{0}^{ \infty}\cosh(t_2)^{-2\lambda_1-1}\sinh(t_2)\,dt_2\\
&=2^{-3-\kappa_1}\lambda_1^{-1}(1+\kappa_1)^{-1}.
\end{align*}
Noting that 
\[
 L(s,\sigma,{\rm std}) = 2^2 (2 \pi)^{-2s-\lambda_1+\lambda_2+1} \Gamma(s+\lambda_1-1) \Gamma(s-\lambda_2), 
\]
we conclude that
\begin{align*}
\int_{H_1^{\circ}(\R)}\Phi_{ }(h_{1})\mathcal{B}_{ }(\sigma_{ }(h_{1})W_{ },W_{ })\,dh_{1} &= 2^{-2-2\lambda_1}(1+\kappa_1)^{-1}\pi^{2-\kappa_1}\Gamma(\lambda_1)\Gamma(-\lambda_2+1)\\
&=\frac{L(1,\,\sigma,\,{\rm std})}{\zeta_{ }(2)\zeta_{ }(4)}\cdot 2^{-\lambda_1-\lambda_2-3}(1+\kappa_1)^{-1}.
\end{align*}
This completes the proof.
\end{proof}

\begin{lemma}\label{L:archi. spherical local zeta integral}
Let $v \in S({\rm PS})$. We have
\begin{align*}
Z_v &= 2^{-4}.
\end{align*}
\end{lemma}

\begin{proof}
We identify $k_v=\R$ and drop the subscript $v$ for brevity. We follow the doubling method in \cite[\S 6]{LNM1254} and \cite{LR2005}. 
Let
$${\rm O}_{2,2} = \left\{ h \in \GL_4 \mbox{ }\left\vert\mbox{ } h \bp 0 & {\bf 1}_2 \\ {\bf 1}_2 & 0\ep  {}^t \! h = \bp 0 & {\bf 1}_2 \\ {\bf 1}_2 & 0\ep \right. \right\} .$$
We identify $H_1(\R)$ with ${\rm O}_{2,2}(\R)$ with respect to the following basis of $V$:
\begin{align*}
	\left\{\bp 0 & 1 \\ 0 & 0 \ep,\, \bp 1 & 0 \\ 0 & 0 \ep,\,  \bp 0 & 0 \\ -1 & 0 \ep,\, \bp 0 & 0 \\ 0 & 1\ep\right\}.
\end{align*}
%\begin{align*}
%	\left\{\bp 0 & 0 \\ 1 & 0 \ep,\, \bp 1 & 0 \\ 0 & 0 \ep,\,  \bp 0 & -1 \\ 0 & 0 \ep,\, \bp 0 & 0 \\ 0 & 1\ep\right\}.
%\end{align*}
Let $Q$ be the standard Siegel parabolic subgroup of ${\rm O}_{2,2}$, $N$ its unipotent radical, and $K_{2,2} = {\rm O}_{2,2}(\R)\cap{\rm O}(4)$ a maximal compact subgroup of ${\rm O}_{2,2}(\R)$. We regard $\GL_2$ as a Levi subgroup of $Q$ via the embedding 
\begin{align*}
\GL_2 \longrightarrow Q,\quad
g  \longmapsto \underline{g}=\bp g& 0 \\ 0 &{}^t \! g^{-1}\ep.
\end{align*}
Then we have
$$\sigma^{\sharp}\vert_{{\rm O}_{2,2}(\R)} = {\rm Ind}_{Q(\R)}^{{\rm O}_{2,2}(\R)}(\tau),$$
where
\[\tau = {\rm Ind}_{B(\R)}^{\GL_2(\R)}(|\mbox{ }|^{\mu_1+\mu_2}\boxtimes |\mbox{ }|^{\mu_1-\mu_2}).\]
Denote by $\omega_\tau$ and $\omega_{\tau^{\vee}}$ the bi-${\rm O}(2)$-invariant matrix coefficients of $\tau$ and $\tau^{\vee}$, respectively, normalized so that $\omega_\tau(1)=\omega_{\tau^{\vee}}(1)=1$.
We identify ${\rm O}_{2,2}\times{\rm O}_{2,2}$ with its image under the embedding 
\[
{\rm O}_{2,2} \times{\rm O}_{2,2}\longrightarrow {\rm O}_{4,4},\quad	\left(\bp a & b \\ c & d \ep , \bp  a' & b' \\ c' & d' \ep \right) \longrightarrow \bp a & 0 & b & 0 \\ 0 & a' & 0 & -b' \\ c & 0 & d & 0 \\ 0 & -c' & 0 & d'  \ep.
\]
Let $P$ be the standard Siegel parabolic subgroup of ${\rm O}_{4,4}$ and $K_{4,4} = {\rm O}_{4,4}(\R) \cap {\rm O}(8)$ a maximal compact subgroup of ${\rm O}_{4,4}(\R)$. We regard $\GL_4$ as a Levi subgroup of $P$ via the embedding 
\[
\GL_4 \longrightarrow P,\quad
g  \longmapsto \underline{g}=\bp g& 0 \\ 0 &{}^t \! g^{-1}\ep.
\]
Let $f^o$ be the right $K_{4,4}$-invariant section of the degenerate principal series representation ${\rm Ind}_{P(\R)}^{{\rm O}_{4,4}(\R)}(\delta_P^{s/3})$ normalized so that $f^o(\delta,s)=1$, where
\[
 \delta =
 \begin{pmatrix}
  0 & 0 & \frac{1}{2} \1_2 & \frac{1}{2} \1_2 \\
  \frac{1}{2} \1_2 & -\frac{1}{2} \1_2 & 0 & 0 \\
  \1_2 & \1_2 & 0 & 0 \\
  0 & 0 & \1_2 & -\1_2
 \end{pmatrix}.
\]

Consider the integral
\[\calZ(s) = \int_{{\rm O}_{2,2}(\R)}f^o(\delta(h,1),s)\Psi(h)\,dh.\]
By \cite[Lemma 7.7]{GI2011}, the integral $\calZ(s)$ is absolutely convergent at $s=1/2$.
By Lemma \ref{L:matrix coeff. of Weil rep PS} and \cite[Proposition 6.4]{LNM1254}, we have
\[\Phi(h) = 2^{-4}\cdot f^o\left(\delta(h,1),\frac{1}{2}\right)\]
for $h \in {\rm O}_{2,2}(\R)$. Therefore
\[Z = 2^{-4}\cdot\frac{\zeta(2)\zeta(4)}{L(1,\sigma,{\rm std})}\cdot\calZ\left(\frac{1}{2}\right).\]
Let $\<\mbox{ },\mbox{ }\>_{\tau} : \tau \times \tau^\vee \rightarrow \C$ be a non-zero invariant pairing. Let
$\phi \in {\rm Ind}_{Q(\R)}^{{\rm O}_{2,2}(\R)}(\tau)$ and $\phi^\vee \in {\rm Ind}_{Q(\R)}^{{\rm O}_{2,2}(\R)}(\tau^\vee)$ be non-zero $K_{2,2}$-invariant sections. Then
\[
\Psi(h) = \Psi(1)\cdot\int_{K_{2,2}}\frac{\<\phi(kh),\phi^\vee(k)\>_{\tau}}{\<\phi(1),\phi^\vee(1)\>_\tau}\,dk
\]
for $h \in {\rm O}_{2,2}(\R)$. Note that $\Psi(1)=2^{-4}$ by Lemma \ref{L:matrix coefficient archimedean case}. Therefore, we have
\begin{align*}
\calZ(s) & = 2^{-4}\int_{{\rm O}_{2,2}(\R)}f^o(\delta(h,1),s)\int_{K_{2,2}} \frac{\<\phi(kh),\phi^\vee(k)\>_{\tau}}{\<\phi(1),\phi^\vee(1)\>_\tau} \,dk\,dh\\
& = 2^{-4}\int_{K_{2,2}}\int_{{\rm O}_{2,2}(\R)}f^o(\delta(h,k),s) \frac{\<\phi(h),\phi^\vee(k)\>_{\tau}}{\<\phi(1),\phi^\vee(1)\>_\tau} \,dh\,dk \\
& = 2^{-4}\int_{K_{2,2}\times K_{2,2}} \int_{\GL_2(\R)} \int_{N(\R)} \delta_Q(\underline{g})^{-1/2} f^o(\delta(u\underline{g}k_1,k_2),s)\omega_{\tau}(g) \,du\,dg\,dk_1\,dk_2\\
& = 2^{-4}\int_{\GL_2(\R)} \int_{N(\R)} \delta_Q(\underline{g})^{-1/2} f^o(\delta(u\underline{g},1),s)\omega_{\tau}(g) \,du\,dg.
\end{align*}
Let 
%$$A= \bp 0 & {\bf 1}_2 \\ {\bf 1}_2 & -{\bf 1}_2\ep.$$ 
\[A = \bp 0 & {\bf 1}_2 \\ {\bf 1}_2 & {\bf 1}_2\ep.\]
Then the  function 
%$$F_1^o(g,s) = |\det(g)|^{1/2}\int_{U(\R)}f^o(\delta (u,1)\underline{A}^{-1}\underline{g},s) \,du$$
\[F_1^o(g,s) = |\det(g)|^{-1/2}\int_{N(\R)}f^o(\delta (u,1)\underline{A}^{-1}\underline{g},s) \,du\]
for $g \in \GL_4(\R)$ defines an ${\rm O}(4)$-invariant section of ${\rm Ind}_{P'(\R)}^{\GL_4(\R)}(\delta_{P'}^{s/2})$, where
\[P' = \left\{\left. \bp a & * \\ 0 & d \ep\mbox{ }\right\vert \mbox{ }a,d \in \GL_2 \right\}.\]
On the other hand, an ${\rm O}(4)$-invariant section of ${\rm Ind}_{P'(\R)}^{\GL_4(\R)}(\delta_{P'}^{s/2})$ is also given by
\[F_2^o(g,s) = |\det(g)|^{s+1}\int_{\GL_2(\R)}\Phi^o((0,x)g)|\det(x)|^{2s+2}\,dx\]
for $g \in \GL_4(\R)$, where $\Phi^o \in \mathcal{S}({\rm M}_{2,4}(\R))$ is defined by
\[\Phi^o(x) = e^{-\pi\,{\rm tr}(x{}^t \! x)}.\]
Therefore, 
\[F_1^o = \frac{F_1^o({A},s)}{F_2^o({A},s)}\cdot F_2^o.\]
By direct calculations, we have
\[F_1^o({A},s) = 2\cdot\frac{\zeta(2s+2)}{\zeta(2s+3)},\quad F_2^o({A},s) = 2^{-2s-2}\zeta(2s+1)\zeta(2s+2).\]
%A direct calculation shows that
%\begin{align*}
%F_2^o({A},s) = 2^{-2s-2}\zeta(2s+1)\zeta(2s+2).\end{align*}
Hence
\[
\frac{F_1^o({A},s)}{F_2^o({A},s)} = 2^{2s+3}\zeta(2s+1)^{-1}\zeta(2s+3)^{-1}.\]
We conclude that
\begin{align*}
& \int_{\GL_2(\R)} \int_{N(\R)} \delta_Q(\underline{g})^{-1/2} f^o(\delta(u\underline{g},1),s)\omega_{\tau}(g) \,du\,dg\\
& = \int_{\GL_2(\R)} F_1^o\left({A}\bp g & 0 \\ 0 & {\bf 1}_2\ep,s \right)\omega_{\tau}(g) \,dg\\
& = 2^{2s+3}\zeta(2s+1)^{-1}\zeta(2s+3)^{-1}\int_{\GL_2(\R)}|\det(g_1)|^{s+1}\left(\int_{\GL_2(\R)}\Phi^o(g_2g_1,g_2)|\det(g_2)|^{2s+2}\,dg_2\right)\omega_{\tau}(g_1)\,dg_1.
\end{align*}
Noting that $\omega_\tau$ is a zonal spherical function on $\GL_2(\R)$ with respect to ${\rm O}(2)$ % in the sense of \cite{Macdonald1971} with the space $\mathcal{L}(\GL_2(\R),{\rm O}(2))$ replaced by the space of bi-${\rm O}(2)$-invariant functions in $C_c^{\infty}(\GL_2(\R))$.
and following the argument in the proof of \cite[Proposition (1.2.5)]{Macdonald1971}, we have
\[
\int_{{\rm O}(2)}\omega_{\tau}(g_1kg_2)\,dk = \omega_{\tau}(g_1)\omega_{\tau}(g_2)
\]
for $g_1,g_2 \in \GL_2(\R)$. Therefore, proceeding as in the proof of \cite[Proposition 6.1]{LNM1254}, we have
\begin{align*}
&\int_{\GL_2(\R)}|\det(g_1)|^{s+1}\left(\int_{\GL_2(\R)}\Phi^o(g_2g_1,g_2)|\det(g_2)|^{2s+2}\,dg_2\right)\omega_{\tau}(g_1)\,dg_1\\
&=\int_{\GL_2(\R)}e^{-\pi\,{\rm tr}(g{}^t \! g)}\omega_\tau(g)|\det(g)|^{s+1}\,dg \int_{\GL_2(\R)}e^{-\pi\,{\rm tr}(g{}^t \! g)}\omega_{\tau^{\vee}}(g)|\det(g)|^{s+1}\,dg\\
&=L\left(s+\frac{1}{2},\sigma,{\rm std}\right).
\end{align*}
Here the last equality follows from a calculation analogous to \cite[Lemma 6.10]{GJ1972}.
It follows that
\[
\calZ(s) = 2^{2s-1}\cdot\frac{L(s+1/2,\sigma,{\rm std})}{\zeta(2s+1)\zeta(2s+3)}.
\]
This completes the proof.
\end{proof}

\subsection{Proof of Proposition \ref{P:Rallis inner product}}\label{S:proof 2}
The assertion follows from Theorem \ref{T:RIF} and Lemmas \ref{L:zeta integral1}-\ref{L:archi. spherical local zeta integral}.

\section{Convergence lemmas}

We keep the notation of \S\,\ref{SS:analytic family}. In this section, we study the convergence of the doubling local zeta integrals and the Rankin-Selberg local zeta integrals defined in (\ref{E:local zeta1}) and (\ref{E:local zeta2}), respectively. 

Except in Lemma \ref{L:local zeta Ur}, let $\pi$ be a representation of $G(F)$ in one of the three types in \S\,\ref{SS:analytic family}. In Case (IIa) and Case (PS), we fix a sign of representations and write $\pi=\pi_\lambda$ for the representation with parameter $\lambda$. Let $\Omega$ be the domain defined by excluding the points of reducibility of the induced representations (\ref{E:induced repre IIa}) and (\ref{E:induced repre PS}) in Case (IIa) and Case (PS), respectively. Recall the domain 
\[
\mathcal{D} = 
\begin{cases}
\left\{\lambda \in \C \mbox{ }\vert\mbox{ } |{\rm Re}(\lambda) |<1/2 \right\} & \mbox{ in Case (IIa)},\\
\left\{\lambda=(\lambda_1,\lambda_2) \in \C^2 \mbox{ }\vert\mbox{ }|{\rm Re}(\lambda_1)|+|{\rm Re}(\lambda_2)|<1 \right\} & \mbox{ in Case (PS)},
\end{cases}
\]
defined in (\ref{E:domain}).
Let 
\[
K= \begin{cases}
G(\o) & \mbox{ if $F$ is non-archimedean},\\
G(\R) \cap {\rm O}(4) & \mbox{ if $F=\R$}.
\end{cases}
\]
%By a $K$-type we mean a finite dimensional representation of $K$.

\subsection{Doubling local zeta integrals}\label{SS:Doubling local zeta integrals}
Recall $P=P_{4,4}$ and $I(s)=I_{4,4}(s)$ in the notation of \S\,\ref{ss:eisenstein}. Denote by $\mathcal{C}(\pi)$ the space of matrix coefficients of $\pi$. If $\phi \in \mathcal{C}(\pi)$  and $F\in I(s)$ is a holomorphic section, let $Z(s,\phi,F)$ be the local zeta integral defined as in (\ref{E:local zeta1}). %Here the definition of the matrix coefficient $\phi_\pi$ is explained in \S\,\ref{SS:analytic family}.

\begin{lemma}\label{L:local zeta Ur}
Assume $\pi$ is an irreducible generic admissible unramified representation of $G(F)$ with trivial central character. Write 
\begin{align*}
\pi \vert_{\Sp_4(F)} = {\rm Ind}_{\Sp_4(F)\cap {\bf B}(F)}^{\Sp_4(F)}(|\mbox{ }|^{\lambda_{1}}\boxtimes|\mbox{ }|^{\lambda_{2}})
\end{align*}
for some $\lambda_1$, $\lambda_2 \in \C$. 
Let $\phi\in\mathcal{C}(\pi)$ and let $F \in I(s)$ be a holomorphic section. 
The integral $Z(s,\phi,F)$ is absolutely convergent for ${\rm Re}(s)>-1/2+\max\{|{\rm Re}(\lambda_1)|,|{\rm Re}(\lambda_2)|\}$.
\end{lemma}

\begin{proof}
Let $s \in \RR$.
Let $K_1 = \Sp_4(\frko)$ and
\[
 A^{+} =
 \{ \diag( \varpi^{n_1}, \varpi^{n_2}, \varpi^{-n_1}, \varpi^{-n_2} )
 \, | \, n_1 \ge n_2 \ge 0 \}.
\]
Then $\Sp_4(F) = K_1 A^{+} K_1$.
We may assume that $\phi$ is bi-$G(\frak{o})$-invariant,  $F$ is $\GSp_8(\frko)$-invariant, and $F(1,s) = 1$.
In particular,
\[
 F(\delta (k_1 g k_2, 1), s) = F(\delta (g, 1), s)
\]
for all $k_1, k_2 \in K_1$.
Let
\[
 a = 
 \diag( \varpi^{n_1}, \varpi^{n_2}, \varpi^{-n_1}, \varpi^{-n_2} )
 \in A^{+}.
\]
By \cite[p.~241]{Wald2003},
there exists a constant $C > 0$ such that
\[
 \vol(K_1aK_1) \le C q^{4n_1 + 2n_2}.
\]
By \cite[Proposition 6.4]{LNM1254},
\[
 F(\delta (a, 1), s) = q^{- (s + {5}/{2})(n_1 + n_2)}.
\]
By Macdonald's formula \cite{Macdonald1971}, \cite{Casselman1980},
there exists a polynomial $\Psi$ such that
\[
 |\phi(a)| \le \Psi(n_1) q^{-2n_1 - n_2 + \eta(n_1 + n_2)},
\]
where $\eta = \max\{|{\rm Re}(\lambda_{1})|,|{\rm Re}(\lambda_{2})|\}$. 
Hence $|Z(s, \phi, F)|$ is majorized by
\begin{align*}
 & \sum_{a \in A^{+}} \vol(K_1aK_1) F(\delta (a,1), s) |\phi(a)| \\
 & \le C \sum_{n_2 = 0}^{\infty}\sum_{n_1 = n_2}^{\infty}
 \Psi(n_1)
 q^{- (s + {1}/{2} - \eta) n_1} q^{- (s + {3}/{2} - \eta) n_2}.
\end{align*}
This completes the proof.
\end{proof}

\begin{lemma}\label{L:local zeta DS}
Assume $\pi$ is of type (DS). Let $\phi\in\mathcal{C}(\pi)$  and let $F \in I(s)$ be a holomorphic section.  The integral $Z(s, \phi, F)$ is 
absolutely convergent for $\Re(s) > -{1}/{2}$.
\end{lemma}

\begin{proof}
Fix $s > - {1}/{2}$.
By H\"older's inequality,
it suffices to show that the integral
\begin{equation} \label{eq:doubling_infty}
 \int_{\Sp_4(\R)} |F(\delta (g,1), s)|^{2(1-\epsilon)} \, dg  
\end{equation}
is convergent for some $\epsilon>0$.
Let $K_1 = \Sp_4(\R) \cap \O(4)$ and
\[
 \frka^{+} = \{ \diag(X_1, X_2, -X_1, -X_2) \, | \, X_1 \ge X_2 \ge 0 \}.
\]
Then $\Sp_4(\R) = K_1 \exp(\frka^{+}) K_1$.
We may assume that $F$ is $(\GSp_8(\R) \cap \O(8))$-invariant and
$F(1,s) = 1$.
In particular,
\[
 F(\delta (k_1 g k_2, 1), s) = F(\delta (g, 1), s)
\]
for all $k_1, k_2 \in K_1$.
Let $\Delta^{+}$ be the set of positive roots
determined by the chamber $\frka^{+}$.
Let
\[
 X = \diag(X_1, X_2, -X_1, -X_2) \in \frka^{+}
\]
and $a_i = \exp(X_i)$.
By \cite[Proposition 6.4]{LNM1254},
\[
 F(\delta (\exp(X), 1), s) =
 \left( \prod_{i=1}^2 \sqrt{(1+a_i^2)(1+a_i^{-2})} \right)^{-s -{5}/{2}}
 \le (a_1 a_2)^{-s-{5}/{2}}.
\]
Obviously,
\[
 \left| \prod_{\alpha \in \Delta^{+}} \sinh(\alpha(X)) \right|
 \le a_1^4 a_2^2.
\]
Hence the integral \eqref{eq:doubling_infty} is majorized by
\begin{align*}
 & \int_{\frka^{+}} \int_{K_1 \times K_1}
 \left\vert F(\delta (k_1 \exp(X) k_2, 1), s)\right\vert^{2(1-\epsilon)}
 \left| \prod_{\alpha \in \Delta^{+}} \sinh(\alpha(X)) \right| 
 \, dk_1 \, dk_2 \, dX \\
 & \le \int_{1}^{\infty}\int_{a_2}^{\infty} 
 a_1^{-2s-1+2\epsilon s +5 \epsilon} 
 a_2^{-2s-3+2\epsilon s +5 \epsilon}
 \, d^{\times} a_1 \, d^{\times} a_2.
\end{align*}
This completes the proof.
\end{proof}

Assume $\pi=\pi_\lambda$ is of type (IIa) or (PS) with parameter $\lambda$. Recall that $P_{2,2}$ is the standard Siegel parabolic subgroup of $G$. We may write
\[
\pi_\lambda = \begin{cases}
\displaystyle{{\rm Ind}_{P_{2,2}(F)}^{G(F)}(\tau_\lambda \boxtimes {\eta^\varepsilon}|\mbox{ }|^{-\lambda})} & \mbox{ in Case (IIa)},\\
\displaystyle{{\rm Ind}_{P_{2,2}(\R)}^{G(\R)}(\tau_\lambda \boxtimes {\rm sgn}^\varepsilon|\mbox{ }|^{(-\lambda_1-\lambda_2)/2})} & \mbox{ in Case (PS)},
\end{cases}
\]
where $\tau_\lambda$ is the representation of $\GL_2(F)$ defined by
\[
\tau_\lambda = \begin{cases}
\displaystyle{{\rm St}\otimes |\mbox{ }|^\lambda} & \mbox{ in Case (IIa)},\\
\displaystyle{{\rm Ind}_{B(\R)}^{\GL_2(\R)}(|\mbox{ }|^{\lambda_1}\boxtimes|\mbox{ }|^{\lambda_2})} & \mbox{ in Case (PS)}.
\end{cases}
\]
Let
\[
{\bf K} = \begin{cases}
\GL_2(\o) & \mbox{ if $\pi$ is of type (IIa)},\\
{\rm O}(2) & \mbox{ if $\pi$ is of type (PS)}.
\end{cases}
\]
Note that the restrictions of $\pi_\lambda$ and $\tau_\lambda$ to $K$ and ${\bf K}$, respectively, do not depend on $\lambda$. We realize the representation $\tau_\lambda$ (resp.~$\tau_\lambda^\vee$) on a space $\mathcal{V}$ (resp.~$\mathcal{V}^\vee$) which does not depend on $\lambda$ and on which the action of $\tau_\lambda \vert_{\bf K}$ (resp.~$\tau_\lambda^\vee \vert_{\bf K}$) does not depend on $\lambda$, and fix a bilinear pairing $\<\mbox{ },\mbox{ }\> : \mathcal{V} \times \mathcal{V}^\vee \rightarrow \C$ which is invariant with respect to the actions of $\tau_\lambda$ and $\tau_\lambda^\vee$ for all $\lambda$.
We call a map
\[
\Omega \longrightarrow C^\infty(G(F)),\quad \lambda \longmapsto \phi_\lambda
\]
a $K$-finite analytic family of matrix coefficients if it satisfies the following conditions:
\begin{itemize}
\item The map $(\lambda,g) \mapsto \phi_\lambda(g)$ is continuous.
\item For each $g \in G(F)$, the map $\lambda \mapsto \phi_\lambda(g)$ is analytic.
\item For each $\lambda \in \Omega$, the function $g \mapsto \phi_\lambda(g)$ belongs to $\mathcal{C}(\pi_\lambda)$.
\item There exist finitely many irreducible representations $\rho_i$ of $K\times K$ such that $\phi_\lambda$ is contained in the direct sum of the $\rho_i$-isotypic subspaces of $\mathcal{C}(\pi_\lambda)$ for all $\lambda \in \Omega$.
\end{itemize}
We call a map 
\[
\Omega \times G(F)  \longrightarrow \mathcal{V},\quad (\lambda,g) \longmapsto h_\lambda(g)
\]
a $K$-finite analytic section of $\pi_\lambda$ if it 
satisfies the following conditions:
\begin{itemize}
\item For each $v \in \mathcal{V}^\vee$, the map $(\lambda,g) \mapsto \<h_\lambda(g),v^\vee\>$ is continuous.
\item For each $g \in G(F)$ and $v \in \mathcal{V}^\vee$, the map $\lambda \mapsto \<h_\lambda(g),v^\vee\>$ is analytic.
\item For each $\lambda \in \Omega$, the map $g \mapsto h_\lambda(g)$ belongs to $\pi_\lambda$.
\item There exist finitely many irreducible representations $\rho_i$ of $K$ such that $h_\lambda$ is contained in the direct sum of the $\rho_i$-isotypic subspaces of $\pi_\lambda$ for all $\lambda \in \Omega$.
\end{itemize}
Similarly, we define the notion of $K$-finite analytic sections for $\pi^\vee_\lambda$. Given $K$-finite analytic sections $h_\lambda$ and $h_\lambda^\vee$ of $\pi_\lambda$ and $\pi^\vee_\lambda$, respectively, it is easy to show that the map
\begin{align}\label{E:explicit analytic family for mc}
\lambda \longmapsto \left[g \mapsto \int_{K\cap \Sp_4(F)}\<h_\lambda(kg),h_\lambda^\vee(k)\>\,dk \right]
\end{align}
defines a $K$-finite analytic family of matrix coefficients. Moreover, for any fixed $\lambda_0 \in \Omega$, any $K$-finite analytic family of matrix coefficients can be written as a linear combination, with analytic functions of $\lambda$ as coefficients, of $K$-finite analytic families of the form (\ref{E:explicit analytic family for mc}) in a neighborhood of $\lambda_0$.

\begin{lemma}\label{L:ab1}
%Assume $\pi$ is of type (IIa) or (PS). 
Let $\phi_\lambda$ be a $K$-finite analytic family of matrix coefficients and $F \in I(s)$ a holomorphic section.  The integral $Z(s,\phi_\lambda,F)$ is absolutely convergent for 
\begin{align*}
\begin{cases}
{\rm Re}(s)>-1/2+\max\{0,|{\rm Re}(\lambda)|-1/2\} & \mbox{ in Case (IIa)},\\ 
{\rm Re}(s)>-1/2+\max\{|{\rm Re}(\lambda_{1})|,|{\rm Re}(\lambda_{2})|\} & \mbox{ in Case (PS)}, 
\end{cases}
\end{align*}
uniformly for $\lambda$ varying in a compact set. In particular, the integral is absolutely convergent for ${\rm Re}(s) \geq 1/2$ if $\lambda \in \mathcal D$.
\end{lemma}

\begin{proof}
Since we only consider the convergence for $\lambda$ varying in a compact set, we may assume $\phi_\lambda$ is of the form (\ref{E:explicit analytic family for mc}) for some $K$-finite analytic sections $h_\lambda$ and $h^\vee_\lambda$ for $\pi_\lambda$ and $\pi^\vee_\lambda$, respectively.
Let $Q$ be the standard Siegel parabolic subgroup of $\Sp_4$ and $N$ its unipotent radical. We identify $\GL_2$ with a Levi subgroup of $Q$ via the embedding 
\[
\GL_2 \longrightarrow Q,\quad g \longmapsto \underline{g} = \bp g & 0 \\ 0 & {}^t \! g^{-1}\ep.
\]
Then
\begin{align*}
Z(s,\phi_\lambda,F) & = \int_{\Sp_4(F)}F(\delta(g,1),s)\int_{K\cap \Sp_4(F)} \<h_\lambda(kg),h_\lambda^\vee(k)\> \,dk\,dg\\
& = \int_{K\cap \Sp_4(F)}\int_{\Sp_4(F)}F(\delta(g,k),s) \<h_\lambda(g),h_\lambda^\vee(k)\> \,dg\,dk \\
& = \int_{(K\cap \Sp_4(F))^2} \int_{\GL_2(F)} \int_{N(F)} \delta_Q(\underline{g})^{-1/2} F(\delta(u\underline{g}k_1,k_2),s)\<\tau_\lambda(g)h_\lambda(k_1),h_\lambda^\vee(k_2)\> \,du\,dg\,dk_1\,dk_2.
\end{align*}
Recall that $P\cap \Sp_8$ is the standard Siegel parabolic subgroup of $\Sp_8$. We identify $\GL_4$ with a Levi subgroup of $P\cap \Sp_8$ via the embedding
\[\GL_4 \longrightarrow P\cap \Sp_8,\quad a \longmapsto \underline{a} = \bp a & 0 \\ 0 & {}^t \! a^{-1}\ep.\] Let $P'$ be the parabolic subgroup of $\GL_4$ defined by
\[P' = \left\{\left. \bp a & * \\ 0 & d \ep\mbox{ }\right\vert \mbox{ }a,d \in \GL_2 \right\}.\]
%Let $Q$ be the standard Siegel parabolic subgroup of $\Sp_4$, and $N$ its unipotent radical. We identify $\GL_2$ with a Levi subgroup of $Q$ via the embedding
%\[\GL_2 \longrightarrow Q,\quad g \longmapsto \underline{g} = \bp g & 0 \\ 0 & {}^tg^{-1}\ep.\] 
%Then \[\pi \vert_{\Sp_4(F)} = {\rm Ind}_{Q(F)}^{\Sp_4(F)}(\tau),\]
%where
%\[
%\tau = \begin{cases} 
%{\rm St}\otimes |\mbox{ }|^{\lambda}& \mbox{ if $\pi$ is of type (IIa)},\\
%{\rm Ind}_{B(\R)}^{\GL_2(\R)}\left(|\mbox{ }|^{\lambda_{1}}\boxtimes|\mbox{ }|^{\lambda_{2}}\right) & \mbox{ if $\pi$ is of type (PS)}.
%\end{cases}
%\]
%Let $\<\cdot,\cdot\>_{\tau} : \tau \times \tau^\vee \rightarrow \C$ be a non-zero invariant pairing. Let 
%\[
%K = \begin{cases}
%\Sp_4(\o) & \mbox{ if $\pi$ is of type (IIa)},\\
%\Sp_4(\R)\cap {\rm O}(4) & \mbox{ if $\pi$ is of type (PS)}.
%\end{cases}
%\]
%There exist analytic families $h \in {\rm Ind}_{Q(F)}^{\Sp_4(F)}(\tau)$ and $h^\vee \in {\rm Ind}_{Q(F)}^{\Sp_4(F)}(\tau^\vee)$ such that
%\[
%\phi_\pi(g) = \int_{K}\<h(kg),h^\vee(k)\>_{\tau}\,dk
%\]
%for $g \in \Sp_4(F)$.
For $g \in \Sp_8(F)$, define $\Psi(g,s) \in {\rm Ind}_{P'(F)}^{\GL_4(F)}(\delta_{P'}^{s/2})$ by
\[\Psi(g,s)(a) = |\det(a)|^{-3/2}\int_{N(F)}F(\delta (u,1)\underline{A}^{-1}\underline{a}\,g,s) \,du\]
for $a \in \GL_4(F).$ Here 
\[A= \bp 0 & {\bf 1}_2 \\ {\bf 1}_2 & -{\bf 1}_2\ep. \]
Note that $\Psi(g,s)$ is an intertwining integral (cf.\,\cite[Lemma 6.2]{LNM1254} and \cite[Proposition 1]{LR2005}) and absolutely convergent for ${\rm Re}(s)>-1/2$.
Then
\[
Z(s,\phi_\lambda,F)  = \int_{(K\cap \Sp_4(F))^2} \int_{\GL_2(F)} \Psi((k_1,k_2),s)\left(A \bp g & 0 \\ 0 & {\bf 1}_2  \ep  \right)\<\tau_\lambda(g)h_\lambda(k_1),h_\lambda^\vee(k_2)\> \,dg\,dk_1\,dk_2.
\]

Let $\eta = \max\{|{\rm Re}(\lambda_{1})|,|{\rm Re}(\lambda_{2})|\}$  and $\mu = \min\{{\rm Re}(\lambda_1),{\rm Re}(\lambda_2)\}$ in Case (PS). Fix $s\in \R$ such that
\[
\begin{cases}
s > -1/2 + \max\{0,|{\rm Re}(\lambda)|-1/2\}  & \mbox{ in Case (IIa)},\\
s > -1/2 + \eta & \mbox{ in Case (PS)}.
\end{cases}
\]
It suffices to show that the integral
\begin{align}\label{eq:doubling_PS}
\int_{(K\cap \Sp_4(F))^2} \int_{\GL_2(F)} \left\vert\Psi((k_1,k_2),s)\left(A \bp g & 0 \\ 0 & {\bf 1}_2  \ep  \right)\<\tau_\lambda(g)h_\lambda(k_1),h_\lambda^\vee(k_2)\> \right\vert \,dg\,dk_1\,dk_2
\end{align}
is uniformly convergent for $\lambda$ varying in a compact set. We may assume the section $F$ is $\GSp_8(\o)$-invariant (resp.~$(\GSp_8(\R)\cap {\rm O}(8))$-invariant) in Case (IIa) (resp.~Case (PS)). Let $f^o \in {\rm Ind}_{P'(F)}^{\GL_4(F)}(\delta_{P'}^{s/2})$ be the $\GL_4(\o)$-invariant section (resp.~${\rm O}(4)$-invariant section) in Case (IIa) (resp.~Case (PS)) defined by
\[f^o(a,s) = |\det(a)|^{s+1}\int_{\GL_2(F)}\Phi^o((0,x)a)|\det(x)|^{2s+2}\,dx\]
for $a \in \GL_4(F)$. Here $\Phi^o \in \mathcal{S}({\rm M}_{2,4}(F))$ is given by
\[
\Phi^o(x) = \begin{cases}
{\mathbb I}_{{\rm M}_{2,4}(\o)}(x) & \mbox{ in Case (IIa)},\\
e^{-\pi\,{\rm tr}(x{}^t \! x)} & \mbox{ in Case (PS)}.
\end{cases}
\]
Note that the above integral is absolutely convergent for ${\rm Re}(s)>-1/2$. Then
\[\Psi(1,s) = \frac{\Psi(1,s)({A})}{f^o({A},s)}\cdot f^o.\]
%Denote by $\zeta(s)$ the local zeta function of $F$. By the Gindikin-Karpelevich formula,
%$$\Psi(1,s)({A}) = \frac{\zeta(s+1/2)\zeta(2s+2)}{\zeta(s+5/2)\zeta(2s+3)}\cdot \begin{cases}
%1 & \mbox{ if $\pi$ is of type (IIa)},\\
%2 & \mbox{ if $\pi$ is of type (PS)}.
%\end{cases}$$
%A direct calculation shows that
%\begin{align*}
%f^o({A},s) = \zeta(2s+1)\zeta(2s+2)\cdot \begin{cases}
%1 & \mbox{ if $\pi$ is of type (IIa)},\\
%2^{-2s-2} & \mbox{ if $\pi$ is of type (PS)}.
%\end{cases}
%\end{align*}
%Therefore, 
%\begin{align*}
%\frac{\Psi(1,s)({A})}{f^o({A},s)} = d_{P}(s)^{-1}\zeta(s+1/2)\cdot \begin{cases}
%1 & \mbox{ if $\pi$ is of type (IIa)},\\
%2^{2s+3} & \mbox{ if $\pi$ is of type (PS)}.
%\end{cases}
%\end{align*}
Put $F^+ = \{ \nu \in F^{\times}\mbox{ }\vert\mbox{ }|\nu| \leq 1\}$. % and 
%\[{\bf K} = \begin{cases}
%\GL_2(\o) & \mbox{ if $\pi$ is of type (IIa)},\\
%{\rm O}(2) & \mbox{ if $\pi$ is of type (PS)}.
%\end{cases}\]
There is a function {$\kappa$}
 on $F^+$ such that $0 \leq {\kappa(\nu)} \leq C\cdot |\nu|^{-1}$ for some constant $C$ and such that 
\[\int_{\GL_2(F)}f(g)\,dg = \int_{F^{\times}}\int_{F^+}\int_{{\bf K}\times {\bf K}}f(tk_1{\bf a}(\nu)k_2){\kappa(\nu)}\,dk_1\,dk_2\,d^{\times}\nu\,d^{\times}t \]
for all $f \in L^1(\GL_2(F))$.
Let $0<\epsilon<s+1/2-\eta$ in Case (PS). 
%By the asymptotic behavior of Whittaker functions on $\GL_2$ (cf.\,\cite[Lemma 14.3]{JLbook2} and \cite[\S\,3]{Jacquet2009}), 
There exists a constant $C_{\lambda}>0$ in Case (IIa) (resp.~$C_{\lambda,\epsilon}>0$ in Case (PS)) bounded uniformly as $\lambda$ varies in a compact set such that
\[
\left\vert\<\tau_\lambda(k_1 {\bf a}(\nu) k_2)h_\lambda(k_3),h_\lambda^\vee(k_4)\>\right\vert \leq  \begin{cases}
C_{\lambda} \cdot |\nu|^{{\rm Re}(\lambda)+1} & \mbox{ in Case (IIa)},\\
C_{\lambda,\epsilon}\cdot |\nu|^{\mu+1/2-\epsilon} & \mbox{ in Case (PS)},
\end{cases}
\]
for all $\nu \in F^+$, $k_1,k_2 \in {\bf K}$, and $k_3,k_4 \in K$. 
In Case (IIa), we have
\[
f^o\left(A\bp \varpi^m{\bf a}(\varpi^n) & 0 \\ 0 & {\bf 1}_2\ep \right) = \zeta(2s+1)\zeta(2s+2)q^{(-n-2m)(s+1)}\cdot\begin{cases}
1 & \mbox{ if $m \geq 0$},\\
q^{m(2s+2)} & \mbox{ if $0 \geq m \geq -n$},\\
q^{(n+2m)(2s+2)} & \mbox{ if $-n\geq m$},
\end{cases}
\]
for $m \in \Z$ and  $n \in \Z_{\geq 0}$.
In Case (PS), we have
\[
f^o\left({A}\bp {t{\bf a}(\nu)} & 0 \\ 0 & {\bf 1}_2 \ep \right) = \zeta(2s+1)\zeta(2s+2)|\nu t^2|^{s+1}(1+\nu^2 t^2)^{-s-1}(1+t^2)^{-s-1}
\]
for $t \in \R^{\times}$ and $\nu \in \R^+$.
By \cite[Lemma 6.1]{LNM1254},  
\[f^o\left({A}\bp {k_1gk_2} & 0 \\ 0  & {\bf 1}_2 \ep\right) = f^o\left(A \bp g & 0 \\ 0 & {\bf 1}_2\ep \right)\]
for all $k_1, k_2 \in {\bf K}$ and $g \in \GL_2(F)$.
Therefore, the integral (\ref{eq:doubling_PS}) is majorized by
\[
\frac{\Psi(1,s)({A})}{f^o({A},s)}\cdot\zeta(2s+1)\zeta(2s+2)C\cdot \begin{cases}
C_{\lambda}\cdot(I_{\lambda}^{(1)}+I_{\lambda}^{(2)}+I_{\lambda}^{(3)}) & \mbox{ in Case (IIa)},\\
C_{\lambda,\epsilon}\cdot I_{\lambda,\epsilon} & \mbox{ in Case (PS)},
\end{cases}
\]
where
\begin{align*}
I_{\lambda}^{(1)} &= \sum_{n=0}^{\infty}\sum_{m=0}^\infty q^{-n(s+{\rm Re}(\lambda)+1)}q^{-m(2s+2+2{\rm Re}(\lambda))},\\
I_{\lambda}^{(2)} &= \sum_{n=0}^{\infty}\sum_{m=1}^nq^{-n(s+{\rm Re}(\lambda)+1)}q^{2m{\rm Re}(\lambda)},\\
I_{\lambda}^{(3)} &= \sum_{n=0}^{\infty}\sum_{m=n+1}^\infty q^{-n(-s+{\rm Re}(\lambda)-1)}q^{-m(2s+2-2{\rm Re}(\lambda))},
\end{align*}
in Case (IIa), and
\[
I_{\lambda,\epsilon} =  \int_{\R^\times}\left(\frac{t^2}{1+t^2}\right)^{s+1}|t|^{-s+{\rm Re}(\lambda_{1})+{\rm Re}(\lambda_{2})-\mu-1/2+\epsilon}\int_{-|t|}^{|t|}\frac{|\nu|^{s+\mu+1/2-\epsilon}}{(1+\nu^2)^{s+1}} \,d^{\times}\nu\,d^\times t
\]
in Case (PS).
This shows the absolute convergence of $Z(s, \phi_\lambda,F)$. It is clear that the integrals $I_{\lambda}^{(1)},I_{\lambda}^{(2)},I_{\lambda}^{(3)},I_{\lambda,\epsilon}$ are bounded uniformly as $\lambda$ varies in a compact set. This completes the proof. 
\end{proof}

\subsection{Local zeta integrals for $\GSp_4 \times \GSp_4$}\label{SS:local zeta for Rankin-Selberg}
Recall $\mathcal{P}=P_{4,3}$ and $\calI(s)=I_{4,3}(s)$ in the notation of \S\,\ref{ss:eisenstein}. Denote by $C^{\infty}(U(F)\backslash G(F),\psi_U)$ the space of smooth Whittaker functions on $G(F)$ with respect to $\psi_U$ and by $\mathcal{W}(\pi,\psi_U)$ the space of smooth Whittaker functions of $\pi$ with respect to $\psi_U$.  If $W \in \mathcal{W}(\pi,\psi_U)$ and $\calF \in \calI(s)$ is a holomorphic section, let $\calZ(s,W,\calF)$ be the local zeta integral defined as in (\ref{E:local zeta2}) with $\overline{W}$ replaced by the left translation of $W$ by $\diag(-1,1,1,-1)$ when $\pi$ is of type (IIa) or (PS). 

We follow the notation in \S\,\ref{SS:theta lifts}. In particular, $V = {\rm M}_{2,2}$ is a quadratic space over $F$ with the norm form.
Let $\omega$ be the Weil representation of $\Sp_4(F)\times H_1(F)$ on $\mathcal{S}(V^2(F))$.
We extend $\omega$ to a representation of ${\rm G}(\Sp_4 \times H_1)(F)$ as in (\ref{Weil rep on R}).  
Let $\sigma_1$ and $\sigma_2$ be the irreducible admissible representations of $\GL_2(F)$ defined as follows:
\begin{itemize}
\item If $\pi$ is of type (IIa), then
$$\sigma_1= {\rm Ind}_{B(F)}^{\GL_2(F)}(|\mbox{ }|^{\lambda}{\eta^\varepsilon}
\boxtimes|\mbox{ }|^{-\lambda}{\eta^\varepsilon}),\quad \sigma_2={\rm St}\otimes {\eta^\varepsilon}.$$
\item If $\pi$ is of type (DS), then
$$\sigma_1 = {\rm DS}(\lambda_1-\lambda_2),\quad \sigma_2 = {\rm DS}(\lambda_1+\lambda_2).$$
\item If $\pi$ is of type (PS), then 
$$\sigma_1 = {\rm Ind}_{B(\R)}^{\GL_2(\R)}(|\mbox{ }|^{(\lambda_1+\lambda_2)/2}{\rm sgn}^\varepsilon\boxtimes|\mbox{ }|^{(-\lambda_1-\lambda_2)/2}{\rm sgn}^\varepsilon),\quad \sigma_2 =  {\rm Ind}_{B(\R)}^{\GL_2(\R)}(|\mbox{ }|^{(\lambda_1-\lambda_2)/2}{\rm sgn}^\varepsilon\boxtimes |\mbox{ }|^{(-\lambda_1+\lambda_2)/2}{\rm sgn}^\varepsilon).$$
\end{itemize}
Via the isomorphism
\[H^\circ \simeq \Delta\mathbb{G}_m\backslash(\GL_2\times\GL_2),\]
we have a representation $\sigma = \sigma_1 \times \sigma_2$ of $H^\circ(F)$. Let $N^\square$ be the unipotent subgroup of $H^\circ$ defined by
\[
N^\square = \{ [{\bf n}(x),{\bf n}(y)] \mbox{ }\vert\mbox{ }x,y \in \mathbb{G}_a \},
\]
and $\psi_{N^\square}$ the additive character of $N^\square(F)$ defined by 
\[
\psi_{N^\square}([{\bf n}(x),{\bf n}(y)]) = \psi(x+y).
\]
Denote by $\mathcal{W}(\sigma,\psi_{N^\square})$ the space of smooth Whittaker functions of $\sigma$ with respect to $\psi_{N^\square}$. We have a surjective equivariant map
\[
S(V^2(F)) \otimes  \mathcal{W}(\sigma,\psi_{N^\square}) \longrightarrow \mathcal{W}(\pi,\psi_U), \quad \varphi \otimes W \longmapsto \mathcal{W}(\varphi,W),
\]
where
\begin{align}\label{E:explicit Whittaker}
\mathcal{W}(\varphi,W)(g) = \int_{\Delta N(F)\backslash H_1^{\circ}(F)}W(h_{1}h)\omega(g,h_{1}h)\varphi({\bf x}_0,{\bf y}_0)\,d\bar{h}_{1}
\end{align}
for $(g,h) \in {\rm G}(\Sp_4 \times H_1)(F)$. % (cf.~Lemmas \ref{L:local Whittaker finite 2}, \ref{L:local Whittaker archi. 1} and \ref{L:local Whittaker archi. 2}). 
Here 
\[
{\bf x}_0 = \bp 0 & -1 \\ 0 & 0\ep,\quad {\bf y}_0 = {\bf a}(-1).
\]

In Case (IIa) and Case (PS), we write $\pi=\pi_\lambda$ for the representation of $G(F)$ with parameter $\lambda$ and $\sigma=\sigma_\lambda$ for the representation $H^\circ(F)$ defined as above with respect to $\lambda$. 
We call a map
\[
\Omega \longrightarrow C^{\infty}(U(F)\backslash G(F),\psi_U),\quad \lambda \longmapsto {W}_\lambda
\]
a $K$-finite analytic family of Whittaker functions if it satisfies the following conditions:
\begin{itemize}
\item The map $(\lambda,g) \mapsto {W}_\lambda(g)$ is continuous.
\item For each $g \in G(F)$, the map $\lambda \mapsto {W}_\lambda(g)$ is analytic.
\item For each $\lambda \in \Omega$, the function $g \mapsto {W}_\lambda(g)$ belongs to $\mathcal{W}(\pi_\lambda,\psi_U)$.
\item There exist finitely many irreducible representations $\rho_i$ of $K$ such that ${W}_\lambda$ is contained in the direct sum of the $\rho_i$-isotypic subspaces of $\mathcal{W}(\pi_\lambda,\psi_U)$ for all $\lambda \in \Omega$.
\end{itemize}
Let 
\[
{\bf K}^\square = \begin{cases}
H^\circ(\o) & \mbox{ if $F$ is non-archimedean},\\
({\rm O}(2) \times {\rm O}(2))/\{\pm 1\} & \mbox{ if $F=\R$}.
\end{cases}
\]
Similarly, we define the notion of ${\bf K}^\square$-finite analytic families valued in $\mathcal{W}(\sigma_\lambda,\psi_{N^\square})$.

Let $\epsilon>0$. If $\pi$ is of type (IIa), let $\frak{X}(\pi,\epsilon)$ be the set consisting of the following characters of ${\bf T}(F)$:
\begin{align}\label{E:charaIIa}
\begin{split}
{\rm diag}(a,b,ca^{-1},cb^{-1})&\longmapsto 1,\\
{\rm diag}(a,b,ca^{-1},cb^{-1})&\longmapsto |a|^{1/2}\cdot |b|^{1/2}\cdot |c|^{-1/2},\\
{\rm diag}(a,b,ca^{-1},cb^{-1})&\longmapsto |a|^{-|{\rm Re}(\lambda)|-\epsilon}\cdot |b|^{-|{\rm Re}(\lambda)|-\epsilon}\cdot|c|^{|{\rm Re}(\lambda)|+\epsilon},\\
{\rm diag}(a,b,ca^{-1},cb^{-1})&\longmapsto |a|^{-|{\rm Re}(\lambda)|-\epsilon+1/2}\cdot |b|^{-|{\rm Re}(\lambda)|-\epsilon+1/2} \cdot |c|^{|{\rm Re}(\lambda)|+\epsilon-1/2}.
%{\rm diag}(a,b,ca^{-1},cb^{-1})&\longmapsto |a|^{\epsilon_1 \lambda+1/2}\cdot |b|^{-3\epsilon_1 \lambda-1/2} \cdot \xi^{\epsilon}(c)|c|^{\epsilon_1\lambda},\\
%{\rm diag}(a,b,ca^{-1},cb^{-1})&\longmapsto |a|^{\epsilon_1 \lambda+1/2}\cdot |b|^{-\epsilon_1 \lambda-3/2} \cdot \xi^{\epsilon}(c)|c|^{1/2},
\end{split}
\end{align} 
If $\pi$ is of type (DS), let $\frak{X}(\pi,\epsilon)$ be the set consisting of the following character of ${\bf T}(\R)$:
\begin{align}\label{E:charaDS}
{\rm diag}(a,b,ca^{-1},cb^{-1})&\longmapsto |a|^{-\epsilon}|b|^{-\epsilon}|c|^\epsilon.
\end{align}
 If $\pi$ is of type (PS), let $\frak{X}(\pi,\epsilon)$ be the set consisting of the following character of ${\bf T}(\R)$:
\begin{align}\label{E:charaPS}
{\rm diag}(a,b,ca^{-1},cb^{-1})&\longmapsto |a|^{-|{\rm Re}(\lambda_1)|-|{\rm Re}(\lambda_2)|-2\epsilon} |c|^{(|{\rm Re}(\lambda_1)|+|{\rm Re}(\lambda_2)|)/2+\epsilon}.
\end{align}

We have the following uniform asymptotic estimate for Whittaker functions and an explicit construction of $K$-finite analytic families by (\ref{E:explicit Whittaker}).
\begin{lemma}\label{L:Whittaker fun}
(1) Let ${\bf W}_\lambda$ be a ${\bf K}^\square$-finite analytic family valued in $\mathcal{W}(\sigma_\lambda,\psi_{N^\square})$ and let $\varphi \in S(V^2(F))$. Then the map 
\begin{align}\label{E:explicit analytic family}
\lambda \longmapsto \mathcal{W}(\varphi,{\bf W}_\lambda)
\end{align}
defines a $K$-finite analytic family of Whittaker functions.
 
(2) Let $W \in \mathcal{W}(\pi,\psi_U)$. For any $\epsilon>0$, there exist a non-negative function $\phi_\epsilon$ on $(F^\times)^2$ and a constant $C_\epsilon>0$ which satisfy the following conditions: 
\begin{itemize}
\item
\[
\left\vert W(utk) \right\vert \leq  \delta_{\bf B}(t)^{1/2}\phi_{\epsilon}(b^2c^{-1},ab^{-1})\sum_{\chi \in \frak{X}(\pi,\epsilon)}\chi(t)
\]
for $t={\rm diag}(a,b,ca^{-1},cb^{-1}) \in {\bf T}(F)$, $u \in U(F)$ and $k \in K$.
\item When $\pi$ is of type (IIa), we can extend $\phi_{\epsilon}$ to a Schwartz function in $\mathcal{S}(F^2)$ of the form
\[
\phi_\epsilon = C_{\epsilon}\cdot \mathbb{I}_{\varpi^{-n}\o}\otimes\mathbb{I}_{\varpi^{-n}\o}
\]
for some $n \in \Z$.
\item When $\pi$ is of type (DS) or (PS), we have
\[\phi_\epsilon(a,b) \leq C_\epsilon\cdot e^{-\pi(|a|+|b|)/4}\]
for $a,b \in \R^\times$. 
%\item When $\pi$ is of type (PS), there exists $C_{\epsilon}>0$ such that
%$$\phi_\chi(a,b) \leq C_\epsilon \cdot e^{-\pi(|a|+|b|)/4}$$
%for $a,b \in \R^\times$. Moreover, the constant $C_\epsilon$ can be chosen to be bounded uniformly as the parameter $\lambda$ varying in a compact set.
\end{itemize}

(3) Let $W_\lambda$ be a $K$-finite analytic family of Whittaker functions. Then in (2), we can choose a function
$\phi_\epsilon = \phi_{\lambda,\epsilon}$ for $W_\lambda$ so that the constant $C_\epsilon = C_{\lambda,\epsilon}$ is bounded uniformly as $\lambda$ varies in a compact set and the integer $n$ is independent of $\lambda$ in Case (IIa).
\end{lemma}

\begin{proof}
To prove (2) for Case (DS), we may assume $W = \mathcal{W}(\varphi,{\bf W})$ for some ${\bf W} \in \mathcal{W}(\sigma,\psi_{N^\square})$ and $\varphi \in S(V^2(F))$. For Case (IIa) and Case (PS), we prove (1) and (2) simultaneously. Let ${\bf W}_\lambda$ be a ${\bf K}^\square$-finite analytic family valued in $\mathcal{W}(\sigma_\lambda,\psi_{N^\square})$ and let $\varphi \in S(V^2(F))$. We write $W_\lambda = \mathcal{W}(\varphi,{\bf W}_\lambda)$.

First we consider Case (IIa). Let $t={\rm diag}(ab,a,b^{-1},1) \in {\bf T}(F)$ and $k \in K$. Choose $h_k \in H^\circ(\o)$ such that $\nu(h_k)=\nu(k)$. Then
\[
W_\lambda(tk) = \int_{H_1^\circ(\o)}Z_\lambda^{(1)}(k_1;t,k)\,dk_1 + \int_{H_1^\circ(\o)}Z^{(2)}_\lambda(k_1;t,k)\,dk_1,
\]
where
\begin{align*}
Z_\lambda^{(1)}(k_1;t,k) & = |a|\int_F\int_{F^\times}\int_{F^\times} {\bf W}_\lambda([{\bf d}(a){\bf m}(y_1),{\bf u}(x){\bf m}(y_2)]k_1h_k) \\
& \quad\quad\times\omega(tk,[{\bf d}(a){\bf m}(y_1),{\bf u}(x){\bf m}(y_2)]k_1h_k)\varphi({\bf x}_0,{\bf y}_0)|y_1|^{-2}|y_2|^{-2}\,d^\times y_1 \, d^\times y_2\,dx,\\
Z_\lambda^{(2)}(k_1;t,k) & = |a\varpi^{-1}|\int_F\int_{F^\times}\int_{F^\times}{\bf W}_\lambda([{\bf a}(\varpi),{\bf a}(\varpi)][{\bf d}(a){\bf m}(y_1),{\bf u}(x){\bf m}(y_2)]k_1h_k) \\
& \quad\quad\quad\quad\times\omega(tk,[{\bf a}(\varpi),{\bf a}(\varpi)][{\bf d}(a){\bf m}(y_1),{\bf u}(x){\bf m}(y_2)]k_1h_k)\varphi({\bf x}_0,{\bf y}_0)|y_1|^{-2}|y_2|^{-2}\,d^\times y_1 \, d^\times y_2\,dx.
\end{align*}
Let $k_1 \in H_1^\circ(\o)$. We have
\begin{align*}
Z_\lambda^{(1)}(k_1;t,k) & = |a^3b^2|\int_{F^\times}\int_{F^\times}    {\bf W}_\lambda([{\bf d}(a){\bf m}(y_1),{\bf m}(y_2)]k_1h_k) |y_1|^{-2}|y_2|^{-2}\\
&\quad\quad\quad\quad\times \int_F\psi(x)\omega(k,k_1h_k)\varphi\left( \bp 0 & -aby_1^{-1}y_2^{-1} \\ 0 & 0 \ep , \bp -ay_1^{-1}y_2 & -ay_1^{-1}y_2^{-1}x\\ 0 & y_1y_2^{-1} \ep \right) \,dx\,d^\times y_1 \, d^\times y_2.
\end{align*}
Write 
\[
\omega(k,k_1h_k)\varphi(x,y) = \sum_{i}\varphi_0^{(i)}(x)\varphi_{11}^{(i)}(y_{11})\varphi_{12}^{(i)}(y_{12})\varphi_{21}^{(i)}(y_{21})\varphi_{22}^{(i)}(y_{22})
\]
for some $\varphi_0^{(i)} \in \mathcal{S}(V(F))$, $\varphi_{11}^{(i)}, \varphi_{12}^{(i)}, \varphi_{21}^{(i)},\varphi_{22}^{(i)} \in \mathcal{S}(F)$.
Then 
\begin{align*}
Z_\lambda^{(1)}(k_1;t,k) & = |a^2b^2|\sum_i \varphi_{21}^{(i)}(0)\int_{F^\times}\int_{F^\times} {\bf W}_\lambda([{\bf d}(a){\bf m}(y_1),{\bf m}(y_2)]k_1h_k) |y_1|^{-1}|y_2|^{-1}\\
&\quad\quad\quad\quad\quad\times \varphi_0^{(i)}\left( \bp 0 & -aby_1^{-1}y_2^{-1} \\ 0 & 0 \ep \right) \varphi_{11}^{(i)}(-ay_1^{-1}y_2) \hat{\varphi}_{12}^{(i)}(-a^{-1}y_1y_2) \varphi_{22}^{(i)}(y_1y_2^{-1}) \,d^\times y_1 \, d^\times y_2,
\end{align*}
where $\hat{\varphi}_{12}^{(i)}$ is the Fourier transform of $\varphi_{12}^{(i)}$ with respect to $\psi$. Choose a constant $C>0$ and a sufficiently large integer $n$ independent of $k, k_1$, and $\lambda$ such that $\sum_i|\varphi_{21}^{(i)}(0)|  \leq C$ and 
\[
| \varphi_0^{(i)} |\leq  C\cdot \mathbb{I}_{V(\varpi^{-n}\o)},\quad |\varphi_{11}^{(i)}|\leq C\cdot \mathbb{I}_{\varpi^{-n}\o},\quad |\hat{\varphi}_{12}^{(i)}| \leq C\cdot \mathbb{I}_{\varpi^{-n}\o},\quad|\varphi_{22}^{(i)}|\leq C\cdot \mathbb{I}_{\varpi^{-n}\o}
\]
for all $i$.
Then
\begin{align*}
|Z_\lambda^{(1)}(k_1;t,k)| &\leq C^5\cdot\delta_{\bf B}(t)^{1/2} \int_{F^\times} \int_{F^\times} |{\bf W}_\lambda([{\bf d}(a){\bf m}(y_1),{\bf m}(y_2)]k_1h_k)| |a|^{1/2}|y_1|^{-1}|y_2|^{-1}\\
&\quad\quad\quad\quad\quad\quad\quad\quad\times\mathbb{I}_{\varpi^{-n}\o}(aby_1^{-1}y_2^{-1})\mathbb{I}_{\varpi^{-n}\o}(ay_1^{-1}y_2)\mathbb{I}_{\varpi^{-n}\o}(a^{-1}y_1y_2)\mathbb{I}_{\varpi^{-n}\o}(y_1y_2^{-1})\, d^\times y_1 \, d^\times y_2.
\end{align*}
Let $\epsilon>0$. There exists $\Phi_{\lambda,\epsilon} \in \mathcal{S}(F^2)$ such that
\[
|{\bf W}_\lambda([{\bf a}(t_1),{\bf a}(t_2)]k')| \leq |t_1|^{-|{\rm Re}(\lambda)|+1/2-\epsilon}|t_2|^{1/2}\Phi_{\lambda,\epsilon}(t_1,t_2)
\]
for $t_1,t_2 \in F^\times$ and $k' \in H^\circ(\o)$. Moreover, we can find such $\Phi_{\lambda,\epsilon}$ which is bounded uniformly as $\lambda$ varies in a compact set and whose support is contained in $\varpi^{-2n}\o \times \varpi^{-2n}\o$ for some integer $n$ independent of $\lambda$. Hence it suffices to consider the integral
\begin{align}\label{E:9.9}
\begin{split}
&\int_{F^\times}\int_{F^\times} \chi_{\lambda,\epsilon}(a^{-1}y_1^2) |y_2|\mathbb{I}_{\varpi^{-2n}\o}(a^{-1}y_1^2)\mathbb{I}_{\varpi^{-n}\o}(y_2)\\
&\times\mathbb{I}_{\varpi^{-n}\o}(aby_1^{-1}y_2^{-1})\mathbb{I}_{\varpi^{-n}\o}(ay_1^{-1}y_2)\mathbb{I}_{\varpi^{-n}\o}(a^{-1}y_1y_2)\mathbb{I}_{\varpi^{-n}\o}(y_1y_2^{-1})\, d^\times y_1 \, d^\times y_2,
\end{split}
\end{align}
where $\chi_{\lambda,\epsilon} = |\mbox{ }|^{-|{\rm Re}(\lambda)|-\epsilon}$. The integral is majorized by
\[
\left [C_{\lambda,\epsilon}^{(1)}  +C_{\lambda,\epsilon}^{(2)}\cdot |ab|^{1/2} +  C_{\lambda,\epsilon}^{(3)}\cdot \chi_{\lambda,\epsilon}(ab)+ C_{\lambda,\epsilon}^{(4)}\cdot \chi_{\lambda,\epsilon}(ab)|ab|^{1/2}\right] \cdot \mathbb{I}_{\varpi^{-2n}\o}(a)\mathbb{I}_{\varpi^{-2n}\o}(b)
\]
for some constants $C_{\lambda,\epsilon}^{(1)}, C_{\lambda,\epsilon}^{(2)}, C_{\lambda,\epsilon}^{(3)}, C_{\lambda,\epsilon}^{(4)}$ which are bounded uniformly as $\lambda$ varies in a compact set. Indeed, it is clear that (\ref{E:9.9}) is majorized by 
\begin{align*}
&\int_{F^\times}\chi_{\lambda,\epsilon}(a^{-1}y_1^2)\mathbb{I}_{\varpi^{-2n}\o}(a^{-1}y_1^2)\mathbb{I}_{\varpi^{-2n}\o}(a^2by_1^{-2})\,d^\times y_1 \cdot \int_{F^\times}|y_2|\mathbb{I}_{\varpi^{-2n}\o}(y_2^2)\mathbb{I}_{\varpi^{-2n}\o}(aby_2^{-2})\,d^\times y_2\\
&\times \mathbb{I}_{\varpi^{-2n}\o}(a)\mathbb{I}_{\varpi^{-2n}\o}(b).
\end{align*}
Assume that both $\log_q|a|$ an $\log_q|b|$ are even integers. The other cases can be treated in a similar way. Then the above integral is equal to 
\[
(\chi_{\lambda,\epsilon}(\varpi^2)-1)^{-1}(\chi_{\lambda,\epsilon}(ab\varpi^{2n+2})-\chi_{\lambda,\epsilon}(\varpi^{-2n}))\cdot (q^{-1}-1)^{-1}(q^{-n-1}|ab|^{1/2}-q^n)\cdot \mathbb{I}_{\varpi^{-2n}\o}(a)\mathbb{I}_{\varpi^{-2n}\o}(b).
\]
%where
%\begin{align*}
%C_{\lambda,\epsilon}^{(1)} &= \chi_{\lambda,\epsilon}(\varpi^{-2n})(1-\chi_{\lambda,\epsilon}(\varpi^2))^{-1}\left[ \chi_{\lambda,\epsilon}(\varpi^{2})(1-\chi_{\lambda,\epsilon}(\varpi^2))^{-1} - \chi_{\lambda,\epsilon}(\varpi^{-2})(1-\chi_{\lambda,\epsilon}^{-1}(\varpi^2))^{-1}\right],\\
%C_{\lambda,\epsilon}^{(2)} &=  (1-\chi_{\lambda,\epsilon}(\varpi^2))^{-1}\left[(1-\chi_{\lambda,\epsilon}^{-1}(\varpi^2))^{-1}-\chi_{\lambda,\epsilon}(\varpi^2)(1-\chi_{\lambda,\epsilon}(\varpi^2))^{-1}\right],\\
%C_{\lambda,\epsilon}^{(3)} &= (1-\chi_{\lambda,\epsilon}(\varpi^2))^{-1} \left[ \chi_{\lambda,\epsilon}(\varpi^{2n})(1-\chi_{\lambda,\epsilon}^{-1}(\varpi^2))^{-1} - \chi_{\lambda,\epsilon}(\varpi^{-2n+2})(1-\chi_{\lambda,\epsilon}(\varpi^2))^{-1}\right],\\
%C_{\lambda,\epsilon}^{(4)} &= \chi_{\lambda,\epsilon}(\varpi^{2n})(1-\chi_{\lambda,\epsilon}(\varpi^2))^{-1} \left[ \chi_{\lambda,\epsilon}(\varpi^4)(1-\chi_{\lambda,\epsilon}(\varpi^2))^{-1} - (1-\chi_{\lambda,\epsilon}^{-1}(\varpi^2))^{-1}\right].
%\end{align*}
Therefore we obtain a uniform estimate for $Z_\lambda^{(1)}(k_1;t,k)$. We have a similar estimate for $Z_\lambda^{(2)}(k_1;t,k)$. Moreover, it follows from the above estimate that the map $\lambda \mapsto W_\lambda(g)$ is analytic for each $g \in G(F)$. The continuity of the map $(\lambda,g)\mapsto W_\lambda(g)$ can be proved in a similar way. Therefore the map $\lambda \mapsto W_\lambda$ defines a $K$-finite analytic family. This completes the proof of (1) and (2) for Case (IIa).

Next we consider Case (PS). Let $t={\rm diag}(ab,a,b^{-1},1) \in {\bf T}(\R)$ and $k \in K$. Choose $h_k \in {\bf K}^\square$ such that $\nu(h_k) = \nu(k)$. Then
\[
W_{\lambda}(tk) = \int_{ {\bf K}^\square \cap H_1^\circ(\R)}Z_{\lambda}(k_1;t,k)\,dk_1,
\]
%where $K_\infty = H_1^\circ(\R) \cap ({\rm O}(2) \times {\rm O}(2))/\{\pm 1\}$ and 
where
\begin{align*}
Z_{\lambda}(k_1;t,k) &= |a|\int_\R\int_{0}^\infty\int_{0}^\infty {\bf W}_\lambda([{\bf d}(a){\bf m}(y_1),{\bf u}(x){\bf m}(y_2)]k_1h_k) \\
& \quad\quad\times\omega(tk,[{\bf d}(a){\bf m}(y_1),{\bf u}(x){\bf m}(y_2)]k_1h_k)\varphi({\bf x}_0,{\bf y}_0)y_1^{-2}y_2^{-2}\,d^\times y_1 \, d^\times y_2\,dx\\
&= \delta_{\bf B}(t)^{1/2}\int_{0}^\infty\int_{0}^\infty  {\bf W}_\lambda([{\bf d}(a){\bf m}(y_1),{\bf m}(y_2)]k_1h_k) |a|^{1/2}y_1^{-1}y_2^{-1}\\
&\quad\quad\quad\quad\times \int_\R\psi(-a^{-1}y_1y_2x)\omega(k,k_1h_k)\varphi\left( \bp 0 & -aby_1^{-1}y_2^{-1} \\ 0 & 0 \ep , \bp -ay_1^{-1}y_2 & -x\\ 0 & y_1y_2^{-1} \ep \right) \,dx\,d^\times y_1 \, d^\times y_2.
\end{align*}
The inner integral over $\R$ is equal to the product of $e^{-\pi(y_1^2y_2^{-2}+a^2y_1^{-2}y_2^2+a^2b^2y_1^{-2}y_2^{-2}+a^{-2}y_1^2y_2^2)}$ and a polynomial in four variables $y_1y_2^{-1}$, $ay_1^{-1}y_2$, $ab y_1^{-1}y_2^{-1}$, and $a^{-1}y_1y_2$. Note that this polynomial is independent of $\lambda$. We conclude that
\begin{align*}
|Z_\lambda(k_1;t,k)| &\ll_{k,k_1} \delta_{\bf B}(t)^{1/2} \int_0^\infty\int_0^\infty |{\bf W}_\lambda([{\bf d}(a){\bf m}(y_1),{\bf m}(y_2)]k_1h_k)| |a|^{1/2}y_1^{-1}y_2^{-1}\\
&\quad\quad\quad\quad\quad\quad\quad\times  e^{-\pi(y_1^2y_2^{-2}+a^2y_1^{-2}y_2^2+a^2b^2y_1^{-2}y_2^{-2}+a^{-2}y_1^2y_2^2)/2}\,d^\times y_1\,d^\times y_2.
\end{align*} 
%If $\pi$ is of type (PS), then for $\epsilon>0$ we have
For $\epsilon_1,\epsilon_2>0$, we have
\[
|{\bf W}_\lambda([{\bf a}(t_1),{\bf a}(t_2)]k')| \ll_{\lambda_1,\lambda_2,\epsilon_1,\epsilon_2} |t_1|^{-\mu_1+1/2-\epsilon_1}|t_2|^{-\mu_2+1/2-\epsilon_2}e^{-\pi(|t_1|+|t_2|)}
\]
for $t_1,t_2 \in \R^\times$ and $k' \in {\bf K}^\square$.
Here
\[
\mu_1 = \frac{|{\rm Re}(\lambda_1)|+|{\rm Re}(\lambda_2)|}{2},\quad \mu_2 = \frac{||{\rm Re}(\lambda_1)|-|{\rm Re}(\lambda_2)||}{2}.
\]
Assume $\mu_1 \geq \mu_2$. The case $\mu_2 \leq \mu_2$ can be proved in a similar way and we omit it. Let $\epsilon_1=2\epsilon_2=\epsilon>0$.
Then 
\begin{align*}
&\int_0^\infty\int_0^\infty |{\bf W}_\lambda([{\bf d}(a){\bf m}(y_1),{\bf m}(y_2)]k_1h_k)| |a|^{1/2}y_1^{-1}y_2^{-1}e^{-\pi(y_1^2y_2^{-2}+a^2y_1^{-2}y_2^2+a^2b^2y_1^{-2}y_2^{-2}+a^{-2}y_1^2y_2^2)/2}\,d^\times y_1\,d^\times y_2\\
& \ll_{\lambda_1,\lambda_2,\epsilon} |a|^{\mu_1+\epsilon} \int_{0}^{\infty}\int_{0}^{\infty} y_1^{-2\mu_1-2\epsilon}y_2^{-2\mu_2-\epsilon}e^{-\pi y_1^2(a^{-2}y_2^2+y_2^{-2})/2 -\pi y_1^{-2}(a^2y_2^2+a^2b^2y_2^{-2})/2}\,d^\times y_1\, d^\times y_2 \\
&\ll_{\lambda_1,\lambda_2,\epsilon}  |a|^{\mu_1+\epsilon} \int_{0}^{\infty} y^{-2\mu_2-\epsilon}\left(a^{-2}y^2+y^{-2}\right)^{(\mu_1+\epsilon)/2}\left(a^2y^2+a^2b^2y^{-2}\right)^{-(\mu_1+\epsilon)/2}\\
&\quad\quad\quad\quad\quad\quad\quad\quad\quad\quad\times K_{\mu_1+\epsilon}\left(\pi (a^{-2}y^2+y^{-2})^{1/2}(a^2y^2+a^2b^2y^{-2})^{1/2}\right)\,d^\times y\\
&\ll_{\lambda_1,\lambda_2,\epsilon} |a|^{\mu_1+\epsilon} \int_{0}^{\infty} y^{-2\mu_2-\epsilon}\left(a^2y^2+a^2b^2y^{-2}\right)^{-(\mu_1+\epsilon)} e^{-\pi(a^{-2}y^2+y^{-2})^{1/2}(a^2y^2+a^2b^2y^{-2})^{1/2}/2}\,d^\times y.
%& \ll_{\lambda_1,\lambda_2,\epsilon} |a|^{-\mu_1-\epsilon}|b|^{-2\mu_1-2\epsilon} \int_{0}^\infty y^{\mu_1-\mu_2+\epsilon/2} e^{-\pi(|a|+|b|+y+|ab|y^{-1})/4}\,d^\times y \\
%& \ll_{\lambda_1,\lambda_2,\epsilon} |a|^{-\mu_1-\epsilon}|b|^{-2\mu_1-2\epsilon} e^{-\pi(|a|+|b|)/4}|ab|^{\mu_1/2-\mu_2/2+\epsilon/4} K_{\mu_1-\mu_2+\epsilon/2}(2^{-1}\pi |ab|^{1/2})\\
%& \ll_{\lambda_1,\lambda_2,\epsilon} |a|^{-\mu_1-\epsilon}|b|^{-2\mu_1-2\epsilon}e^{-\pi(|a|+|b|)/4}.
\end{align*}
Here, in the last two inequalities, we use the integral representation (\ref{E:K Bessel}) of $K_{\mu_1+\epsilon}$ and the estimate 
\begin{align}\label{E:K Bessel estimate}
|K_{{\alpha_1}}(y)| \ll_{{\alpha_1,\alpha_2}} 
 y^{-{\alpha_2}}e^{-y/2}
\end{align}
for $0<{\alpha_1 \leq \alpha_2}$ and $y>0$. By the inequality
\[
\sqrt{x}+\sqrt{y} \leq \sqrt{2(x+y)}
\]
for $x,y \geq 0$, we have
\begin{align*}
&\int_{0}^{\infty} y^{-2\mu_2-\epsilon}\left(a^2y^2+a^2b^2y^{-2}\right)^{-(\mu_1+\epsilon)} e^{-\pi(a^{-2}y^2+y^{-2})^{1/2}(a^2y^2+a^2b^2y^{-2})^{1/2}/2}\,d^\times y\\
& \ll_{\lambda_1,\lambda_2,\epsilon}|a|^{-2\mu_1-2\epsilon}|b|^{-2\mu_1-2\epsilon}\int_{0}^\infty y^{\mu_1-\mu_2+\epsilon/2} e^{-\pi(|a|+|b|+y+|ab|y^{-1})/4}\,d^\times y.
\end{align*}
By the integral representation (\ref{E:K Bessel}) of $K_{\mu_1-\mu_2+\epsilon/2}$, the estimate (\ref{E:K Bessel estimate}), and the assumption that $\mu_1 \geq \mu_2$, we have
\begin{align*}
&\int_{0}^\infty y^{\mu_1-\mu_2+\epsilon/2} e^{-\pi(|a|+|b|+y+|ab|y^{-1})/4}\,d^\times y \\
& \ll_{\lambda_1,\lambda_2,\epsilon} e^{-\pi(|a|+|b|)/4}|ab|^{\mu_1/2-\mu_2/2+\epsilon/4}K_{\mu_1-\mu_2+\epsilon/2}(2^{-1}\pi|ab|^{1/2})\\
& \ll_{\lambda_1,\lambda_2,\epsilon} e^{-\pi(|a|+|b|)/4}.
\end{align*}
Moreover, we can choose constants in the above inequalities so that they are bounded uniformly as $\lambda$ varies in a compact set. Therefore we obtain a uniform estimate for $Z_\lambda(k_1;t,k)$. A similar estimate shows that the map $\lambda \mapsto W_\lambda$ is a $K$-finite analytic family. This completes the proof of (1) and (2) for Case (PS). 

Finally we assume $\pi$ is of type (DS). There exists a polynomial $P \in \C[X,Y]$ divisible by $XY$ such that
\[
|W([{\bf a}(t_1),{\bf a}(t_2)]k')| \leq P(|t_1|,|t_2|) e^{-2\pi(|t_1|+|t_2|)}
\]
for $t_1,t_2 \in \R^\times$ and $k' \in {\bf K}^\square$. Similarly as in Case (PS), we only need to consider the integral 
\[
\int_{0}^{\infty}\int_{0}^{\infty}P(|a^{-1}y_1^2|,|y_2|^2) |a|^{1/2}y_1^{-1}y_2^{-1}e^{-\pi(y_1^2y_2^{-2}+a^2y_1^{-2}y_2^2+a^2b^2y_1^{-2}y_2^{-2}+a^{-2}y_1^2y_2^2)/2}\,d^\times y_1\,d^\times y_2.
\]
We may assume $P(X,Y) = X^mY^n$ for some $m,n\in\Z_{\geq 1}$.
By an estimate similar to the one in Case (PS), we have
\begin{align*}
&\int_{0}^{\infty}\int_{0}^{\infty}P(|a^{-1}y_1^2|,|y_2|^2) |a|^{1/2}y_1^{-1}y_2^{-1}e^{-\pi(y_1^2y_2^{-2}+a^2y_1^{-2}y_2^2+a^2b^2y_1^{-2}y_2^{-2}+a^{-2}y_1^2y_2^2)/2}\,d^\times y_1\,d^\times y_2\\
&\ll_{m,n} |a|^{-m+1/2}\int_0^\infty y^{2n-1}\left(a^2y^2+a^2b^2y^{-2}\right)^{(2m-1)/4}\left(a^{-2}y^2+y^{-2}\right)^{-(2m-1)/4}\\
&\quad\quad\quad\quad\quad\quad\quad\quad\quad\times K_{m-1/2}\left(\pi(a^{-2}y^2+y^{-2})^{1/2}(a^2y^2+a^2b^2y^{-2})^{1/2}\right)\,d^\times y\\
&\ll_{m,n} |a|^{-m+1/2}\int_0^\infty y^{2n-1}\left(a^{-2}y^2+y^{-2}\right)^{-(2m-1)/2}e^{-\pi(a^{-2}y^2+y^{-2})^{1/2}(a^2y^2+a^2b^2y^{-2})^{1/2}/2}\,d^\times y\\
&\ll_{m,n}\int_0^\infty y^{n-1/2} e^{-\pi(|a|+|b|+y+|ab|y^{-1})/4}\,d^\times y\\
&\ll_{m,n} e^{-\pi(|a|+|b|)/4}|ab|^{(2n-1)/4}K_{n-1/2}(2^{-1}\pi|ab|^{1/2})\\
&\ll_{m,n,\epsilon} |ab|^{-\epsilon}e^{-\pi(|a|+|b|)/4}.
\end{align*}
Here the last inequality follows from (\ref{E:K Bessel estimate}) with ${\alpha_1} = n-1/2$ and ${\alpha_2} = n-1/2+2\epsilon$. This completes the proof of (2) for Case (DS).

It remains to prove (3). Note that for any fixed $\lambda_0 \in \Omega$, any $K$-finite analytic family of Whittaker functions can be written as a linear combination, with analytic functions of $\lambda$ as coefficients, of $K$-finite analytic families of the form (\ref{E:explicit analytic family}) in a neighborhood of $\lambda_0$. Since we only consider the convergence for $\lambda$ varying in a compact set, it suffices to consider $K$-analytic families of the form (\ref{E:explicit analytic family}). It follows from the above estimate for Case (IIa) and Case (PS) that (3) holds for $K$-analytic families of the form (\ref{E:explicit analytic family}). This completes the proof.
\end{proof}

%\begin{rmk}\label{R:Whittaker std family}
%The analytic families constructed in (2) are called standard families. For $\lambda_0 \in \mathcal{C}$, any $K$-finite analytic family can be written as a linear combination, with analytic functions of $\lambda$ as coefficients, of standard $K$-finite analytic families in a neighborhood of $\lambda_0$.
%\end{rmk}

\begin{lemma}\label{L:uniform2}
Let $\calF \in \calI(s)$ be a holomorphic section.

(1) Let $W \in \mathcal{W}(\pi,\psi_U)$. The integral $\calZ(s,W,\calF)$ is absolutely convergent for 
\begin{align*}
\begin{cases}
{\rm Re}(s) >  -1 + 4|{\rm Re}(\lambda)|  & \mbox{ if $\pi$ is of type (IIa)},\\ 
{\rm Re}(s) > -1 & \mbox{ if $\pi$ is of type (DS)},\\
{\rm Re}(s) > -1 + 2(|{\rm Re}(\lambda_1)|+|{\rm Re}(\lambda_2)|) & \mbox{ if $\pi$ is of type (PS)}.
\end{cases}
\end{align*}
In particular, the integral is absolutely convergent for ${\rm Re}(s) \geq 1$ if $\pi$ is of type (IIa) or (PS) with parameter in $\mathcal D$.

(2) Let $W_\lambda$ be a $K$-finite analytic family of Whittaker functions. The integral $\calZ(s,\mathcal{W}(\lambda),\calF)$ is uniformly convergent as $\lambda$ varies in a compact set.
\end{lemma}

\begin{proof}
We prove (1) and (2) simultaneously.
Let 
\[(K\times K)^\circ = \left\{ (k_1,k_2) \in K\times K \mbox{ }\vert\mbox{ } \nu(k_1)=\nu(k_2)\right\}.\]
We define $({\bf B}\times {\bf B})^\circ$ and $({\bf T} \times {\bf T})^\circ$ in a similar way. Let \[U' = \left\{ \left.\begin{pmatrix}
   1 & 0 & x & 0 \\
   0 & 1 & 0 & 0 \\
   0 & 0 & 1 & 0 \\
   0 & 0 & 0 & 1
 \end{pmatrix}\mbox{ } \right\vert\mbox{ }x \in {\mathbb G}_a\right\}\]
be a unipotent subgroup of $G$. For $W \in \mathcal{W}(\pi,\psi_U)$, let
\[
W'(g) = \begin{cases}
W({\rm diag}(-1,1,1,-1)g) & \mbox{ if $\pi$ is of type (IIa) or (PS)},\\
\overline{W(g)} & \mbox{ if $\pi$ is of type (DS)}.
\end{cases}
\]
We have
\begin{align*}
\calZ(s, W, \calF) &= 
 \int_{Z_{\GSp_8}(F) \tilde{U}(F) \backslash {\bf G}(F)}
 \calF(\eta g, s) (W\otimes W')(g) \, dg\\
 &=\int_{(K\times K)^\circ}\int_{Z_{\GSp_8}(F)\backslash({\bf T} \times {\bf T})^\circ(F)} \delta_{({\bf B}\times {\bf B})^\circ}(t)^{-1} (W\otimes W')(tk)\\
 &\quad\quad\quad\quad\quad\quad\quad\quad\times \int_{U'(F)\backslash U(F)} \calF(\eta(u,1)tk)\psi_U(u) \,du\,dt\,dk.
\end{align*}
Let $s \in \R$ and $\epsilon>0$. Since
\[
\left\{ \left({\rm diag}(ab,a,b^{-1},1), {\rm diag}(cd,c,c^{-1}d^{-1}a,c^{-1}a)\right) \mbox{ }\vert\mbox{ }a,b,c,d \in F^\times\right\}
\]
is a set of representatives for $Z_{\GSp_8}(F)\backslash({\bf T} \times {\bf T})^\circ(F)$, by Lemma \ref{L:Whittaker fun}, it suffices to consider the integrals
\begin{align*}
&\int_{(F^\times)^4}\delta_{({\bf B}\times {\bf B})^\circ}((t_1,t_2))^{-1/2}\phi_{\epsilon}(a,b)\phi_{\epsilon}(c^2a^{-1},d)\chi_1(t_1)\chi_2(t_2)\\
&\quad\quad\quad\quad\quad\quad\times \int_{U'(F)\backslash U(F)} \calF(\eta(u,1)(t_1,t_2)) \,du\,d(a,b,c,d)
\end{align*}
for $\chi_1,\chi_2 \in \frak{X}(\pi,\epsilon)$. Here $\phi_\epsilon$ is a function on $(F^\times)^2$ satisfying the conditions in Lemma \ref{L:Whittaker fun}-(2), $d(a,b,c,d)$ is a Haar measure on $(F^\times)^4$, and 
\[
t_1={\rm diag}(ab,a,b^{-1},1),\quad t_2={\rm diag}(cd,c,c^{-1}d^{-1}a,c^{-1}a).
\]
We may assume $\calF$ is $\GSp_8(\o)$-invariant (resp. $\GSp_8(\R) \cap {\rm O}(8)$-invariant) when $F$ is non-archimedean (resp. $F=\R$) and $\calF(1,s)=1$. By \cite[p.~173 and 177, (198) and (214)]{Jiang1996}, we have
\begin{align*}
&\int_{U'(F)\backslash U(F)} \calF(\eta(u,1)(t_1,t_2)) \,du\\
&=\delta_{({\bf B}\times {\bf B})^\circ}((t_1,t_2))^{1/2}|a^{3/2}bc^{-1}|^{s+1}\frac{\zeta(s+1)^2}{\zeta(s+2)\zeta(s+3)}f^o(u,v,s),
\end{align*}
where
\[
 u = abc^{-1}d^{-1}, \qquad
 v = ac^{-1},
\]
and
\[
f^o(u,v,s)=
\begin{cases}
\max\{1,|u|\}^{-s-1}\max\{1,|v|\}^{-2s-2} & \mbox{ if $F$ is non-archimedean},\\
(1+u^2)^{-(s+1)/2}(1+v^2)^{-s-1} & \mbox{ if $F=\R$}.
\end{cases}
\]
Fix $\chi_1,\chi_2 \in \frak{X}(\pi,\epsilon)$.  Let $k_1,k_2,k_3,k_1',k_2',k_3' \in \R$ be such that  
\[
\chi_1(t) = |a|^{k_1}|b|^{k_2}|c|^{k_3},\quad \chi_2(t) = |a|^{k_1'}|b|^{k_2'}|c|^{k_3'}
\]
for $t = {\rm diga}(a,b,ca^{-1},cb^{-1})$. Then 
\[\chi_1(t_1)\chi_2(t_2) = |a|^{k_a} |b|^{k_b} |c|^{k_c} |d|^{k_d},\]
where
\[
 k_a = k_1+k_2+k_3+k_3', \qquad
 k_b = k_1, \qquad
 k_c = k_1'+k_2', \qquad
 k_d = k_1'.
\]
Hence it suffices to consider the integral 
\[
 \int_{(F^\times)^4} \phi_{\chi_1}(a,b)\phi_{\chi_2}(c^2a^{-1},d) \cdot |a|^{{3}s/{2}  + k_a + {3}/{2}} |b|^{s + k_b + 1} |c|^{-s + k_c - 1} |d|^{k_d} f^o(u,v,s)\, d(a,b,c,d).
\]
We change variables $(a,b,c,d) \mapsto (a,b,u,v)$ to get
\begin{align}\label{E:Jiang's integral}
 \int_{(F^\times)^4} \phi_{\chi_1}(a,b)\phi_{\chi_2}(av^{-2},bu^{-1}v) \cdot |a|^{{s}/{2}  + l_a + {1}/{2}} |b|^{s + l_b + 1} |u|^{l_u} |v|^{s + l_v + 1} f^o(u,v,s) \, d(a,b,u,v), 
\end{align}
where 
\begin{align*}
 l_a & = k_a + k_c = k_1 + k_2 + k_3 + k_1' + k_2' + k_3' , \\
 l_b & = k_b + k_d = k_1 + k_1', \\
 l_u & = -k_d = -k_1', \\
 l_v & = -k_c + k_d = - k_2'. 
\end{align*}

First we consider Case (IIa). 
%Note that by (\ref{E:charaIIa}),
%\begin{align*}
%-1<l_a\leq 1,\quad 0 \leq l_b \leq 1,\quad -1<l_u<0,\quad -1 \leq l_v < 1.
%\end{align*} 
By Lemma \ref{L:Whittaker fun}-(3), we can choose a Schwartz function $\phi_\epsilon = \phi_{\lambda,\epsilon}$ for $W_\lambda$ of the form
\[
\phi_{\lambda,\epsilon}(a,b) = C_{\lambda,\epsilon} \cdot \mathbb{I}_{\varpi^{-n}\o}(a)\mathbb{I}_{\varpi^{-n}\o}(b).
\]
so that the constant $C_{\lambda,\epsilon}>0$ is bounded uniformly as $\lambda$ varies in a compact set and the integer $n$ is independent of $\lambda$.
Put $r = q^n > 1$ and $\phi = \mathbb{I}_{\varpi^{-n}\o}$.
We write the integral (\ref{E:Jiang's integral}) as 
\[
C_{\lambda,\epsilon}^2\cdot (I_1 + I_2 + I_3 + I_4), 
\]
where 
\begin{align*}
 I_1 & = \int_{F^\times} \int_{F^\times} \int_{|u| \le 1} \int_{|v| \le 1}
 \phi(a) \phi(b) \phi(a v^{-2}) \phi(b u^{-1} v) \cdot |a|^{{s}/{2}  + l_a + {1}/{2}} |b|^{s + l_b + 1} |u|^{l_u} |v|^{s + l_v + 1} \, d(a,b,u,v) \\
 & = \int_{|a| \le r} \int_{|b| \le r} \int_{|u| \le 1} \int_{|v| \le 1}
 \phi(a v^{-2}) \phi(b u^{-1} v) \cdot |a|^{{s}/{2}  + l_a + {1}/{2}} |b|^{s + l_b + 1} |u|^{l_u} |v|^{s + l_v + 1} \, d(a,b,u,v), \\
 I_2 & = \int_{F^\times} \int_{F^\times} \int_{|u| \le 1} \int_{|v| > 1}
 \phi(a) \phi(b) \phi(a v^{-2}) \phi(b u^{-1} v) \cdot |a|^{{s}/{2}  + l_a + {1}/{2}} |b|^{s + l_b + 1} |u|^{l_u} |v|^{-s + l_v - 1} \, d(a,b,u,v) \\
 & = \int_{|a| \le r} \int_{|b| \le r} \int_{|u| \le 1} \int_{|v| < 1}
 \phi(a v^2) \phi(b u^{-1} v^{-1}) \cdot |a|^{{s}/{2}  + l_a + {1}/{2}} |b|^{s + l_b + 1} |u|^{l_u} |v|^{s - l_v + 1} \, d(a,b,u,v), \\
 I_3 & = \int_{F^\times} \int_{F^\times} \int_{|u| > 1} \int_{|v| \le 1}
 \phi(a) \phi(b) \phi(a v^{-2}) \phi(b u^{-1} v) \cdot |a|^{{s}/{2}  + l_a + {1}/{2}} |b|^{s + l_b + 1} |u|^{-s + l_u - 1} |v|^{s + l_v + 1} \, d(a,b,u,v) \\
 & = \int_{|a| \le r} \int_{|b| \le r} \int_{|u| < 1} \int_{|v| \le 1}
 \phi(a v^{-2}) \phi(b u v) \cdot |a|^{{s}/{2}  + l_a + {1}/{2}} |b|^{s + l_b + 1} |u|^{s - l_u + 1} |v|^{s + l_v + 1} \, d(a,b,u,v), \\
 I_4 & = \int_{F^\times} \int_{F^\times} \int_{|u| > 1} \int_{|v| > 1}
 \phi(a) \phi(b) \phi(a v^{-2}) \phi(b u^{-1} v) \cdot |a|^{{s}/{2}  + l_a + {1}/{2}} |b|^{s + l_b + 1} |u|^{-s + l_u - 1} |v|^{-s + l_v - 1} \, d(a,b,u,v) \\
 & = \int_{|a| \le r} \int_{|b| \le r} \int_{|u| < 1} \int_{|v| < 1}
 \phi(a v^2) \phi(b u v^{-1}) \cdot |a|^{{s}/{2}  + l_a + {1}/{2}} |b|^{s + l_b + 1} |u|^{s - l_u + 1} |v|^{s - l_v + 1} \, d(a,b,u,v). \\
\end{align*}
We assume $l_u \neq 0$. The case $l_u=0$ can be proved in a similar way and we omit it.
For $|b| \leq r$ and $|v|\leq 1$, we have 
\[
 \int_{|u| \le 1} \phi(b u^{-1} v) |u|^{l_u} \, du
 = \int_{r^{-1} |bv| \le |u| \le 1} |u|^{l_u} \, du
 =\frac{1- (q^{-1} r^{-1} |bv|)^{l_u}}{1-q^{-l_u}}.
 %=\begin{cases} \displaystyle{\frac{1- (q^{-1} r^{-1} |bv|)^{l_u}}{1-q^{-l_u}}} & \mbox{ if $l_u \neq 0$},\\
 %-\log_q|b| - \log_q|v| + n + 1 & \mbox{ if $l_u=0$.}
 %\end{cases}
\]
Then
\begin{align*}
 |1-q^{-l_u}| \cdot I_1 & \le  
 \int_{|a| \le r} \int_{|b| \le r} \int_{|v| \le 1} \phi(av^{-2})
 |a|^{{s}/{2}  + l_a + {1}/{2}} |b|^{s + l_b + 1} |v|^{s + l_v + 1} \, d(a,b,v) \\
 & + q^{-l_u} r^{-l_u} \int_{|a| \le r} \int_{|b| \le r} \int_{|v| \le 1} \phi(av^{-2})
 |a|^{{s}/{2}  + l_a + {1}/{2}} |b|^{s + l_b + l_u + 1} |v|^{s + l_u + l_v + 1} \, d(a,b,v). 
\end{align*}
Similarly, we have
\[
 \int_{|u| \le 1} \phi(b u^{-1} v^{-1}) |u|^{l_u} \, du
% = \int_{r^{-1}|bv^{-1}| \le |u| \le 1} \phi(b u^{-1} v^{-1}) |u|^{l_u} \, du
 = \frac{1- (q^{-1} r^{-1} |bv^{-1}|)^{l_u}}{1-q^{-l_u}}
\]
if $r^{-1}|bv^{-1}| \le 1$ and 
\[
 \int_{|u| \le 1} \phi(b u^{-1} v^{-1}) |u|^{l_u} \, du = 0
\]
otherwise.
Hence we have
\begin{align*}
 |1-q^{-l_u}| \cdot I_2 & \le  
 \int_{|a| \le r} \int_{|b| \le r} \int_{|v| < 1}
 |a|^{{s}/{2}  + l_a + {1}/{2}} |b|^{s + l_b + 1} |v|^{s - l_v + 1} \, d(a,b,v) \\
 & + q^{-l_u} r^{-l_u} \int_{|a| \le r} \int_{|b| \le r} \int_{|v| < 1}
 |a|^{{s}/{2}  + l_a + {1}/{2}} |b|^{s + l_b + l_u + 1} |v|^{s - l_u - l_v + 1} \, d(a,b,v).
\end{align*}
We also have
\begin{align*}
 I_3 & \le \int_{|a| \le r} \int_{|b| \le r} \int_{|u| < 1} \int_{|v| \le 1} \phi(av^{-2})
 |a|^{{s}/{2}  + l_a + {1}/{2}} |b|^{s + l_b + 1} |u|^{s - l_u + 1} |v|^{s + l_v + 1} \, d(a,b,u,v), \\
 I_4 & \le \int_{|a| \le r} \int_{|b| \le r} \int_{|u| < 1} \int_{|v| < 1}
 |a|^{{s}/{2}  + l_a + {1}/{2}} |b|^{s + l_b + 1} |u|^{s - l_u + 1} |v|^{s - l_v + 1} \, d(a,b,u,v).
\end{align*}
Note that the integrals
\[
\int_{|a| \leq r} \int_{|v| \leq 1} \phi(av^{-2})|a|^{s/2+l_a+1/2}|v|^{s+l_v+1}\,d(a,v), \quad
\int_{|a| \leq r} \int_{|v| \leq 1} \phi(av^{-2})|a|^{s/2+l_a+1/2}|v|^{s+l_u+l_v+1}\,d(a,v)
\]
are absolutely convergent for 
\[
s > \max\{-2l_a-1,-l_a-l_v/2-1\},\quad s > \max\{-2l_a-1,-l_a-l_u/2-l_v/2-1\},
\]
respectively.
We conclude that the integrals $I_1$, $I_2$, $I_3$, and $I_4$ are absolutely convergent for 
\[
s > \max\left\{ -2l_a-1,-l_b-1, l_u-1, l_v-1, -l_b-l_u-1, l_u+l_v-1, -l_a-l_v/2-1, -l_a-l_u/2-l_v/2-1     \right\}.
\]
From (\ref{E:charaIIa}), one can verify that the above inequality holds if 
$
s >  4|{\rm Re}(\lambda)|+4\epsilon-1.
$
Moreover, the above integrals are uniformly convergent as $\lambda$ varies in a compact set. This completes the proof for Case (IIa).

Next we consider case (PS). Put 
\[\mu = \frac{|{\rm Re}(\lambda_1)|+|{\rm Re}(\lambda_2)|}{2}.\] By (\ref{E:charaPS}), 
$$l_a=-2\mu-2\epsilon,\quad l_b=-4\mu-4\epsilon,\quad l_u=2\mu+2\epsilon,\quad l_v=0.$$  
Assume $s > 4\mu+4\epsilon-1$. By Lemma \ref{L:Whittaker fun}-(3), we can choose a function $\phi_\epsilon = \phi_{\lambda,\epsilon}$ for $W_\lambda$ satisfying
\[
\phi_{\lambda,\epsilon}(a,b) \leq C_{\lambda,\epsilon} \cdot e^{-\pi(|a|+|b|)/4}
\]
for $a,b \in \R^\times$ so that the constant $C_{\lambda,\epsilon}>0$ is bounded uniformly as $\lambda$ varies in a compact set.
Then the integral (\ref{E:Jiang's integral}) is bounded by
\[
C_{\lambda,\epsilon}^2\int_{(\R^\times)^4}  |a|^{{s}/{2}  + {1}/{2} - 2 (\mu+\epsilon) } |b|^{s + 1 - 4 (\mu+\epsilon)}e^{-\pi(|a|+|b|)/4} \frac{|u|^{2(\mu+\epsilon)}}{(1+u^2)^{(s+1)/2}} \frac{|v|^{s + 1 }}{(1+v^2)^{s+1}} \, d(a,b,u,v),
\]
which is absolutely convergent.
Moreover, it is clear that the above integral is uniformly convergent as $\lambda$ varies in a compact set. This completes the proof for Case (PS).
%Assume $\pi$ is of type (DS). By Lemma \ref{L:Whittaker fun}, there exists $C>0$ such that
%$$\phi_\chi(a,b) \leq C\cdot e^{-\pi(|a|+|b|)/4}$$
%for $a,b \in \R^\times$. Then the integral (\ref{E:Jiang's integral}) is bounded by
%$$|I_1|+|I_2|,$$
%where
%\begin{align*}
%I_1 &= C^2\int_{\R^\times}\int_{\R^\times}\int_{|u|\geq 1}\int_{\R^\times}  |a|^{{s}/{2}  + {1}/{2} } |b|^{s + 1}e^{-\pi(|a|+|b|)/4} \frac{1}{(1+u^2)^{(s+1)/2}} \frac{|v|^{s + 1}}{(1+v^2)^{s+1}} \, d(a,b,u,v),\\
%I_2 &= C^2\int_{\R^\times}\int_{\R^\times}\int_{|u|\leq 1}\int_{\R^\times}  |a|^{{s}/{2}  + {1}/{2} } |b|^{s + 1}e^{-\pi(|a|+|b|)/4} e^{-\pi |bu^{-1}v|/4} \frac{|v|^{s + 1}}{(1+v^2)^{s+1}} \, d(a,b,u,v).
%\end{align*}
%It is clear that $I_1$ is absolutely convergent for ${\rm Re}(s)>-1$. We have
%\begin{align*}
%I_2 \leq 8\pi^{-1} C^2  \int_{\R^\times}\int_{\R^\times}\int_{\R^\times}  |a|^{{s}/{2}  + {1}/{2} } |b|^{s} e^{-\pi(|a|+|b|)/4}\frac{|v|^{s}}{(1+v^2)^{s+1}} \, d(a,b,v),
%\end{align*}
%so that $I_2$ is absolutely convergent for ${\rm Re}(s)>0$. This completes the proof.

Finally we assume $\pi$ is of type (DS). 
By (\ref{E:charaDS}), 
\[l_a=-2\epsilon,\quad l_b=-2\epsilon,\quad l_u=\epsilon,\quad l_v=\epsilon.\]
Assume $s > 4\epsilon-1$. By Lemma \ref{L:Whittaker fun}-(2), there exists a constant $C_\epsilon>0$ such that 
\[\phi_\epsilon(a,b) \leq C_\epsilon \cdot e^{-\pi(|a|+|b|)/4}\]
for $a,b \in \R^\times$.
Then the integral (\ref{E:Jiang's integral}) is bounded by
\[
C_\epsilon^2\int_{(\R^\times)^4}  |a|^{{s}/{2}  + {1}/{2} - 2 \epsilon } |b|^{s + 1 - 2 \epsilon}e^{-\pi(|a|+|b|)/4} \frac{|u|^{\epsilon}}{(1+u^2)^{(s+1)/2}} \frac{|v|^{s + 1 +\epsilon}}{(1+v^2)^{s+1}} \, d(a,b,u,v),
\]
which is absolutely convergent.
This completes the proof.
\end{proof}

\end{document}